\documentclass[12pt]{amsart}
\usepackage{a4wide,amssymb,amscd,amsmath,mathtools,stmaryrd,url}
\usepackage[initials,alphabetic]{amsrefs}
\usepackage{enumitem}
\usepackage{tikz-cd}
\usepackage{stackengine}
\usepackage{hyperref}
\usepackage[tableposition=bottom]{caption}
\stackMath

\begin{document}

\newcommand{\genG}{\mathcal{G}}                               
\newcommand{\dirac}{\delta}                                   
\newcommand{\comfam}{\mathcal{K}}                             
\newcommand{\fixco}{\K_0^S}                                   
\newcommand{\tor}{\operatorname{tor}}                         
\newcommand{\PB}{polynomially bounded}
\newcommand{\divides}{\,\big|\,}
\newcommand{\OI}{\mathcal{OI}}
\newcommand{\lcm}{\operatorname{lcm}}
\newcommand{\Hecke}{\mathcal{H}}
\newcommand{\types}{\mathcal{F}}
\newcommand{\TWN}{(TWN)}
\newcommand{\BD}{(BD)}
\newcommand{\C}{{\mathbb C}}
\newcommand{\N}{{\mathbb N}}
\newcommand{\R}{{\mathbb R}}
\newcommand{\Z}{{\mathbb Z}}
\newcommand{\Q}{{\mathbb Q}}
\newcommand{\Id}{\operatorname{Id}}
\newcommand{\A}{{\mathbb A}}
\newcommand{\AF}{{\mathcal A}}
\newcommand{\K}{\mathbf{K}}
\newcommand{\plnch}{\operatorname{pl}}
\newcommand{\bs}{\backslash}
\newcommand{\temp}{\operatorname{temp}}
\newcommand{\disc}{\operatorname{disc}}
\newcommand{\cusp}{\operatorname{cusp}}
\newcommand{\spec}{\operatorname{spec}}
\newcommand{\sprod}[2]{\left\langle#1,#2\right\rangle}
\renewcommand{\Im}{\operatorname{Im}}
\renewcommand{\Re}{\operatorname{Re}}
\newcommand{\Ind}{\operatorname{Ind}}
\newcommand{\tr}{\operatorname{tr}}
\newcommand{\Ad}{\operatorname{Ad}}
\newcommand{\Lie}{\operatorname{Lie}}
\newcommand{\Hom}{\operatorname{Hom}}
\newcommand{\Ker}{\operatorname{Ker}}
\newcommand{\vol}{\operatorname{vol}}
\newcommand{\SL}{\operatorname{SL}}
\newcommand{\GL}{\operatorname{GL}}
\newcommand{\card}[1]{\lvert#1\rvert}
\newcommand{\abs}[1]{\lvert{#1}\rvert}
\newcommand{\norm}[1]{\lVert#1\rVert}
\newcommand{\one}{\mathbf 1}
\newcommand{\aaa}{\mathfrak{a}}
\newcommand{\eps}{\epsilon}
\newcommand{\ad}{\operatorname{ad}}
\newcommand{\proj}{\operatorname{proj}}
\newcommand{\rest}{\big|}
\newcommand{\dsum}{\oplus}
\newcommand{\univ}{\mathcal{U}}
\newcommand{\modulus}{\delta}
\renewcommand{\Re}{\operatorname{Re}}
\renewcommand{\Im}{\operatorname{Im}}
\newcommand{\inorm}{\operatorname{N}}
\newcommand{\nnn}{\mathfrak{n}}
\newcommand{\fin}{{\operatorname{fin}}}
\newcommand{\level}{\operatorname{lev}}
\newcommand{\minlevel}{\operatorname{minlev}}
\newcommand{\fctr}[1]{G_{#1}}              
\newcommand{\der}{\operatorname{der}}
\newcommand{\cls}{\mathfrak{o}}
\newcommand{\truncJ}{\tilde J}
\newcommand{\SC}{\operatorname{sc}}
\newcommand{\fctrz}{\mathfrak{F}}
\newcommand{\idl}{\mathfrak{I}}
\newcommand{\intgr}{\mathcal{O}}
\newcommand{\Nm}{\operatorname{Nm}}
\newcommand{\primes}{\mathcal{P}}
\newcommand{\pr}{\operatorname{pr}}
\newcommand{\alfctrs}[1]{\widetilde{#1}}
\newcommand{\cntgr}{\Theta}
\newcommand{\CmC}{\mathcal{C}}
\newcommand{\dgp}{\operatorname{der}}
\newcommand{\PW}{\mathcal{PW}}
\newcommand{\af}{\mathfrak{a}}
\newcommand{\gf}{\mathfrak{g}}
\newcommand{\cpt}{\mathbf{K}}
\newcommand{\CmJ}{\mathcal{J}}
  \newcommand {\mf}{{\mathfrak m}}
  \newcommand {\kf}{{\mathfrak k}}
    \newcommand {\vf}{{\mathfrak v}}
  \newcommand {\tf}{{\mathfrak t}}
  \newcommand {\gl}{{\mathfrak gl}}
  \newcommand {\hf}{{\mathfrak h}}
  \newcommand {\of}{{\mathfrak o}}
  \newcommand {\nf}{{\mathfrak n}}
\newcommand {\uf}{{\mathfrak u}}
  \newcommand {\qf}{{\mathfrak q}}
  \newcommand {\pg}{{\mathfrak p}}
   \newcommand {\pf}{{\mathfrak p}}
   \newcommand {\bbf}{{\mathfrak b}}
   \newcommand {\sg}{{\mathfrak s}}
  \newcommand {\ffk}{{\mathfrak f}}
  \newcommand {\Pg}{{\mathfrak P}}
 \newcommand {\Xf}{{\mathfrak X}}
 \newcommand {\Mf}{{\mathfrak M}}
 \newcommand {\Nf}{{\mathfrak N}}
 \newcommand {\Of}{{\mathfrak O}}
 \newcommand {\Ff}{{\mathfrak F}}
 \newcommand{\ho}{{\mathfrak o}}
 \newcommand{\fS}{{\mathfrak S}}
  \newcommand {\cF}{{\mathcal F}}
  \newcommand {\Co}{{\mathcal C}}
 \newcommand {\cO}{{\mathcal O}}
 \newcommand{\cX}{{\mathcal X}}
 \newcommand {\cH}{{\mathcal H}}
 \newcommand {\cP}{{\mathcal P}}
 \newcommand {\cL}{{\mathcal L}}
 \newcommand {\cA}{{\mathcal A}}
 \newcommand {\cM}{{\mathcal M}}
\newcommand {\cT}{{\mathcal T}}
\newcommand {\cS}{{\mathcal S}}
\newcommand {\cB}{{\mathcal B}}
\newcommand {\cU}{{\mathcal U}}
\newcommand  {\cZ}{{\mathcal Z}}
\newcommand{\F}{{\mathbb F}}
\newcommand{\Tr}{\operatorname{Tr}}
\newcommand{\lev}{\operatorname{lev}}
\newcommand{\Mid}{\operatorname{id}}
\newcommand{\dq}{{/\!/}}
\newcommand{\Conv}{\operatorname{Conv}}
\newcommand {\uxi}{{\underline{\xi}}}

\newcommand{\sm}[4]{\left(\begin{smallmatrix}{#1}&{#2}\\{#3}&{#4}\end{smallmatrix}\right)}

\newcommand{\Lieg}{\mathfrak{g}}
\newcommand{\Lieh}{\mathfrak{h}}
\newcommand{\Pro}{\operatorname{Pr}}
\newcommand{\nc}{\operatorname{nc}}
\newcommand{\Res}{\operatorname{Res}}

\newtheorem{theorem}{Theorem}[section]
\newtheorem{lemma}[theorem]{Lemma}
\newtheorem{proposition}[theorem]{Proposition}
\newtheorem{remark}[theorem]{Remark}
\newtheorem{conjecture}[theorem]{Conjecture}
\newtheorem{definition}[theorem]{Definition}
\newtheorem{corollary}[theorem]{Corollary}
\newtheorem{example}[theorem]{Example}
\newtheorem{claim}[theorem]{``Claim''}
\newtheorem{assumption}[theorem]{Assumption}
\newtheorem{wassumption}[theorem]{Working Assumption}

\newcommand\fna[2]{f^{#1}_{#2}}
\newcommand\fn[3]{f^{#1}_{#2}\otimes{#3}}

\newenvironment{thmbis}[1]
  {\renewcommand{\thetheorem}{\ref{#1}$'$}%
   \addtocounter{theorem}{-1}%
   \begin{theorem}}
  {\end{theorem}}

\numberwithin{equation}{section}

\title[Asymptotics of Hecke operators]{On the asymptotics of Hecke operators for reductive groups}
\author{Tobias Finis}
\address{Universit\"at Leipzig, Mathematisches Institut, 
PF 10 09 20, D-04009 Leipzig,
Germany}
\email{finis@math.uni-leipzig.de}
\author{Jasmin Matz}
\thanks{Second author partially supported by the Israel Science Foundation (grant no.\ 1676/17)}
\address{Einstein Institute of Mathematics, The Hebrew University of Jerusalem, Jerusalem, 91904, Israel, and
Universit\"at Duisburg-Essen, Fakult\"at f\"ur Mathematik, D-45117 Essen, Germany}
\email{jasmin.matz@mail.huji.ac.il}
\date{\today}

\begin{abstract}
In this paper, we study the asymptotic behavior of the traces of Hecke operators for spherical discrete automorphic representations of fixed level on general split reductive groups over $\Q$. Under a condition on the analytic behavior of intertwining operators, which is known for the classical groups and the exceptional group $G_2$, we obtain the expected asymptotics in terms of the spherical Plancherel measure and an explicit estimate for the remainder.
\end{abstract}

\maketitle

\setcounter{tocdepth}{1}
\tableofcontents

\section{Introduction}\label{sec:intro}
In this paper, we study the asymptotic behavior of the traces of Hecke operators for spherical discrete automorphic representations of fixed level on general split reductive groups. Under a condition on the analytic behavior of intertwining operators, which is known in a large number of cases, we obtain the expected asymptotics in terms of the spherical Plancherel measure and an explicit estimate for the remainder.

We first state our main result in its simplest, most intuitive form, which is modeled after the main theorem of \cite{DKV79}. A more technical but more informative variant will appear in Section \ref{SectionSummary} below.

Let $G$ be a split reductive algebraic group defined over $\Q$ of semisimple rank $r$. Let $\Pi_{\disc} (G (\A))$ be the set of all irreducible unitary representations $\pi$ of $G (\A)$ which appear discretely in the regular representation $L^2 (A_G G (\Q) \backslash G (\A))$, and let
\[
m_{\disc} (\pi) = \dim \Hom (\pi, L^2 (A_G G (\Q) \backslash G (\A)))
\]
be the corresponding multiplicities (see \S \ref{sec:notation} for any unexplained notation).
Let $\Omega\subset i(\af_0^G)^*$ be a 
bounded domain with rectifiable boundary. Fix a maximal compact subgroup $\cpt_\infty$ of $G(\R)$ and define
\[
\Lambda_\Omega (t) = \frac{\vol(G(\Q)\backslash G(\A)^1)}{\abs{W}} 
\int_{t\Omega} \beta(\lambda) \, d\lambda,
\]
where $\beta$ denotes the spherical Plancherel measure of the group $G(\R)^1$ (cf. \S \ref{sec:plancherel}). Asymptotically the function $\Lambda_\Omega (t)$ behaves like
$C_\Omega t^d + O (t^{d-1})$ as $t\rightarrow\infty$, where $d = r + \abs{\Phi^+}$ is the dimension of the real symmetric space $G(\R)^1/\cpt_\infty$ and $C_\Omega$ is a suitable positive constant
(cf. \cite{DKV79}*{Lemma 3.11}). Fix an open compact subgroup $K$ of $G(\A_{\fin})$.
For a compactly supported, bi-$K$-invariant function $\tau:G(\A_{\fin})\longrightarrow\C$ and $\lambda \in (\af^G_0)_\C^*$ define
\begin{equation} \label{EqnMultiplicities}
m (\lambda, \tau) = \frac{1}{\abs{W \lambda}} \sum_{\substack{\pi \in \Pi_{\disc} (G (\A)), \ \pi_\infty^{\cpt_\infty} \neq 0, \\
W \lambda_{\pi_\infty}^G = W \lambda}}  m_{\disc} (\pi) \tr \pi_{\fin} (\tau).
\end{equation}
Here $\lambda^G_{\pi_\infty}$ denotes the infinitesimal character of $\pi_\infty|_{G(\R)^1}$, viewed as an element of $(\af^G_0)_\C^* / W$. 
The functions $m (\cdot, \tau)$ are $W$-invariant and supported on the discrete subset of $(\af^G_0)_\C^*$ which consists of the infinitesimal characters of all discrete automorphic representations of
$G (\A)$ containing a
$\cpt_\infty K$-fixed vector.
For any bounded subset 
$B$ of $(\af^G_0)_\C^*$ set
\[
m (B, \tau) = \sum_{\lambda \in B} m (\lambda, \tau).
\]
Similarly, we define $m_{\cusp} (\lambda, \tau)$ and 
$m_{\cusp} (B, \tau)$ by replacing
$m_{\disc} (\pi)$ by $m_{\cusp} (\pi)$, the multiplicity of $\pi$ in 
the cuspidal part of $L^2 (A_G G (\Q) \backslash G (\A))$ (which is well-known to decompose discretely).

In \cite{FiLaintertwining}*{Definition 3.3}, property (TWN+) for a reductive group $G$ was introduced, a conditional bound on the derivatives of the global intertwining operators between parabolically induced automorphic representations of $G$ (or, more precisely, of the associated scalar normalizing factors). Property (TWN+) was verified (among other cases) for groups with derived group isogenous to $\operatorname{SL} (n)$, as well as for split classical groups and the split exceptional group of type $G_2$. In \eqref{eq:twn} below we recall the precise property we need.
Assuming this property for $G$, one of our main results is the following asymptotics for the traces of 
Hecke operators with an explicit bound for the remainder term.

\begin{theorem}\label{thm:main}
Suppose that the derived group of $G$ is simple and that $G$ satisfies property (TWN+). 
For all bi-$K$-invariant compactly supported functions $\tau:G(\A_f)\longrightarrow\C$ and all $t\ge1$ we have 
 \begin{equation}\label{eq:main}
m (t \Omega, \tau) - \Lambda_\Omega (t) \sum_{\gamma \in
Z(\Q)} \tau (\gamma)
\ll_{\Omega}  \|\tau\|_{1} t^{d-1},
 \end{equation}
if the root system of $G$ is not of type $A_1$ or $A_2$. In the latter cases,
the remainder term is $\|\tau\|_{1} t^{d-1/2} \log (1+t)$ and $\|\tau\|_{1} t^{d-1} (\log (1+t))^2$, 
respectively.
\end{theorem}

See Remark \ref{RemarkErrorTermA12} and Remark \ref{RemarkSmoothedVariant} below for possible improvements and variants, and Section \ref{SectionSummary} for local asymptotics and related results. For possible applications of Theorem \ref{thm:main} to the distribution of 
low-lying zeros of $L$-functions in families we refer to \cites{ShTe,SaShTe}.

While the above asymptotics concerns the spherical automorphic representations that are tempered at the archimedean place, we can also bound the part of the spectrum that is non-tempered at infinity.

\begin{theorem} \label{thm:complementary}
Suppose that $G$ satisfies property (TWN+). For
all bi-$K$-invariant compactly supported functions $\tau$ on $G(\A_{\fin})$
and all $t\ge 1$
we have 
 \[
\abs{m (\{ \lambda \in (\af^G_0)_\C^* \smallsetminus i (\af^G_0)^* : \| \lambda \| \le t \}, \tau)} 
   \le  \sum_{\substack{\lambda \in (\af^G_0)_\C^* \smallsetminus i (\af^G_0)^* : \\ \| \lambda \| \le t}} 
   \abs{m (\lambda, \tau)}
  \ll 
  \|\tau\|_{1} t^{d-2}.
 \]
\end{theorem}

It will become clear from the proofs that Theorem 
\ref{thm:complementary} lies less deep than Theorem \ref{thm:main}, since Theorem 
\ref{thm:complementary} uses only basic upper bounds for the entire geometric side of the trace formula, while for Theorem \ref{thm:main} one needs to control the non-central contribution much more precisely.

\begin{remark} 
By a result of Wallach \cite{Wa84}, any non-cuspidal discrete automorphic representation $\pi$ of $G$ is necessarily non-tempered at $\infty$, i.e. it satisfies $\lambda^G_{\pi_\infty} \notin i (\af^G_0)^*$. This immediately implies that Theorem \ref{thm:main} and Theorem
\ref{thm:complementary} remain valid if we 
replace $m (B, \tau)$ by $m_{\cusp} (B, \tau)$.
\end{remark}

If we take $\tau$ to be the characteristic function of $K$, then Theorem \ref{thm:main} and Theorem
\ref{thm:complementary} together give the Weyl law for the discrete (or cuspidal) spectrum of $G$ together with an upper bound for the remainder term.
For the groups $G=\GL(n)$ the Weyl law was first proven in \cite{MR2276771}, and the Weyl law with a remainder term of $O(t^{d-1} (\log t)^{\max\{3,n\}})$ was obtained in \cite{LaMu09} (the latter for sufficiently small open compact subgroups $K$).
Without error bound, the Weyl law for the \emph{cuspidal} spectrum of an arbitrary adjoint group was established in \cite{LiVe07}. 
In the $\operatorname{GL}(n)$ case, a variant of Theorem \ref{thm:main} was established previously by Nicolas Templier and the second-named author \cite{MaTe}, following the work of the second-named author on 
$\operatorname{GL}(n)$ over imaginary quadratic fields \cite{weyl} (however in both cases with an estimate for the remainder term that is weaker in its dependence on the Hecke operator $\tau$). 

As in previous works, the present paper is based on Arthur's trace formula.
The new feature here is the treatment of its geometric side.  
Following \cites{FiLa11,FiLa16}, we do not split the geometric side according to geometric conjugacy or a finer equivalence relation.
Instead we use reduction theory, the Bruhat decomposition and the theory of intertwining operators for the principal series. The core case is to understand the contributions of the Bruhat cells not contained in a proper standard parabolic subgroup. In \cite{FiLa11}, these contributions had been estimated by the $L^1$-norms of sufficiently many derivatives of the test function. By a more refined analysis, we can (up to logarithmic correction factors) reduce to the $L^1$-norm of derivatives in the direction of the maximal unipotent subgroup of order up to its dimension. The contribution of the non-central elements of the Borel subgroup is treated by a more direct approach, which reduces it to lattice point sums over certain ''small'' unipotent radicals of maximal parabolic subgroups. The resulting bound for a non-negative test function is already close to being optimal in some cases (namely for groups with root system of type $A_n$ or $C_n$).
In addition, we prove estimates for spherical functions that are more refined than the basic estimates of \cites{BlPo16,MaTe}. While they are not essential to our approach, they allow us to sharpen the final result.

The spectral side is treated by a straightforward generalization of the argument of \cite{LaMu09}. We are able to appeal to \cite{FiLaintertwining} (which is based on the power of Arthur's functoriality results, as far as the classical groups are concerned, and on Shahidi's work on the symmetric cube $L$-function of $\GL(2)$ for the group $G_2$) to deal with a large number of cases. 
In \cite{FL19}, Erez Lapid and the first-named author use the results of the present paper on the geometric side to 
derive quantitative bounds on the non-cuspidal contribution to the spectral side, and to establish a weaker version of Theorem \ref{thm:main} for the \emph{cuspidal} spectrum (with a remainder term of $O(\norm{\tau \nu_G}_1 t^{d-\delta})$ for some $\delta > 0$ and a logarithmic weight function $\nu_G$ depending on $G$).
This result,
which can be regarded as a quantitative version of
\cite{LiVe07},
is valid for any simple Chevalley group $G$, and its proof does not use any information on automorphic $L$-functions. The Weyl law for the full discrete spectrum remains unresolved at this point if we do not assume property (TWN+). Unconditionally, we know only M\"{u}ller's polynomial upper bound \cite{MR1025165}.

Let us give a quick overview of the structure of the paper. In Section \ref{SectionTestFunctions}, we set up the necessary notation, review the spherical Plancherel measure and the Paley-Wiener theorem for spherical functions (at the real place), and define the test functions that will be put into the trace formula. We also prove some simple auxiliary results. In Section \ref{SectionSummary}, we introduce Arthur's trace formula, summarize our geometric and spectral results in some detail, and reduce the geometric statements to the fundamental estimate of Theorem \ref{TheoremMainEstimateNonnegative} and Corollary \ref{CorollaryMainEstimateNonnegative}. The proof of these results is the object of Section \ref{MainGeometric}, which forms the heart of the paper. Section \ref{SectionSpectral} deals with the spectral side and contains the proofs of Theorem \ref{thm:main} and Theorem \ref{thm:complementary}.
Finally, Appendix \ref{AppendixSpherical} establishes the improved estimates for spherical functions (for group elements contained in a compact set and arbitrary spectral parameters) that are used in the estimation of the geometric side in Section \ref{SectionSummary}, and Appendix
\ref{AppendixInequalityClassical} contains the proof of an elementary inequality that 
simplifies the statement of our results in the case of classical groups.

\section{Test functions and auxiliary results} \label{SectionTestFunctions} 

\subsection{Notation}\label{sec:notation}
Let $G$ be a split reductive algebraic group over $\Q$ of semisimple rank $r$. We fix a maximal split torus $T_0$ and a minimal parabolic subgroup $P_0=T_0U_0$, where $U_0$ denotes the unipotent radical of $P_0$. We call a parabolic subgroup $P\subset G$ standard if $P_0\subset P$, and semistandard if $T_0\subset P$. Let $\cL$ denote the set of all Levi components containing $T_0$ of semistandard parabolic subgroups. If $M\in \cL$, let $\cL(M)=\{L\in\cL\mid M\subset L\}$, and let $\cP(M)$ be the set of all parabolic subgroups with Levi component $M$. We write $r_M$ for the semisimple rank of $M\in\cL$.

Let $Z_G$ be the center of $G$, and $A_G$ the identity component of $Z_G(\R)$. Let $G(\R)^1$ (resp.\ $G(\A)^1$) denote the intersection of the kernels of the absolute values of all $\Q$-characters of $G$ in $G(\R)$ (resp.\ $G(\A)$). Then $G(\R)\simeq G(\R)^1\times A_G$ and $G(\A)\simeq G(\A)^1\times A_G$. If $G$ is clear from the context, we might also write $Z=Z_G$.

Let $\af_0$ denote the Lie algebra of $T_0(\R)$ and let
$A_0 = A_{T_0}$ be the identity 
component of $T_0(\R)$.
If $M\in\cL$, let $\af_0^M$ denote the Lie algebra of $T_0(\R)\cap M(\R)^1$, and $\af_M$ the Lie algebra of $A_M$. Then $\af_0\simeq \af_0^M\oplus \af_M$ for any $M\in\cL$, and $\dim\af_0^G$ is equal to the semisimple rank $r$ of $G$. More generally we define $\af_M^L$ for every $M\in\cL$ and $L\in \cL(M)$.
We denote by $(\af_M^L)^*$ the dual of $\af_M^L$, and by $(\af_M^L)^*_\C$ the complexification of the dual. 
As in \cite{DKV79}*{\S 3.7}, we use the Killing form to give $\af_0^G$ and its dual the structure of Euclidean vector spaces.

Let $\Phi\subset (\af_0^G)^*$ be the set of roots of $T_0$ on the Lie algebra $\gf$ of $G$. For a root $\alpha \in \Phi$ denote by $\mathfrak{u}_{\alpha}$
the corresponding (one-dimensional) root space in $\gf$.
Let $\Phi^+\subset \Phi$ be the subset of positive roots corresponding to our choice of minimal parabolic subgroup $P_0$, and let $\Delta\subset \Phi^+$ be the set of simple roots. If $M\in \cL$, we write $\Phi^M\subset \Phi$ and $\Phi^{M, +}\subset \Phi^+$ for the respective subsets consisting of the roots on the Lie algebra of $M$. If $P=MU$ is a standard parabolic subgroup, set $\Delta^M=\Delta\cap \Phi^{M,+}$ and write $\Phi_U^+=\Phi^+\smallsetminus \Phi^{M, +}$, which is the set of roots of $T_0$ on the Lie algebra of $U$. We denote by $\rho\in (\af_0^G)^*$ the half-sum of all positive roots, and by $W = N_{G(\Q)} (T_0) / T_0$ the Weyl group of $(G,T_0)$. Similarly, $\rho^M$ denotes the half-sum of all roots in $\Phi^{M,+}$ and $W^M$ the Weyl group of $T_0$ in $M$. Furthermore, $\Phi_M$ denotes the set of roots of $T_M = T_0 \cap Z_M$ on $\gf$, and $\Phi_P^+$ the roots of $T_M$ on $U$.

If $P=MU$ is a standard parabolic subgroup, we write $\delta_P$ for the modulus function of $M(\A)$ on $U(\A)$, and we set $\delta_0=\delta_{P_0}$. If $L\in\cL(M)$, we write $\delta_{P\cap L}= \delta_{P\cap L}^L$ for the modulus function of $M(\A)$ on $U(\A)\cap L(\A)$.

We fix an integral model for $G$ and regard it as a group scheme over $\Z$. For convenience, we make the following additional assumptions (cf. \cite{MR0197467}). We assume that $T_0$ with its integral structure is isomorphic to $\mathbb{G}_m^{r'}$ over $\Z$, and that the integral structure of $G$ is induced by the integral structure of $T_0$ and by a Chevalley system $(X_{\alpha})_{\alpha \in \Phi}$ for $\gf$ \cite{MR0453824}*{Ch. VIII, \S 2, Definition 3}. Setting $\mathfrak{u}_{\alpha,\Z} = \Z X_{\alpha}$ defines a lattice in $\mathfrak{u}_{\alpha}$ for every root $\alpha \in \Phi$. Let $\mathfrak{g}_{\Z}$ be the Lie algebra of the fixed integral model of $G$. It is 
a lattice in $\gf$ that is a Chevalley order with respect to the Chevalley system $(X_{\alpha})$ [ibid., Ch. VIII, \S 12]. In particular, the intersection of $\gf_{\Z}$ with $\mathfrak{u}_{\alpha}$ is $\mathfrak{u}_{\alpha,\Z}$. For any subspace 
$\mathfrak{x} \subset \gf$ we write
$\mathfrak{x}_{\Z} = \mathfrak{x} \cap \gf_{\Z}$, a lattice in $\mathfrak{x}$.
For each prime $p$, $\cpt_p = G(\Z_p)$ is a hyperspecial maximal compact subgroup of $G(\Q_p)$ corresponding to a hyperspecial point in the apartment defined by $T_0$. 
The group $\cpt_p$ satisfies $G(\Q_p)=P_0(\Q_p)\cpt_p$ and stabilizes the 
$\Z_p$-lattice
$\mathfrak{g}_{\Z_p} = \mathfrak{g}_{\Z} \otimes \Z_p$ under the adjoint representation.

Fix a maximal compact subgroup $\cpt_\infty$ of $G(\R)$ with $G(\R)=P_0(\R)\cpt_\infty$. Let $d = r + \abs{\Phi^+}$ be the dimension of the symmetric space $G(\R)^1/\cpt_\infty$.  Let $\cpt$ denote the maximal compact subgroup $\prod_{v\le  \infty} \cpt_v$ of $G(\A)$. $K$ will denote an arbitrary but fixed open compact subgroup of $G (\A_{\fin})$.

We will use the notation $f \ll g$ throughout the paper to indicate that there exists a constant $A>0$ such that $\abs{f}\le A g$. This constant $A$ is allowed to depend on $G$ and all choices of objects made in \S\ref{sec:notation}. If $A$ depends on further variables, say $x, y, \ldots$, we write $f\ll_{x,y,\ldots} g$.

\subsection{Plancherel measure}\label{sec:plancherel}
We will need estimates for the spherical Plancherel density $\beta$ of $G(\R)$. 
We follow the convention of \cite{DKV79}*{(3.29)} for the Haar measure on $G(\R)^1$, and in particular fix a Haar measure on the maximal unipotent subgroup $U_0(\R)$ as in [ibid., p. 37].
The function $\beta$ is supported on the tempered subspace $i(\af_0^G)^*$ of $(\af_0^G)_\C^*$. It can be computed in terms of Harish--Chandra's $c$-function, namely $\beta(\mu)=|c(\mu)|^{-2}$ for $\mu\in i(\af_0^G)^*$.
By \cite{DKV79}*{\S 3}, we have 
\[
 c(\mu)=\prod_{\alpha\in\Phi^{+}} \frac{I_0 (\langle\mu,\alpha^\vee\rangle)}{I_0 (\langle\rho,\alpha^\vee\rangle)}
\quad
\text{for}
\quad
I_0 (s)=\frac{\Gamma(\frac{s}{2})}{\Gamma(\frac{s+1}{2})}.
\]
By standard properties of the Gamma function we obtain
\[
\beta_0 (it) := |I_0 (it)|^{-2} = \frac{t}{2} \operatorname{tanh} (\frac{\pi t}2) \le \frac{\abs{t}}{2}. 
\]
Therefore we have the explicit formula
\[
\beta (\mu)=\prod_{\alpha\in\Phi^{+}} \left[ \frac{\Gamma (\frac{\langle\rho,\alpha^\vee\rangle}{2})}{\Gamma (\frac{\langle\rho,\alpha^\vee\rangle+1}{2})} \right]^2
\,
\prod_{\alpha\in\Phi^{+}} \beta_0 (\langle\mu,\alpha^\vee\rangle),
\quad \mu \in i(\af_0^G)^*.
\]
Similarly as in \cite{LaMu09} we define 
\[
 \tilde\beta(t,\mu)
= \prod_{\alpha\in\Phi^{+}} (t+ |\langle \mu, \alpha^\vee \rangle|)
\]
and set 
$\tilde\beta(\mu) = \tilde\beta(1,\mu)$.
With this notation we have
\begin{equation}\label{eq:upper:bound:plancherel}
 \beta(\mu+t\nu)
 \ll \tilde\beta(t,\mu) \tilde\beta(\nu)
\end{equation}
for all $\mu,\nu\in i(\af_0^G)^*$ and $t\ge1$.
For $M\in\cL$, we define $\beta^M$, $\tilde\beta^M(\mu)$, and $\tilde\beta^M(t,\mu)$ similarly with respect to $M$ instead of $G$.

\subsection{Estimates for spherical functions}\label{sec:estimatesspherical}
We will need estimates for the elementary spherical functions $\phi_\lambda$, $\lambda\in (\af_0^G)^*_\C$. These are defined by
\[
 \phi_\lambda(g)=\int_{\cpt_\infty} e^{\langle\lambda + \rho, H_0(kg)\rangle} \, dk, 
 \quad g\in G(\R),
\]
where $H_0: G(\R)\longrightarrow \af_0^G$ denotes the Iwasawa projection. The elementary spherical function satisfy the following basic properties:
\begin{itemize}
\item $\phi_\lambda$ is bi-$\cpt_\infty$-invariant and invariant under $A_G$,
\item $\phi_\lambda(k)=1$ for every $\lambda\in(\af_0^G)^*_\C$, $k\in\cpt_\infty$, 
\item $|\phi_\lambda|\le \phi_{\Re\lambda}$ for all $\lambda\in (\af_0^G)^*_\C$,
\item in particular, $|\phi_\lambda|\le \phi_0 \le 1$ for $\lambda\in i(\af_0^G)^*$,
\item $\phi_\mu=\phi_\lambda$ if and only if the Weyl group orbits $W\mu$ and $W\lambda$ coincide.
\end{itemize}

By the Cartan decomposition, any element $g \in G(\R)$ can be written as $g = k_1 a k_2$ with $a \in A_0$ and $k_1$, $k_2 \in \cpt_\infty$, and $a$ is unique up to $W$-conjugation. Let $X(g) \in \af_0^G / W$ be the projection to $\af_0^G$ of $H_0 (a)$.

\begin{proposition}[\cites{MaTe,BlPo16}]\label{PropositionSphericalSimple}
Assume that the derived group $G^{\dgp}$ of $G$ is simple. Let $\CmC\subset G(\R)^1$ be a compact set and $A >0$. Then 
we have the estimate
\[
\phi_\lambda(g)
 \ll_{\CmC, A} \frac{1}{(1+\|\Im \lambda\| \|X(g)\|)^{1/2}}
\]
for all $g\in \CmC$ and $\lambda\in(\af_0^G)_\C^*$ with $\|\Re\lambda\|\le A$.
\end{proposition}

While this estimate is in fact best possible in some situations, it can be improved if we take the position of the vectors $\lambda$ and $X(g)$ into account. We will however
let $g$ enter only through its distance to the identity.
Set
\begin{equation} \label{EquationTildeDLambda}
\tilde{D} ( \lambda) = \min_{M \in \cL, \, M \neq G}
\prod_{\alpha \in \Phi^+ \smallsetminus \Phi^{M,+}} (1 + \abs{\sprod{\lambda}{\alpha^\vee}})^{\frac12}.
\end{equation} 
In Appendix \ref{AppendixSpherical}, we will prove the 
following estimate, which like Proposition
\ref{PropositionSphericalSimple} is valid for $g$ in a compact set and arbitrary $\lambda$. It in fact optimal if we let $g$ enter only through
$\|X(g)\|$. 
Compared to the similar estimate of
\cite{DKV}*{Theorem 11.1}, the new feature here is that we may let $g$ approach the identity, which is crucial for our later application
(cf. the proof of
Proposition \ref{PropositionTruncatedGeometric} below).

\begin{proposition}[cf.\ Corollary \ref{prop:est:sph:fct:levis} below]\label{PropositionSpherical}
Let $\CmC\subset G(\R)^1$ be a compact set. Then 
\[
\phi_\lambda(g)
 \ll_{\CmC} \tilde{D} (\|X(g)\| \lambda )^{-1}
\]
for all $g\in \CmC$ and $\lambda\in i (\af_0^G)^*$.
\end{proposition}

\subsection{Spherical functions and the Paley-Wiener theorem}
We first recall the Paley-Wiener theorem for spherical functions on the group $G(\R)$. Let $C^\infty_c (G(\R) \dq A_G \cpt_\infty)$ be the
space of all functions in $C^\infty_c (G(\R))$ that are 
bi-$\cpt_\infty$-invariant and invariant under $A_G$.
Let $\PW ((\af^G_0)_\C^*)$ be the space of Paley-Wiener functions on
the complex vector space $(\af^G_0)_\C^*$, and $\PW ((\af^G_0)_\C^*)^W$ the subspace of $W$-invariants.
The spherical Fourier transform 
\[
\cH: C^\infty_c (G(\R) \dq A_G \cpt_\infty) \to
\PW ((\af^G_0)_\C^*)^W
\] 
defined by
\[
(\cH f) (\lambda) = \int_{G(\R)/A_G}  
f (g) \phi_{\lambda} (g) d g
\]
can be expressed as the composition of the Abel transform 
\[
\cA: C^\infty_c (G(\R) \dq  A_G \cpt_\infty) \to
C^\infty_c (\af^G_0)^W, 
\]
which is an isomorphism of topological vector spaces, with the usual Fourier transform $C^\infty_c (\af^G_0) \to \PW ((\af^G_0)_\C^*)$. 
The inverse $\cB$ of the Abel transform is given by
\begin{equation} \label{EqnInverseAbelTransform} 
(\cB h) (g) = \frac{1}{|W|} \int_{i (\af^G_0)^*}  
\hat{h} (\lambda) \phi_{-\lambda} (g) \beta (\lambda) d \lambda,
\end{equation}
where $\beta (\lambda)= | c(\lambda)|^{-2}$ denotes the spherical Plancherel density as in  \S\ref{sec:plancherel}. 
It is convenient to extend this definition to general (not necessarily $W$-invariant) functions $h \in C^\infty_c (\af^G_0)$. Defined this way, the map $\cB: C^\infty_c (\af^G_0) \to
C^\infty_c (G(\R) \dq A_G \cpt_\infty)$ obviously factors through the projection to the $W$-invariants.

For a compactly supported function $h$, a real number $t \ge 1$ and a vector
$\mu \in (\af^G_0)_{\C}^*$, 
let 
\[
h_{t,\mu} (X) = t^r h (t X) e^{-\sprod{\mu}{X}}.
\]
The Fourier transform of this function is given by 
\[
\hat{h}_{t,\mu} (\lambda) =  \hat{h} (t^{-1} (\lambda - \mu)). 
\]
In the following we will be concerned with the test functions
\[
\fna{t,\mu}{h}
= \cB (h_{t,\mu}) \in C^\infty_c (G(\R) \dq A_G\cpt_\infty) 
\]
at infinity.
For $t=1$ and $\mu=0$ we simply write $f_h = \cB (h)$.

From the definition of the map $\cB$ we get the trivial bound
\[
(\cB h) (g) \le \frac{1}{|W|} 
\int_{i (\af^G_0)^*}  
\abs{\hat{h} (\lambda)} \beta (\lambda) 
d \lambda, \quad g \in G (\R),
\]
which implies that
\begin{equation}\label{eq:test:fct:triv:est}
\fna{t,\mu}{h} (g) = (\cB h_{t,\mu}) (g) \ll_{h} t^r \tilde\beta (t, \mu), \quad g \in G (\R).
\end{equation}

\subsection{Simple test functions and their derivatives}
By \cite{PlaRa}*{\S 3.2}, there exists an embedding $\iota:G(\R) \hookrightarrow \GL(n,\R)$
with the following properties. 
\begin{itemize}
 \item $\iota(G(\R))$ is self--adjoint, that is, if $x\in \iota(G(\R))$, then its transpose $x^{\tr}$ is also contained in $\iota(G(\R))$.
 \item $\iota^{-1} (\operatorname{O} (n)) = \cpt_\infty$, the fixed maximal compact subgroup of $G(\R)$.
\item Let $D$ denote the torus of all diagonal matrices in $\GL(n,\R)$. Then the connected component of $\iota^{-1} (D)$ equals $A_0$.
 \item Let $V_0$ denote the group of all unipotent upper triangular matrices in $\GL(n,\R)$. Then $\iota^{-1} (V_0)=U_0(\R)$ 
\end{itemize}
We fix such an embedding $\iota$ once and for all.

 Let $g^1\in G(\R)^1$ denote the image of $g\in A_G\backslash G(\R)$ under the natural isomorphism $A_G\backslash G(\R)\simeq G(\R)^1$. 
Let 
\[
Q: A_G \K_\infty \backslash G (\R) \longrightarrow S, \quad
g \mapsto \iota (g^1)^{\tr} \iota (g^1)
\]
where $S$ denotes the set of positive definite symmetric $n\times n$-matrices, 
and 
\[
q (g) = \tr Q (g) - n.
\]
The entries $Q_{ij}$ of the matrix $Q$, as well as the function $q$, are regular functions of the variable $g^1$ defined over $\R$. In particular, they are smooth functions on $G(\R)$. The functions $Q_{ij}$ are left-invariant under translation by $A_G \cpt_\infty$, and the function $q$ factors through the projection to the double coset space $G(\R) \dq A_G\cpt_\infty$.
Since we have $\det Q (g) = 1$, the function $q$ is everywhere non-negative and vanishes precisely on $\cpt_\infty$.

An elementary estimate connects $q(g)$ and $X(g)$ (defined in \S \ref{sec:estimatesspherical}): for any compact subset $\CmC$ of $G( \R) / A_G$ there exist constants $c_2 \ge c_1 > 0$ with
\begin{equation}\label{EqnXandq}
c_1 \norm{X(g)}^2 \le q(g) \le c_2 \norm{X(g)}^2, \quad
g \in \CmC.
 \end{equation}

For a function $\phi: [0, \infty [ \to \C$ and $\epsilon >0$ let
\begin{equation}\label{eq:test:fct:phi}
 F^\phi(g) = \phi (q(g))
\end{equation}
and
\begin{equation}\label{eq:test:fct:eps}
 F^{\phi,\epsilon} (g) = \phi(\epsilon^{-1}q(g)) .
\end{equation}
Let $\cU(\uf_0)$ denote the universal enveloping algebra of $\uf_0\otimes\C$.

\begin{lemma} \label{PropDerivativeF}
For any $X \in \cU (\uf_0)$, homogeneous of degree $k$, there exist smooth functions $\Delta_{kj}$, $j = 1, \ldots, k$, on $G(\R)$, depending on $X$, such that
\begin{equation} \label{EqnDerivativeF} 
X * F^\phi (g) = \sum_{j=1}^k \Delta_{kj} (g) F^{\phi^{(j)}} (g)
\end{equation} 
for all functions $\phi \in C^k ([0,\infty[)$.
Moreover, $\Delta_{kj}$ is a homogeneous polynomial of degree $j$ in the entries of $Q (g)$, and its degree with respect to the diagonal entries is at most $k-j$.
\end{lemma}

\begin{proof}
For 
$Y \in \uf_0 \otimes \R$ we have 
\[
Y * Q (g) = \left. \frac{d}{dt} \right|_{t=0} Q (g \exp (tY))
= \left. \frac{d}{dt} \right|_{t=0} \iota (\exp (tY))^{\tr}  Q (g)  \iota (\exp (tY)).
\]
By the properties of the embedding $\iota$, we know that the function $\iota (\exp (tY))$ takes values in the upper triangular unipotent matrices. Its derivative at $t=0$ is the upper triangular nilpotent matrix 
$d \iota (Y)$. It follows that 
\[
Y * Q (g) =  (d\iota) (Y)^{\tr}  Q (g)  + Q (g) (d\iota) (Y),
\]
and the derivative $Y * Q_{ij}$, $1 \le i, j \le n$, is therefore a linear combination (depending on $Y$) of the functions $Q_{kj}$, $k < i$, and $Q_{ki}$, $k < j$. In particular, $Y * q (g) = 2 \tr (Q (g) (d\iota) (Y))$ is a linear combination of the non-diagonal entries of $Q$, which implies the case $k=1$ of the proposition.

The general case follows easily by induction on $k$.
\end{proof}

\begin{lemma} \label{LemmaDerivativesTestFunctions}
Let $R > 0$ be fixed, and let $C_R^k ([0,\infty[)$ be the space of all functions in $C^k ([0,\infty[)$ with support contained in $[0,R]$. 
For any $X \in \cU (\uf_0)$, homogeneous of degree $k$, 
and any $\phi \in C_R^k ([0,\infty[)$, we have
\[
X * F^\phi (g)  \ll_R  \sum_{j=1}^k q (g)^{\frac12 \max (2j-k, 0)} F^{\phi^{(j)}} (g).
\]
In particular, for any $\phi \in C^k_R ([0,\infty[)$ 
there exists
$\psi \in C^k_R ([0,\infty[)$ with
\[
\abs{X * F^{\phi,\epsilon} (g)}  \le   \epsilon^{-k/2} F^{\psi, \epsilon} (g).
\]
\end{lemma}

\begin{proof}
We use Lemma \ref{PropDerivativeF}.
By assumption, $F^\phi$ is supported on the set of all $g$ with
$q(g) \le R$. Since for any such $g$ the positive definite matrix $Q(g)$ has trace at most $R+n$, all its entries are also bounded by $R+n$. 

Let $\lambda_1, \ldots, \lambda_n > 0$ be the eigenvalues of $Q(g)$. Then $\sum_i \lambda_i = q (g)+n$ and $\prod_i \lambda_i = 1$. Using the arithmetic-geometric mean inequality, we obtain 
\[
\lambda_i + (n-1) \lambda_i^{-\frac{1}{n-1}} \le q (g) + n, 
\quad i = 1, \ldots, n.
\]
By an application of Taylor's formula for the function $(1+x)^{-\frac{1}{n-1}}$, we easily derive from this the bounds
\[
\abs{\lambda_i-1} \le  \left( \frac{2 (n-1)}{n} \right)^{1/2} (R+n)^{1 + \frac{1}{2(n-1)}} q (g)^{1/2}, 
\quad i = 1, \ldots, n.
\]
Therefore, any non-diagonal entry of $Q(g)$ is bounded by 
$n \max_i \abs{\lambda_i-1} \ll q (g)^{1/2}$ (while the diagonal entries are just bounded in terms of $R$).

These bounds on the entries of $Q(g)$ imply that $\Delta_{kj} (g) 
\ll q (g)^{\frac12 \max (2j-k, 0)}$ for all $g$ in the support of $F^\phi$. Inserting this estimate into \eqref{EqnDerivativeF} gives the assertion.
\end{proof}

 \subsection{$p$-adic spherical functions}
As in \cite{Macdonald} and \cite{Ta83}*{(1.6)}, we define for any prime $p$ the zonal spherical function $\phi_{p,\lambda}$ on $G(\Q_p)$ by 
 \[
  \phi_{p,\lambda}(g)
  =\int_{\cpt_p} e^{\langle\lambda+\rho, H_{0,p}(kg) \rangle}\, dk,
  \quad \lambda\in(\af_0^G)^*_\C,
 \]
 where $H_{0,p}:G(\Q_p)\longrightarrow \af_0$ denotes the $p$-adic Iwasawa projection characterized by $e^{\langle \lambda, H_{0,p}(t)\rangle}=\abs{\lambda(t)}_p$ for any $t\in T_0(\Q_p)$ and any character $\lambda$ of $T_0$ defined over $\Q$ (which we identify with an element 
 of $\af_0^*$).
By definition, the function $\phi_{p,\lambda}$ is bi-invariant under $\cpt_p$. 
The following properties of $\phi_{p,\lambda}$ follow from its definition:
\begin{itemize}
 \item $\phi_{p,\lambda}=\phi_{p,\mu}$ if and only if $\lambda$ and $\mu$ are in the same Weyl group orbit.
 \item If $\lambda\in(\af_0^G)^*$, then $\phi_{p,\lambda}$ is a positive real-valued function.
 \item If $\lambda\in (\af_0^G)^*_\C$, then $|\phi_{p,\lambda}|\le \phi_{p,\Re\lambda}$.
\end{itemize}

If $\lambda\in(\af_0^G)^*$, then  $\Conv(W \lambda)$ denotes the closed convex hull of the Weyl group orbit $W\lambda$ of $\lambda$  in $(\af_0^G)^*$. 

\begin{lemma}\label{lem:spherical:fct:1}
If $\lambda\in(\af_0^G)_\C^*$ with $\Re\lambda\in \Conv(W\rho)$, then 
\[
 A(f,\lambda)=\int_{G(\Q_p)} f(g)\phi_{p,\lambda}(g)\, dg
\]
satisfies
$|A(f,\lambda)|\le \norm{f}_1 = A(|f|, \rho)$ for any $f\in L^1(G(\Q_p)\dq \cpt_p)$, and 
$|\phi_{p,\lambda}|\le 1$. 
\end{lemma}

\begin{proof}
Because of $|\phi_{p,\lambda}|\le \phi_{p,\Re\lambda}$ it suffices to consider the case $\lambda=\Re\lambda\in \Conv(W\rho)$. 
Let $f\in L^1(G(\Q_p)\dq\cpt_p)$ be non-negative. By the definition of $\phi_{p,\lambda}$ we obtain
\[
A(f,\lambda)
 = \frac{1}{|W|}\sum_{w\in W} \int_{G(\Q_p)} f(g) \phi_{p,w\lambda}(g)\, dg
 = \frac{1}{|W|} \int_{G(\Q_p)} f(g) e^{\langle \rho, H_{0,p}(g)\rangle} \sum_{w\in W} e^{\langle w\lambda, H_{0,p}(g)\rangle} \, dg.
\]
The function $(\af_0^G)^*\ni\mu\mapsto\sum_{w\in W} e^{\langle w\mu, X\rangle}$ is convex for any $X\in\af_0^G$, so that 
\[
\sum_{w\in W} e^{\langle w\lambda, H_{0,p}(g)\rangle}
\le \sum_{w\in W} e^{\langle w\rho, H_{0,p}(g)\rangle},
 \]
since by assumption $\lambda\in \Conv(W\rho)$.
Hence 
\[
A(f,\lambda)
  \le \frac{1}{|W|} \int_{G(\Q_p)} f(g) e^{\langle \rho, H_{0,p}(g)\rangle} \sum_{w\in W} e^{\langle w\rho, H_{0,p}(g)\rangle} \, dg  
  =  A(f,\rho).
\]
If $\lambda=\rho$, then $-\rho\in W \rho$ so that $\phi_{p,\rho}=\phi_{p,-\rho}\equiv 1$, 
which implies the first assertion. The second assertion follows directly from the first.
\end{proof}

\subsection{Adelic test functions}\label{sec:adelic:function}

With above notation, we now define our adelic test functions 
\begin{equation}\label{eq:def:test:fct}
\fn{t,\mu}{h}{\tau}= \fna{t,\mu}{h}\tau |_{G (\A)^1} \in C^\infty_c (G (\mathbb{A})^1 \dq \K_\infty K), 
\end{equation}
where $\tau$ is a compactly supported function on the discrete set $G (\mathbb{A}_\fin) \dq K$.
Here, the archimedean factor $\fna{t,\mu}{h}$ is supported on the set of
all $g \in G (\R) \dq A_G$ with $\norm{X(g)} \le c t^{-1}$, where $c$ is a constant depending only on $h$, and in particular on a compact subset of  
$G (\R) \dq A_G$ independent of $\mu$ and $t \ge 1$.

\section{Summary of the main results} \label{SectionSummary} 

\subsection{Local asymptotics}
We can now formulate more precise versions of our main results. We consider the trace of the test functions $\fn{}{h}{\tau}$ on the discrete spectrum,
where $h \in C^\infty_c (\af^G_0)$ is supported in a fixed compact set,
and $\tau$ is a compactly supported, right-$K$-invariant function on
$G(\A_{\fin})$.
Note first that by the definition of $m (\lambda, \tau)$ in \eqref{EqnMultiplicities} we have
\[
J_{\disc} (\fn{}{h}{\tau}) = \tr R_{\disc} (\fn{}{h}{\tau}) = \sum_{\lambda} m (\lambda, \tau)
\hat{h} (\lambda)
\]
for all $h \in C^\infty_c (\af^G_0)$. Note that $h$ is not necessarily $W$-invariant.

Recall that we introduced $\tilde{D}(\lambda)$ in \eqref{EquationTildeDLambda}. 
Assume that $G^{\dgp}$ is simple, set $D ( \lambda) = \tilde{D} ( \lambda)$
for classical $G$, and
\begin{equation}
D ( \lambda) = \frac{1}{\log (2 + \norm{\lambda})} \min_{\substack{M \in \cL, \\ M \neq G}} 
\max_{\substack{S \subset \Phi^+ \smallsetminus \Phi^{M,+}, \\ \abs{S} = 2r}}
\prod_{\alpha \in S} (1 + \abs{\sprod{\lambda}{\alpha^\vee}})^{\frac12}
\end{equation}
for exceptional $G$.
We remark that in all cases $D (\lambda) \gg (1 + \norm{\lambda})^{\frac12}$, and even $D (\lambda) \gg 1 + \norm{\lambda}$ if the root system of $G$ is not of type $A_n$.

For the test functions $\fn{t,\mu}{h}{\tau}$ with fixed $h$ our main result (local asymptotics) is:

\begin{theorem} \label{TheoremLocal} 
Assume that $G^{\dgp}$ is simple and that $G$ satisfies property (TWN+). 
Then for all $t \ge 1$ and all $\mu \in i (\af^G_0)^*$ we have
\begin{multline*}
J_{\disc} (\fn{t,\mu}{h}{\tau}) - \frac{\vol (G (\Q) \backslash G (\A)^1)}{\abs{W}}
\sum_{\gamma \in Z (\Q)} \tau (\gamma) 
\int_{i (\af^G_0)^*}  \hat{h}_{t,\mu} (\lambda)\beta (\lambda) \, d \lambda
\\ \ll_h
\norm{\tau}_1 
\log (1 +  t + \norm{\mu})^r
\tilde \beta (t, \mu) 
D (t^{-1} \mu)^{-1}.
\end{multline*}
\end{theorem}

\begin{remark}
Using only the simpler estimate for spherical functions of Proposition \ref{PropositionSphericalSimple}, one obtains a remainder term of 
\[
\norm{\tau}_1 
\log (1 +  t + \norm{\mu})^r
\tilde \beta (t, \mu) 
(1 + t^{-1} \norm{\mu})^{-\frac12}.
\]
\end{remark}

Note that the integral 
$\int_{i (\af^G_0)^*} \hat{h}_{t,\mu} (\lambda)\beta (\lambda) \, d \lambda$ is bounded from above and below by a constant multiple of $t^r \tilde{\beta} (t, \mu)$ if $h(0)\neq 0$ (cf. \cite{DKV79}*{Proposition 6.10}).

The deduction of Theorem \ref{thm:main} from this result is standard and will be given in \S \ref{SectionAsymptotics} below. The proof of Theorem \ref{thm:complementary} will be given in \S \ref{SectionNontempered}.
We can also show estimates for the general spherical test functions $\fn{}{h}{\tau}$ uniformly in $h$ (Theorem \ref{TheoremGeometricArbitrary}
and Proposition \ref{prop:spectral:smooth:copy}). 

The proof rests on Arthur's trace formula. Arthur's trace formula provides two alternative expressions for a certain distribution $J$ on $G (\A)^1$ which
depends on $T_0$ and $\K$, both fixed as specified in \S \ref{sec:notation}.
One expression, the geometric side, is given in terms of integrals of sums over the lattice $G (\Q)$, and can be split into the contributions of the geometric conjugacy classes (and according to finer equivalence relations). On the other hand, the spectral side involves the discrete automorphic spectra of all Levi subgroups $M$ of $G$, in particular of $G$ itself.

For a test function $f \in C^\infty_c (G (\A)^1)$, Arthur defines $J(f)$ to be the
value at the point $T=\mathcal{T}_0$ specified in \cite{Ar81}*{Lemma 1.1}
of a certain polynomial $J^T (f)$ on $T\in\af_0^G$ of degree at most $r$.
The polynomial $J^T(f)$ depends on the additional choice of a parabolic subgroup $P_0$ of $G$ with Levi part $T_0$. It is
uniquely defined by the property that the difference between 
$J^T (f)$ and the truncated integral
\[
j^T (f) = 
\int_{G (\Q)\backslash G(\A)^1} F (g,T) \sum_{\gamma\in G(\Q)} f (g^{-1}\gamma g)\, dg, \quad T \in \af_0^G,
\]
becomes exponentially small as $d(T) = \min_{\alpha\in\Delta}\alpha(T) \to \infty$ (cf. \cite{MR828844}, \cite{FiLa16}*{Theorem 5.1}). Here
$F (\cdot,T)$ is the characteristic function of the truncation of 
the quotient $G (\Q)\backslash G(\A)^1$ at height $T$ as in \cite{MR828844}*{p. 1242}.
We assume that $d(T) > c_G$, where $c_G$ is a reduction-theoretic constant determined by $G$ only (cf. \cite{FiLa16}*{\S 2.1}). We call such parameters $T$ sufficiently regular.

We will now summarize our results on the geometric and spectral sides of the trace formula. Taken together, they imply our asymptotic results on the discrete spectrum.

\subsection{Estimates for the geometric side of the trace formula}
We denote 
by 
\[
J_{\operatorname{nc}} (f) = J (f) - \vol(G(\Q)\backslash G(\A)^1) \sum_{\gamma \in Z (\Q)} f (\gamma)
\]
the non-central part of the geometric side of the trace formula. By
\begin{equation} \label{EquationTruncatedNc}
j_{\operatorname{nc}}^T (f) = 
\int_{G (\Q)\backslash G(\A)^1} F (g,T) \sum_{\gamma\in G(\Q) \smallsetminus Z (\Q)} f (g^{-1}\gamma g)\, dg, \quad T \in \af_0^G,
\end{equation}
we denote the integral of the non-central part of the kernel over the truncated fundamental domain with sufficiently regular truncation parameter $T$. As above,
$J_{\operatorname{nc}} (f)$ is the value at $T=T_0$ of the polynomial
$J^T_{\operatorname{nc}} (f) = J^T (f) - \vol(G(\Q)\backslash G(\A)^1) \sum_{\gamma \in Z (\Q)} f (\gamma)$, and the difference between
$J^T_{\operatorname{nc}} (f)$ and
$j_{\operatorname{nc}}^T (f)$ becomes exponentially small as $d(T) \to \infty$. By \cite{FiLa16}*{Theorem 5.1, Theorem 7.1}, we can make this statement more precise. We have the explicit estimate 
\begin{equation} \label{EquationApproximation}
J^T_{\operatorname{nc}} (f) - j_{\operatorname{nc}}^T (f)
\ll_K (1 + \norm{T})^r e^{-d(T)} \sum_i \norm{X_i * f}_1,
\end{equation}
where $(X_i)$ denotes a basis of the space of differential operators in $\mathcal{U} (\operatorname{Lie} G (\R)^1)$ of degree up to $\dim G (\R)^1$, and
$K$ is an open compact subgroup of $G (\A_{\fin})$ such that $f$ is right-$K$-invariant. In particular, for the test functions
$\fn{}{h}{\tau}$, where $h \in C^\infty_c (\af^G_0)$ is supported in a compact set $\CmC$ and
$\tau$ is right-$K$-invariant for an open compact subgroup $K$, we have
\begin{equation} \label{EquationApproximation2}
J_{\operatorname{nc}}^{T} (\fn{}{h}{\tau})
- j_{\operatorname{nc}}^{T} (\fn{}{h}{\tau}) \ll_{\CmC, K}
\norm{\tau}_1 (1 + \norm{T})^r e^{-d(T)} 
 \int_{i (\af^G_0)^*}  
\abs{\hat{h} (\lambda)} \beta (\lambda) 
(1 + \norm{\lambda})^N
\, d \lambda
\end{equation}
for some fixed integer $N$ (depending only on $G$).

The following estimate is an easy consequence of Arthur's fine geometric expansion. We record it here since it will be used later to bound the terms on the spectral side.
\begin{lemma} \label{LemmaSimpleGeometricBound}
For all functions $h \in C^\infty_c (\af^G_0)$ with support contained in a compact set $\CmC$ we have
\begin{equation}\label{eq:simple1}
J (\fn{}{h}{\tau}) \ll_{\CmC, \tau}
\int_{i (\af^G_0)^*}  
\abs{\hat{h} (\lambda)} \beta (\lambda) 
d \lambda.
\end{equation}
In particular, for all $t \ge 1$ and $\mu \in i (\af^G_0)^*$ we have
\begin{equation}\label{eq:geom:simple}
J (\fn{t,\mu}{h}{\tau}) \ll_{h, \tau}
t^r \tilde\beta (t, \mu).
\end{equation}
\end{lemma}

\begin{proof}
For fixed $\tau$, the functions $\fn{}{h}{\tau}$ for $h$ as above are  supported in a fixed compact subset $\Omega$ of $G(\A)^1$.
By the basic compactness property of Arthur's fine geometric expansion
\cite{MR835041}*{Lemma 9.1}, $J (\fn{}{h}{\tau})$ can be expressed as a finite linear combination (depending only on $\Omega$ and $K$) of weighted orbital integrals [ibid., Corollary 9.2], which are given by absolutely continuous complex-valued measures on the corresponding orbits. We can therefore estimate $J (\fn{}{h}{\tau})$ for given $\CmC$ and $\tau$ by a constant multiple of the maximum of $\abs{f_h (g)}$ on $G(\R)$, which is clearly bounded from above by
$\abs{W}^{-1} \int_{i (\af^G_0)^*}  
\abs{\hat{h} (\lambda)} \beta (\lambda) 
d \lambda$.
\end{proof}

\begin{remark}
We could make the paper more self-contained (i.e., avoid any appeal to
\cite{MR835041}) by generalizing
Theorem \ref{TheoremGeometricArbitrary} to the case where the derived
group of $G$ is not necessarily simple. Estimating the central contribution trivially, by this method
one even gets the upper bound 
\[
J (\fn{}{h}{\tau})
\ll_{\CmC} \norm{\tau}_1
\int_{i (\af^G_0)^*}  
\abs{\hat{h} (\lambda)} \beta (\lambda) 
d \lambda
\]
instead of \eqref{eq:simple1}. Note that even if we are only interested in the
case where $G^{\dgp}$ is simple, we still need the upper bound of Lemma 
\ref{LemmaSimpleGeometricBound}
for all Levi subgroups $M$ of $G$ to bound the spectral side.
\end{remark}

For the non-central contribution we will improve the bound \eqref{eq:simple1} (it is evidently best possible for the central part if $\hat{h}$ is non-negative), and make the dependence on $\tau$ explicit.
The basis of our estimates for the geometric side is the following 
estimate for non-negative test functions which will be proven in Section \ref{MainGeometric}:

\begin{theorem}[cf. Corollary \ref{CorollaryMainEstimateNonnegative} below]\label{Thm:MainEstimateNonnegative} 
Fix $c > 0$. For any non-negative monotonically decreasing function $\phi$ on $[0,c]$, extended by zero to $\R^{\ge 0}$, and any 
compactly supported, right-$K$-invariant function $\tau$ on
$G(\A_{\fin})$,
we have the estimate
\[
j_{\operatorname{nc}}^T (F^{\phi} \otimes \tau) 
\ll_c  \norm{\tau}_1 \int_0^c\phi (x)
x^{\frac{r}{2} - 1}  \left( \abs{\log x}^{r - 1} + 
(1+\norm{T})^{r} \right) \, dx.
\]
\end{theorem}

As a consequence we obtain the following estimates for the test functions
$\fn{}{h}{\tau}$. For simplicity, we assume that $G^{\dgp}$ is simple.

\begin{proposition} \label{PropositionTruncatedGeometric}
Assume that $G^{\dgp}$ is simple.
For all functions $h \in C^\infty_c (\af^G_0)$ with support in a ball of radius $c t^{-1}$, $t \ge 1$, and sufficiently regular $T \in \af_0^G$ we have the estimate
\begin{multline*}
j_{\operatorname{nc}}^T (\fn{}{h}{\tau}) \ll
\norm{\tau}_1 t^{-r}
\int_{i (\af^G_0)^*}  
\abs{\hat{h} (\lambda)} \beta (\lambda) D (t^{-1} \lambda)^{-1} \\
\left[ (\log t)^{r-1} 
+ (1 + \norm{T})^r 
+ \nu_G (\log (2 + t^{-1} \norm{\lambda}))^{r-1} \right]
\, d \lambda,
\end{multline*}
where $\nu_G = 0$ for classical and $\nu_G = 1$ for exceptional groups $G$.
\end{proposition}

\begin{theorem} \label{TheoremGeometricArbitrary}
Assume that $G^{\dgp}$ is simple.
For all functions $h \in C^\infty_c (\af^G_0)$ with support in a ball of radius $c t^{-1}$, $t \ge 1$, we have the estimate
\[
J_{\operatorname{nc}} (\fn{}{h}{\tau}) \ll
\norm{\tau}_1 t^{-r}
\int_{i (\af^G_0)^*}  
\abs{\hat{h} (\lambda)} \beta (\lambda) D (t^{-1} \lambda)^{-1} 
(\log (1 + t + \norm{\lambda}))^{r}
\,
d \lambda.
\]
Slightly more generally, we have the estimate
\[
J_{\operatorname{nc}}^T (\fn{}{h}{\tau}) \ll
(1 + \norm{T})^r 
\norm{\tau}_1 t^{-r}
\int_{i (\af^G_0)^*}  
\abs{\hat{h} (\lambda)} \beta (\lambda) D (t^{-1} \lambda)^{-1} 
(\log (1 + t + \norm{\lambda}))^{r}
\,
d \lambda
\]
for all $T \in \af_0$.
\end{theorem}

Considering the test functions $\fn{t,\mu}{h}{\tau}$ for fixed $h$, the following local bound follows easily from Theorem \ref{TheoremGeometricArbitrary},
using \eqref{eq:upper:bound:plancherel} and the rapid decrease of $\hat{h}$.

\begin{corollary} \label{TheoremGeometricFtmu}
If $G^{\dgp}$ is simple, we have for $t \ge 1$ and all $\mu \in i (\af^G_0)^*$:
\[
J_{\operatorname{nc}} (\fn{t,\mu}{h}{\tau}) \ll_h
\norm{\tau}_1 
\log (1 +  t + \norm{\mu})^r
\tilde \beta (t, \mu) 
D (t^{-1} \mu)^{-1}.
\]
\end{corollary}

The proofs of Proposition \ref{PropositionTruncatedGeometric}
and Theorem \ref{TheoremGeometricArbitrary} will be given in \S
\ref{SubsectionProofsGeometric}.

\subsection{Estimates for the spectral side of the trace formula}
We now turn to the spectral side of the trace formula. The spectral expansion has the form
\[
J (f)
 = \sum_{[M]} J_{\spec,M} (f) =
 \sum_{[M]} \sum_{s\in W(M)} J_{\spec,M,s}(f),
\]
where $[M]$ ranges over conjugacy classes of Levi subgroups, represented by members of $\mathcal{L}$, and $s$ over the elements of $W(M) = N (M) / M$, represented by members of $W$. Here $N(M)$ is the normalizer of $M$ in $G$. The term with $M=G$ is simply the trace of $f$ on the discrete spectrum of $G$. For the other terms we have the following bounds, conditional under a bound on the logarithmic derivatives of normalizing factors, called property (TWN+) in \cite{FiLaintertwining}. 

Recall that for $f \in C^\infty_c (G (\A))$ and any parabolic subgroup $P=MU$ of $G$ containing $T_0$, the Harish-Chandra descent 
$f^{(P)} \in C^\infty_c (M (\A))$ of $f$ is defined by
\begin{equation}
f^{(P)} (m) =  \delta_P (m)^{1/2} \int_{U(\A)} f_{\K} (m u) \, du,
\quad m \in M (\A),
\end{equation}
where $f_{\K}(g)=\int_{\K} f(k^{-1} g k)\, dk$.
For a compactly supported continuous function $\tau$ 
on $G(\A_{\fin})$ we define $\tau_{\K_{\fin}}$ and $\tau^{(P)}$ analogously.
We also define the following seminorm of such a function $\tau$:
\begin{equation}
\norm{\tau}_1^{(P)} =
\int_{\K_{\fin}} \int_{M (\A_{\fin})} 
\delta_P (m)^{1/2} 
\left| \int_{U (\A_{\fin})} \tau_{\K_{\fin}} (muk) \, du \right|
\, dm dk.
\end{equation}
For bi-$\K_{\fin}$-invariant functions $\tau$, the seminorm $\norm{\tau}_1^{(P)}$ evidently coincides with the 
$L^1$-norm of $\tau^{(P)}$.

For $M, L\in\cL$, $M\subset L$, let $(\af_M^L)^\perp= (\af_0^M)^*\oplus (\af_L^G)^*$. 
If $\nu\in i(\af_M^L)^\perp$, we write $\nu^M$ for the projection of $\nu$ onto $i (\af_0^M)^*$ along $i(\af_L^G)^*$. For $M \in\cL$ and $w\in W$
let $wM \in\cL$ be the conjugate of $M$ by $w$.

\begin{proposition}[cf.\ Corollary \ref{prop:spectral} below] \label{PropositionSpectralFtmu}
Assume $G$ satisfies property (TWN+).
Let $M \in\cL$, $s\in W(M)$ and $P \in\mathcal{P} (M)$.
For any compactly supported function $\tau: G (\A_\fin) / K \to \C$, any $t \ge 1$ and all $\mu \in i (\af^G_0)^*$ we have
\[
J_{\spec,M,s}(\fn{t,\mu}{h}{\tau})
\ll_h  \|\tau\|_{1}^{(P)} t^{r_G + r_M - r_{L_s}} 
\sum_{w \in W} \frac{(\log (1 + t + \norm{\mu}))^{r_G-r_{L_s}}}{(1 + t^{-1} \norm{\mu^{wL_s}_{wM}})^{N}}
\tilde{\beta}^{wM} (t, \mu^{wM}).
\]
\end{proposition}

A more general estimate for the spectral terms can be formulated as follows.
If $\lambda\in (\af_0^G)_\C^*$, let 
$\mathcal{M}_M (\hat{h}) (\lambda)$ 
denote the maximum of $|\widehat{h}|$ on the ball of radius $2+\|\rho^M\|$ in $(\af_0^G)_\C^*$ around $\lambda$.

\begin{proposition}[cf.\ Proposition \ref{prop:spectral:smooth} below] \label{prop:spectral:smooth:copy} 
Assume $G$ satisfies property (TWN+). Let $M \in\cL$, $s\in W(M)$ and $P \in\mathcal{P} (M)$.
For any $h\in C_c^\infty(\af_0^G)^W$, any compactly supported function $\tau: G (\A_\fin) / K \to \C$ we have
 \begin{equation}
J_{\spec,M,s}(\fn{}{h}{\tau})
\ll  \|\tau\|_{1}^{(P)}
\int_{i(\af_M^L)^\perp} 
\mathcal{M}_M (\hat{h}) (\nu)  \; \tilde \beta^M (\lambda^M) \left(\log\left(2+\|\lambda\|\right)\right)^{r_G - r_{L_s}} \, d \lambda.
\end{equation}
\end{proposition}

The proofs of these spectral bounds will be given in Section \ref{SectionSpectral}.
Theorem \ref{TheoremLocal} follows now from the combination of our estimates for the geometric side (Corollary \ref{TheoremGeometricFtmu}) and for the spectral side (Proposition \ref{PropositionSpectralFtmu}).

\subsection{Proof of the geometric estimates} \label{SubsectionProofsGeometric} 
We now deduce Proposition \ref{PropositionTruncatedGeometric} and
Theorem \ref{TheoremGeometricArbitrary}
from Theorem \ref{Thm:MainEstimateNonnegative}.

\begin{proof}[Proof of Proposition \ref{PropositionTruncatedGeometric}]
By the estimate for the spherical functions 
$\phi_{-\lambda} (g)$ of Proposition \ref{PropositionSpherical} we have
\[
f_h (g) = (\mathcal{B} h) (g) \ll
\int_{i (\af^G_0)^*}  
\abs{\hat{h} (\lambda)} \beta (\lambda) \tilde{D} (\norm{X(g)} \lambda)^{-1} \, d \lambda.
\]
Note that the right-hand side is a non-negative monotonically decreasing function of
$\norm{X(g)}$. By assumption, $h$ is supported in a ball of radius $c t^{-1}$ around the origin. Taking \eqref{EqnXandq} into account, we can therefore invoke 
Theorem \ref{Thm:MainEstimateNonnegative} 
to estimate
\begin{multline*}
j_{\operatorname{nc}}^T (\fn{}{h}{\tau}) \ll
\norm{\tau}_1 \\
\int_{i (\af^G_0)^*}  
\abs{\hat{h} (\lambda)} \beta (\lambda) 
\left[ \int_{0}^{c t^{-1}}  
\tilde{D} (x \lambda)^{-1} x^{r-1}
\left( \abs{\log x}^{r-1} 
+ (1 + \norm{T})^r 
\right) \, dx 
\right]
\, d \lambda.
\end{multline*}
After an obvious substitution, the inner integral becomes
\[
t^{-r} \int_{0}^{c}  
\tilde{D} (x t^{-1} \lambda)^{-1} x^{r-1}
\left( (\abs{\log x} + \log t)^{r-1} 
+ (1 + \norm{T})^r 
\right) \, dx. 
\]
The function $\tilde{D}$ is defined as a minimum over all $M \in \mathcal{L}$, $M \neq G$, and it is enough to take the minimum over all maximal $M$. For
each maximal $M$, we can trivially estimate
\[
\tilde{D} (x t^{-1} \lambda)^{-1} \ll x^{-\frac{\abs{S'}}{2}}
\prod_{\alpha \in S'} \abs{\sprod{t^{-1} \lambda}{\alpha^\vee}}^{-\frac12},
\]
where $S'$ is the set of all $\alpha \in \Phi^+ \smallsetminus \Phi^{M,+}$ with
$t^{-1} \abs{\sprod{\lambda}{\alpha^\vee}} \ge 1$.

In the case of classical groups, by Proposition \ref{PropositionMinimum} (applied to the
dual root system $\Phi^\vee$)
we can restrict further to those $M$ with the maximum number $\abs{\Phi^{M,+}}$ of positive roots. For these $M$ there are $R = \min_{P \neq G} \dim U_P$ many factors in the product. 
We have therefore $\abs{S'} \le R < 2r$ (using that $G$ is classical, cf. Appendix \ref{AppendixInequalityClassical} and Table \ref{tab:parabolics}), and
the inner integral can be estimated by a constant multiple
of 
\[
t^{-r} \tilde{D} (t^{-1} \lambda)^{-1} 
\left( (\log t)^{r-1} 
+ (1 + \norm{T})^r \right),
\]
which finishes the proof for classical $G$.

For exceptional groups (where $R > 2 r$), we can estimate the inner integral in the same way as long as $\abs{S'} < 2r$. If  $\abs{S'} \ge 2r$, then
we consider subsets $S \subset S'$ with $2r$ elements. In this case the inner integral can be estimated by
\[
t^{-r} \Delta^{-\frac12} (\log (1+\Delta))
\left( (\log t)^{r-1} 
+ (1 + \norm{T})^r + (\log (1+\Delta))^{r-1} \right),
\]
where $\Delta = \prod_{\alpha \in S} \abs{\sprod{t^{-1} \lambda}{\alpha^\vee}}$. This proves the proposition in the exceptional case.
\end{proof}

\begin{proof}[Proof of Theorem \ref{TheoremGeometricArbitrary}]
We use a variant of the classical Littlewood-Paley decomposition. Let $g$ be a 
fixed radial Paley-Wiener function on $(\af_0^G)^*_{\C}$ such that $g(0)=1$ and $g(\lambda)$ is monotonically decreasing in $\norm{\lambda}$ for $\lambda \in  i (\af_0^G)^*$. (Using \cite{MR3019774}*{Theorem 1.1}, the existence of such a function 
follows easily from the existence of radial Paley-Wiener functions in dimension $r+2$ that are non-negative for imaginary parameters.)
Set
\[
\hat{\psi}_0 (\lambda) = g (t^{-1}\lambda),
\]
and
\[
\hat{\psi}_j (\lambda) = g (t^{-1} 2^{-j} \lambda)
- g (t^{-1} 2^{1-j} \lambda), \quad j \ge 1.
\]
These Paley-Wiener functions are the Fourier transforms of functions 
on $\af_0^G$ supported in a ball of radius $\le c t^{-1}$, they are
non-negative on $i (\af_0^G)^*$ and satisfy
\[
\sum_{j=0}^{\infty} \hat{\psi}_j (\lambda) = 1.
\]
We note the following fact, which follows easily from the rapid decrease of $g$ and the
vanishing of $\hat{\psi}_j$, $j \ge 1$, at the origin:
\[
\sum_{j=0}^{\infty} (1 + j + \log t)^r e^{- C (1+ j + \log t)} \hat{\psi}_j (\lambda) 
\ll_{C} (t + \norm{\lambda})^{-C} (\log (1 + t + \norm{\lambda}))^{r}
\]
for any $C \ge 0$. 

Fix in addition to $C>0$, which we will specify later, a constant $\eta >0$.
We apply Proposition \ref{PropositionTruncatedGeometric} to the test functions $h_j = h * \psi_j$, $j \ge 0$, which satisfy $\sum_{j=0}^{\infty} h_j = h$, and approximate the polynomial $J^T_{\operatorname{nc}} (\fn{}{h_j}{\tau})$ by the truncated integral 
$j^T_{\operatorname{nc}} (\fn{}{h_j}{\tau})$ for 
$T=T_j$ with $d (T_j) \ge \eta \norm{T_j}$ and
$C (1 + j + \log t) \le d (T_j) \le
2 C (1 + j + \log t)$. 

We obtain from Proposition \ref{PropositionTruncatedGeometric} that
\begin{multline} \label{LittlewoodPaleyEstimate1}
j_{\operatorname{nc}}^{T_j} (\fn{}{h_j}{\tau}) \ll
\norm{\tau}_1 t^{-r}
\int_{i (\af^G_0)^*}  
\abs{\hat{h} (\lambda)} \hat{\psi}_j (\lambda) \beta (\lambda) D (t^{-1} \lambda)^{-1} \\
\left[ (\log t)^{r-1} 
+ \nu_G (\log (2 + t^{-1} \norm{\lambda}))^{r-1} 
+  (1 +  j + \log t)^r \right]
\, d \lambda,
\end{multline}
and from \eqref{EquationApproximation2} that
\begin{multline} \label{LittlewoodPaleyEstimate2}
J_{\operatorname{nc}}^{T_j} (\fn{}{h_j}{\tau})
- j_{\operatorname{nc}}^{T_j} (\fn{}{h_j}{\tau}) \ll
\norm{\tau}_1 (1 +  j + \log t)^r e^{-C (1+ j + \log t)} \\
\int_{i (\af^G_0)^*}  
\abs{\hat{h} (\lambda)} \hat{\psi}_j (\lambda) \beta (\lambda) 
(1 + \norm{\lambda})^N
\, d \lambda
\end{multline}
for some fixed integer $N$. The constant $C$ is at our disposal.
By a variant of Arthur's interpolation argument \cite{Ar82}*{Lemma 5.2}, $J_{\operatorname{nc}} (\fn{}{h_j}{\tau})$, and
more generally $J_{\operatorname{nc}}^T (\fn{}{h_j}{\tau})$
for $T$ in a fixed compact subset $\CmC$ of $\af_0$,
is bounded
by a constant multiple of the sum of the right-hand sides of
\eqref{LittlewoodPaleyEstimate1} and
\eqref{LittlewoodPaleyEstimate2}. Summing over all $j$, we obtain as upper 
bound for
$J_{\operatorname{nc}}^T (\fn{}{h}{\tau}) = \sum_{j=0}^{\infty}
J_{\operatorname{nc}}^T (\fn{}{h_j}{\tau})$, $T \in \CmC$, the sum of
\[
\norm{\tau}_1 t^{-r}
\int_{i (\af^G_0)^*}  
\abs{\hat{h} (\lambda)} \beta (\lambda) D (t^{-1} \lambda)^{-1} 
\left[ (\log t)^{r-1} 
+ \nu_G (\log (2 + t^{-1} \norm{\lambda}))^{r-1} \right]
\, d \lambda,
\]
of
\[
\norm{\tau}_1 t^{-r}
\int_{i (\af^G_0)^*}  
\abs{\hat{h} (\lambda)} \beta (\lambda) D (t^{-1} \lambda)^{-1} 
\left[ \sum_{j=0}^{\infty} (1 + j + \log t)^r \hat{\psi}_j (\lambda) \right]
\, d \lambda,
\]
and of
\[
\norm{\tau}_1 
\int_{i (\af^G_0)^*}  
\abs{\hat{h} (\lambda)} \beta (\lambda) 
\left[ 
\sum_{j=0}^{\infty} (1 + j + \log t)^r e^{- C (1 + j + \log t)} \hat{\psi}_j (\lambda) 
\right]
(1 + \norm{\lambda})^N
\, d \lambda.
\]
The first two terms clearly satisfy the bound asserted by the theorem.
Noting that $t^{-r} D (t^{-1} \lambda)^{-1} \gg
(t + \norm{\lambda})^{-r}$,
we see that the third term satisfies the same bound if we take $C = N + r$.

We have proved the first assertion of the theorem, and the second assertion for all $T$ in a fixed compact set. By interpolation, the second assertion is easily seen to hold for all $T$.
\end{proof}

\section{The main geometric estimate} \label{MainGeometric}

\subsection{Reduction to integrals over Bruhat cells}
As in \eqref{EquationTruncatedNc}, consider the truncated integral $j_{\operatorname{nc}}^T (f)$ of the non-central part of the kernel
for a sufficiently regular truncation parameter $T$.
Let $\phi$ be a fixed non-negative smooth function of compact support on $\R^{\ge 0}$, and $\tau$ a compactly supported function on $G(\A_\fin)$, right-invariant under a fixed open subgroup $K$ of $\K_\fin$.
We want to estimate the truncated integral for the non-negative test functions
$f=F^{\phi,\epsilon} \otimes \tau$, where $0 < \epsilon < 1/2$. 
Let $\tilde{\tau} (g_{\fin}) = \int_{\K_\fin} \int_{\K_\fin} \abs{\tau} (k_1 g k_2) \, dk_1 dk_2$ be the projection of $\abs{\tau}$ to the bi-$\K_\fin$-invariant functions.
For a measurable function $\phi$ on $T_0 (\A_\fin)$ set
\begin{equation} \label{l1normdivisor}
\norm{\phi}_{1,\operatorname{div}} = \int_{T_0 (\A_\fin)} 
C^{\abs{\{ p \, : \, t_p \notin Z_G (\Q_p) T_0 (\Z_p) \}}} \abs{\phi (t)}
\, dt,
\end{equation}
where 
$C>1$ is a constant (depending only on $G$). Note that
$\norm{\tilde{\tau}^{(P_0)}}_{1,\operatorname{div}} \ll
\norm{\tau}_{2-\eta}$ for any $0 < \eta \le 1$. 
Our result is
\begin{theorem} \label{TheoremMainEstimateNonnegative} 
For all $0 < \epsilon < 1/2$ 
we have:
\begin{align*}
j_{\operatorname{nc}}^T (F^{\phi,\epsilon} \otimes \tau) & \ll_\phi    
\epsilon^{\frac{r}{2}} \sum_{Q = LV, \, Q \neq P_0} (1+\norm{T_{L}})^{r - r_L} \abs{\log \epsilon}^{r_L - 1} \norm{\tilde{\tau}^{(Q)}}_1 \\
&  
\quad\quad\quad\quad\quad\quad\quad\quad\quad\quad +  \epsilon^{\frac{r}{2}} (1+\norm{T})^{r}  \norm{\tilde{\tau}^{(P_0)}}_{1,\operatorname{div}}
\\
& \ll   
\epsilon^{\frac{r}{2}}  \left( \abs{\log \epsilon}^{r - 1} + 
(1+\norm{T})^{r} \right) \norm{\tau}_1.
\end{align*}
\end{theorem}

A more flexible variant is the following.

\begin{corollary} \label{CorollaryMainEstimateNonnegative} 
Fix $c > 0$. For any non-negative monotonically decreasing function $\phi$ on $[0,c]$, extended by zero to $\R^{\ge 0}$, we have the estimate
\begin{equation}
j_{\operatorname{nc}}^T (F^{\phi} \otimes \tau) 
\ll_c  \norm{\tau}_1 \int_0^c\phi (x)
x^{\frac{r}{2} - 1}  \left( \abs{\log x}^{r - 1} + 
(1+\norm{T})^{r} \right) \, dx.
\end{equation}
\end{corollary}

\begin{proof} 
For fixed $\tau$ and fixed $T$, we can write
$j_{\operatorname{nc}}^T (F^{\phi} \otimes \abs{\tau}) = 
\int_0^c \phi (x) d \nu$ for a Radon measure $\nu$ on $[0,c]$.
Let $N (x) = \int_0^x d \nu$. By Theorem \ref{TheoremMainEstimateNonnegative}, applied for a suitable $\phi$, we have
\[
N(x) \le  N_0 (x) = C x^{\frac{r}{2}}  \left( \abs{\log x}^{r - 1} + 
(1+\norm{T})^{r} \right) \norm{\tau}_1
\]
for a constant $C > 0$, where the right-hand side is monotonically increasing in $x$.
This implies that
$\int_0^c \phi (x) d \nu  = \int_0^c \phi (x) d N (x)
\le \int_0^c \phi (x) d N_0 (x) = \int_0^c \phi (x) N'_0 (x) dx$
for all monotonically decreasing step functions $\phi$ on $[0,c]$. Since monotonically decreasing functions are Riemann integrable, we obtain the
assertion for arbitrary monotonically decreasing $\phi$ by approximating 
$\phi$ from above by step functions.
\end{proof} 

\begin{remark} For groups $G$ with a root system of type $A_n$ or $C_n$, there exists a geometric unipotent orbit of dimension $2r$. A consideration of the corresponding orbital integrals shows that for these groups the exponents $\frac{r}{2}$ of $\epsilon$ in Theorem \ref{TheoremMainEstimateNonnegative} and $\frac{r}{2}-1$ of $x$ in 
Corollary \ref{CorollaryMainEstimateNonnegative} are already optimal. For other groups the minimal dimension of a non-trivial geometric unipotent orbit, which we denote by
$2 d_{\operatorname{min}}$, is strictly greater than $2r$ (cf. Table \ref{tab:parabolics} below), and we expect a bound of the order
$O_\tau (\epsilon^{d_{\operatorname{min}}/2})$ in Theorem \ref{TheoremMainEstimateNonnegative}. 
\end{remark} 

The proof of Theorem \ref{TheoremMainEstimateNonnegative} will occupy the remainder of this section. Note first that by replacing $\tau$ by $\tilde{\tau}$, we can assume that $\tau$ is non-negative and bi-$\K_\fin$-invariant.
Set
\begin{equation}\label{eq:chi}
\chi_{T} (a)=\tau_0 (H_0 (a)-T_1)\widehat{\tau}_0 (T-H_0(a)).
\end{equation}
By reduction theory, we can bound the truncated integral of a non-negative measurable function $f$ by 
\begin{multline*}
j_{\operatorname{nc}}^T (f)  \le  
\int_{\cpt} \int_{U_0(\Q)\backslash U_0(\A)} \int_{A^G_0} \int_{T_0(\Q)\backslash T_0(\A)^1} \sum_{\gamma\in G(\Q) \smallsetminus Z(\Q)} f ((uatk)^{-1}\gamma uatk) \cdot  \\
 \chi_T(a) \delta_0 (a)^{-1}
\, dt \, da \, du \, dk.
\end{multline*}
Since we integrate on the right hand side over a domain consisting of $P_0 (\Q)$-cosets, we can split the kernel sum according to the Bruhat decomposition and obtain
\[
j_{\operatorname{nc}}^T (f) \le 
\sum_{w \in W, \, w \neq 1} \CmJ_w^{G,T}(f) + \CmJ_{1, \operatorname{nc}}^{G,T}(f),
\]
where for each Weyl group element $w \in W$ we denote by
\begin{multline}\label{eq:weylsum}
\CmJ_w^{G,T}(f)
 = \int_{\cpt} \int_{U_0(\Q)\backslash U_0(\A)} \int_{A^G_0} \int_{T_0(\Q)\backslash T_0(\A)^1} \sum_{\gamma\in P_0(\Q) w U_0(\Q)} f ((uatk)^{-1}\gamma uatk) \cdot \\ 
 \chi_T(a) \delta_0 (a)^{-1}
\, dt \, da \, du \, dk
\end{multline}
the contribution of the corresponding Bruhat cell, and by
\begin{multline}\label{eq:borelsum}
\CmJ_{1, \operatorname{nc}}^{G,T}(f)  = 
\int_{\cpt} \int_{U_0(\Q)\backslash U_0(\A)} \int_{A^G_0} \int_{T_0(\Q)\backslash T_0(\A)^1}  \sum_{\gamma\in P_0(\Q) \smallsetminus Z(\Q)} f ((uatk)^{-1}\gamma uatk) \cdot \\ 
 \chi_T(a) \delta_0 (a)^{-1} \, dt \, da \, du \, dk
\end{multline}
the contribution of the non-central elements of the Borel subgroup ($w=1$).

We will estimate $\CmJ_w^{G,T}(f)$ for $w \neq 1$ by the method of \cites{FiLa11,FiLa16}. Our precise result on these integrals is the following.
Let $Q(w)=L(w)V(w)$ be the smallest standard parabolic containing $w$.
In the following, we call Weyl group elements $w$ with $Q(w)=G$ regular.
For regular $w$, the sum over $\gamma$ in \eqref{eq:weylsum}
is compactly supported as $a$ ranges over $\{ a \in A_0^G \, : \,
\tau_0 (H_0 (a) - T_1) = 1 \}$, and the integral
$\CmJ_w^{G,T}(f)$ therefore assumes a constant value
for $d(T)\to\infty$. We 
denote this limit value by $\CmJ_w^{G}(f)$.

\begin{proposition} \label{PropositionWeylNeqOne}
For regular $w \in W$ and all $0 < \epsilon < 1/2$ we have
\[
\CmJ_w^{G}(F^{\phi,\epsilon} \otimes \tau) \ll_\phi  \epsilon^{\frac{r}{2}} \abs{\log \epsilon}^{\dim (\aaa^G_0)^w}
\norm{\tau}_1,
\]
where $(\af^G_0)^w$ denotes the subspace of $w$-invariant elements of $\af_0^G$.
For general $w \in W$, $w \neq 1$, with $Q(w)=Q=LV$ we have
\[
\CmJ_w^{G,T}(F^{\phi,\epsilon} \otimes \tau) \ll_\phi  
(1+\norm{T_{L}})^{r - r_L}
\epsilon^{\frac{r}{2}} \abs{\log \epsilon}^{\dim (\aaa^L_0)^w}
\norm{\tau^{(Q)}}_1.
\]
\end{proposition}

This is complemented by an estimate for the contribution of the Borel 
subgroup.

\begin{proposition} \label{PropositionUnipotent}
For all $0 < \epsilon < 1/2$ 
we have
\[
\CmJ_{1, \operatorname{nc}}^{G,T}(F^{\phi,\epsilon} \otimes \tau) 
\ll_\phi
(1+\norm{T})^r \epsilon^{r/2} \norm{\tau^{(P_0)}}_{1,\operatorname{div}}. 
\]
\end{proposition}

Taken together, these estimates imply 
Theorem \ref{TheoremMainEstimateNonnegative}.

\begin{proof}[Proof of Theorem \ref{TheoremMainEstimateNonnegative}]
Sum the estimates of Proposition \ref{PropositionWeylNeqOne}
over all $w \neq 1$, noting that 
$\dim (\aaa^L_0)^w \le r_L - 1$ for $L=L(w)$, and add the estimate of
Proposition \ref{PropositionUnipotent} to account for $w=1$.
\end{proof}

In the following, we first prove Proposition \ref{PropositionWeylNeqOne}.
The proof will be finished in \S\ref{SubsectionNonBorel} below.
Proposition \ref{PropositionUnipotent} will be proved later in 
\S\S\ref{SubsectionBorel}--\ref{SubsectionUnipotent}, which will finish the proof of Theorem \ref{TheoremMainEstimateNonnegative} and its corollary.

In a first step, we will use the method of \cites{FiLa11,FiLa16} to estimate $\CmJ_w^{G,T}(f)$ for a non-negative test function $f$ of bounded level by an adelic integral.
Set
\begin{equation} \label{BruhatTruncated}
I_w^{G,T}(f)
= \int_{U_0(\A) / U_w (\A)} \int_{A^G_0} 
 \int_{U_0(\A)}
\int_{T_0(\A)^1}  f (a^{-1} u w a u_1 t)  \chi_T(a) \, dt \, du_1 \, da \, du,
\end{equation}
where $U_w=U_0\cap wU_0w^{-1}$. This converges for any compactly supported continuous function $f$. 
The analogous non-compact adelic integral is
\begin{equation} \label{BruhatSiegel}
I_w^{G}(f)
= \int_{U_0(\A) / U_w (\A)} \int_{A^G_0} 
 \int_{U_0(\A)}
\int_{T_0(\A)^1} f (a^{-1} u w a u_1 t)  \tau_0 (H_0 (a)-T_1) \, dt \, du_1 \, da \, du,
\end{equation}
whenever it is convergent, in which case it is obviously the limit of $I_w^{G,T}(f)$ for $d(T) \to \infty$.
In \cite{FiLa11}*{Proposition 5.1}, convergence of \eqref{BruhatSiegel} is established for regular Weyl group elements $w$ (cf. Proposition
\ref{PropositionBruhatSmooth} below).

\begin{lemma}
For any non-negative test function $f$, invariant under 
conjugation by $\K$ and right-$K$-invariant, we can estimate
\begin{equation} \label{EstimateAdelicIntegral}
\CmJ_w^{G,T}(f) \le I_w^{G,T}(\tilde f),
\end{equation} 
where $\tilde f (g) = \max_{k \in \K} \sum_{X \in B_{\le \abs{\Phi^+}}} \abs{f * X} (gk)$ for a suitable basis $B_{\le \abs{\Phi^+}}$
of $\cU (\uf_0)_{\le \abs{\Phi^+}}$, the subspace of $\cU(\uf_0)$ spanned by elements of degree $\le\abs{\Phi^+}$.
In particular, for non-negative $\tau$ we have
\begin{equation} \label{EqnAdelicIntegralSpecialCase}
\CmJ_w^{G,T}(F^{\phi,\epsilon} \otimes \tau) \le 
\epsilon^{-\frac{\abs{\Phi^+}}{2}} 
I_w^{G,T}
(F^{\tilde\phi,\epsilon} \otimes \tau)
\end{equation} 
for a suitable function $\tilde\phi \in C^\infty_c (\R^{\ge 0})$ depending on $\phi$.
\end{lemma}

\begin{proof}
The lemma follows immediately from \cite{FiLa16}*{pp. 432--433}. For the convenience of the reader, we give a sketch of the argument.
Using that $T_0 (\A)^1 = T_0 (\Q) (T_0 (\A) \cap \K)$, we can 
estimate $\CmJ_w^{G,T}(f)$ from the definition
\eqref{eq:weylsum} by
\begin{multline*}
\CmJ_w^{G,T}(f)
 \le  \int_{U_0(\Q)\backslash U_0(\A)} \int_{A^G_0} 
\sum_{u_2 \in U_w (\Q) \backslash U_0(\Q)} 
\sum_{m \in T_0(\Q)}
\sum_{u_1 \in U_0(\Q)}
f (a^{-1} u^{-1} u_2^{-1} m w u_1 u a)\cdot \\ 
 \chi_T(a) \delta_0 (a)^{-1}
\, da \, du 
\end{multline*}
We apply \cite{FiLa16}*{Lemma 3.3} to the sum over $u_1$, estimating it by the integral of suitable derivatives on $U_0$ up to order $\abs{\Phi^+}$. We
obtain the estimate of 
\begin{multline*}
\CmJ_w^{G,T}(f)
 \le  \int_{U_0(\Q)\backslash U_0(\A)} \int_{A^G_0} 
\int_{U_0 (\A)} 
\sum_{u_2 \in U_w (\Q) \backslash U_0(\Q)} 
\sum_{m \in T_0(\Q)}
\tilde{f} (a^{-1} u^{-1} u_2^{-1} m w a u_1) \cdot \\ 
 \chi_T(a) 
\, d u_1 \, da \, du. 
\end{multline*}
It is easy to see that the right-hand side reduces to
\[
\int_{U_w(\A)\backslash U_0(\A)} \int_{A^G_0} 
\int_{U_0 (\A)} 
\sum_{m \in T_0(\Q)}
\tilde{f} (a^{-1} u^{-1} w a u_1 m) \chi_T(a) 
\, du_1 \, da \, du, 
\]
where we can finally estimate the sum over $m$ by an integral over 
$T_0 (\A)^1 = T_0 (\Q) (T_0 (\A) \cap \K)$ using the invariance of $\tilde{f}$ under right translation by
$\K$. This shows \eqref{EstimateAdelicIntegral}.
An application of Lemma \ref{LemmaDerivativesTestFunctions} to estimate $\tilde{f}$ yields 
\eqref{EqnAdelicIntegralSpecialCase}.
\end{proof}

For regular $w \in W$ this yields the estimate
\[
\CmJ_w^{G} (F^{\phi,\epsilon} \otimes \tau) \le 
\epsilon^{-\frac{\abs{\Phi^+}}{2}} 
I_w^{G}
(F^{\tilde\phi,\epsilon} \otimes \tau).
\]
For non-regular $w \neq 1$ 
we can easily estimate $I_w^{G,T} (f)$ in terms of 
$I_w^{L(w)}$ (cf. Remark \ref{RemarkDescent}). However, it is possible to do the descent 
to $L(w)$ more carefully and to gain an additional factor 
of $\epsilon^{\frac{r-r_{L(w)}}{2}}$ in the final estimate.

For $T_L \in \aaa_L^G$ let $v_L (T_L)$ denote the volume of the bounded convex domain in the space $\aaa_L^G$ defined by $\tau_L (X - T_{1,L}) \hat\tau_L (T_L - X) = 1$.
It is clear from the definition that
\[
v_{L} (T_{L}) \ll (1+\norm{T_{L}})^{\dim \af^G_{L}}.
\]

The descent from $G$ to $L(w)$ can now be formulated as follows.

\begin{lemma} \label{LemmaDescent} 
Let $w \in W$, $w \neq 1$, and $Q(w) = L(w) V(w)$. Then
we have
\begin{align*}
\CmJ_w^{G,T} (F^{\phi,\epsilon;G} \otimes \tau) & \le 
v_L (T_L) \epsilon^{\frac{r-r_{L(w)}}{2}} 
\CmJ_w^{L(w)}
(F^{\psi,\epsilon;L(w)} \otimes \tau^{(Q(w))}) \\
& \le   
v_L (T_L) \epsilon^{\frac{r-r_{L(w)}-\abs{\Phi_L^+}}{2}} 
I_w^{L(w)}
(F^{\tilde\psi,\epsilon;L(w)} \otimes \tau^{(Q(w))}),
\end{align*}
where $\psi$, $\tilde\psi \in C^\infty_c (\R^{\ge 0})$ depend on $\phi$.
\end{lemma}

\begin{remark} \label{RemarkDescent} 
It is easy to see that
\[
I_w^{G,T} (f) = 
\int_{U^L_0(\A) / U^L_w (\A)} \int_{A^L_0} 
 \int_{U^L_0(\A)}
\int_{T_0(\A)^1}  f^{(Q(w))} (a^{-1} u w a u_1 t)  \nu^L_T(a) \, dt \, du_1 \, da \, du
\]
with
\[
\nu^L_T(a) =  \int_{A_L^G}  \chi_T (ab)  \, db,
\]
and we can estimate
\[
\nu^L_T(a) 
\le  v_L (T_L) \tau^L_0 (H_0 (a) - T_1).
\]
Therefore, we obtain
\[
I_w^{G,T}(f) \le v_{L(w)} (T_{L(w)}) I_w^{L(w)} (f^{(Q(w))})
\ll (1+\norm{T_{L(w)}})^{\dim \af^G_{L(w)}}
 I_w^{L(w)} (f^{(Q(w))}).
\]
Noting that
\[
(F^{\phi,\epsilon; G} \otimes \tau)^{(Q)} \le
\epsilon^{\dim V / 2}
F^{\psi,\epsilon; L} \otimes \tau^{(Q)}
\]
for all non-negative $\tau$, we obtain a weaker variant of 
Lemma \ref{LemmaDescent} without the
factor $\epsilon^{\frac{r-r_{L(w)}}{2}}$.
\end{remark}

For the proof of Lemma \ref{LemmaDescent}, we first isolate 
the most important technical step as a separate lemma.

\begin{lemma} \label{LemmaLatticePoints}
Let $Q=LV$ be a standard maximal parabolic subgroup of $G$
and $\vf^{\operatorname{ab}}$ be the Lie algebra of the abelianization $V^{\operatorname{ab}} = V / [V,V]$ of the unipotent radical $V$ of $Q$.

Let $f_0$ be a fixed bounded compactly supported function on $\R^{\ge 0}$ and
$\tau_0$ be a compactly supported continuous function on $\vf^{\operatorname{ab}} \otimes \A_{\fin}$ that is invariant under translation by 
$\vf^{\operatorname{ab}}_{\Z} \otimes \hat{\Z}$. 
Let $k$ be the rank of the endomorphism 
$1 - \Ad(\gamma_L)^{-1}$
of the $\Q$-vector space $\vf^{\operatorname{ab}}$.
Let $f_\epsilon = f_0 (\norm{\cdot}^2 \epsilon^{-1})$.
Then we have
\begin{equation} \label{EquationSumIntegralDescent}
\int_{\vf^{\operatorname{ab}} \backslash 
\vf^{\operatorname{ab}} \otimes \A}
\sum_{\nu \in \vf^{\operatorname{ab}}} 
( f_\epsilon \otimes \tau_0 ) \left(\Ad (a)^{-1} 
((1 - \Ad (\gamma_L)^{-1}) v + \nu) \right)
\, d v
\ll_{f_0} \epsilon^{k/2}  \delta_{V^{\operatorname{ab}}} (a) 
\norm{\tau_0}_1
\end{equation}
for all $0 < \epsilon < 1$ and
$a \in A_0$ with $\alpha (a) \gg 1$ for all $\alpha \in \Phi_V^+$. Here $\delta_{V^{\operatorname{ab}}}$ denotes the modulus function of $A_0$ on $V^{\operatorname{ab}}$. 
\end{lemma} 

\begin{proof}
Let $Y \subset \vf^{\operatorname{ab}}$ be the image of the endomorphism $1 - \Ad (\gamma_L)^{-1}$, a subspace of 
dimension $k$. Rewrite the left-hand side of \eqref{EquationSumIntegralDescent} as
\[
\int_{Y \otimes \A}
\sum_{\nu \in \vf^{\operatorname{ab}} / Y} 
( f_\epsilon \otimes \tau_0 ) \left(\Ad (a)^{-1} 
(y + \nu) \right)
\, d y.
\]
Choose $S \subset \Phi_{V^{\operatorname{ab}}}^+$ with complement
$\bar{S} = \Phi_{V^{\operatorname{ab}}}^+ \smallsetminus S$ such that
$\vf^{\operatorname{ab}}$ is the direct sum of $Y$ and
the root spaces $\uf_{\alpha}$, 
$\alpha \in \bar{S}$. The projection map $Y \longrightarrow
\bigoplus_{\alpha \in S} \uf_{\alpha}
$ is then an isomorphism. The space
of all $\nu \in \vf^{\operatorname{ab}}$ with
$\nu_\alpha = 0$ for all $\alpha \in S$ is a system of representatives
for the quotient $\vf^{\operatorname{ab}} / Y$. 

Consider first the case where $\tau_0$ is the characteristic function of a
coset $\xi + \vf^{\operatorname{ab}}_{\Z} \otimes \hat{\Z}$
for some 
$\xi \in \vf^{\operatorname{ab}} \otimes \A_{\fin}$.
In this case we can estimate the sum over $\nu$ trivially by 
$\ll_{f_0} \prod_{\alpha \in \bar{S}} \alpha (a)$. Moreover, the integrand can only be non-zero when the set $\Ad (a)^{-1}  y +
\bigoplus_{\alpha \in \bar{S}} \uf_{\alpha} \otimes \A$ meets the support of $f_\epsilon \otimes \tau_0$, which translates into 
the conditions that $\abs{y_{\alpha,\infty}} \ll_{f_0} \alpha (a) \epsilon^{1/2}$ and $y_{\alpha,\fin} \in \xi_\alpha + \hat{\Z}$ for $\alpha \in S$.
The integration over
$Y \otimes \A$ can be replaced by integration over the coordinates $y_\alpha$, $\alpha \in S$, and 
we obtain a final bound of 
$\ll_{f_0} \epsilon^{k/2}  \delta_{V^{\operatorname{ab}}} (a)$, since the set $S$ has $k$ elements. Writing a general $\tau_0$ as a linear combination of characteristic functions finishes the proof.
\end{proof}

\begin{proof}[Proof of Lemma \ref{LemmaDescent}]
The proof proceeds by induction over standard Levi subgroups $L$ of $G$ containing $L(w)$.
For a compactly supported bounded measurable function $\chi$ on $A^G_0$ let $\CmJ_w^{G}
(f; \chi)$ be the variant of $\CmJ_w^{G,T} (f)$ obtained by 
replacing $\chi_T$ by $\chi$.

The induction step is that
\begin{equation} \label{EquationDescentStep}
\CmJ_w^{G} (F^{\phi,\epsilon;G} \otimes \tau; \chi) \le 
\epsilon^{1/2} 
\CmJ_w^{L}
(F^{\psi,\epsilon;L} \otimes \tau^{(Q)};
\int_{A_L^G} \chi (\cdot b) \, db),
\end{equation}
where $Q=LV$ is a standard maximal parabolic subgroup of $G$ containing $Q(w)$.
Assuming this, induction easily gives
\[
\CmJ_w^{G} (F^{\phi,\epsilon;G} \otimes \tau; \chi) \le 
\epsilon^{(r - r_L)/2} 
\CmJ_w^{L}
(F^{\psi,\epsilon;L} \otimes \tau^{(Q)};
\int_{A_L^G} \chi (\cdot b) \, db)
\]
for any $Q$ containing $Q(w)$, in particular for $Q=Q(w)$ itself, and setting $\chi = \chi_T$ proves the lemma.

To show \eqref{EquationDescentStep}, consider 
the definition of  $\CmJ_w^{G} (f; \chi)$ and write
$\gamma = \gamma_L \nu$ with $\gamma_L \in P_0^L (\Q) w U_0^L (\Q)$, $\nu \in V (\Q)$, as well as $u = v u_L$, $v \in V (\Q) \backslash V (\A)$, $u_L \in U_0^L (\Q) \backslash U_0^L (\A)$.
We get
\begin{align*}
\CmJ_w^{G,T}(f; \chi)
 = & \int_{U^L_0 (\Q)\backslash U^L_0(\A)} \int_{A^G_0} \int_{T_0(\Q)\backslash T_0(\A)^1} \sum_{\gamma_L \in P^L_0(\Q) w U^L_0(\Q)} \\
 & \bigg[ \delta_Q (a)^{-1} \int_{V (\Q) \backslash V (\A)} 
\sum_{\nu \in V (\Q)} 
 f (l^{-1} \gamma_L l \cdot l^{-1} (\gamma_L^{-1} v^{-1} \gamma_L
\nu v) l) \, d v \bigg] \\ 
 &\quad\quad\quad\quad\quad\quad\quad\quad\quad\quad\quad\quad\quad\quad\quad\quad\quad\quad \chi (a) \delta^L_0 (a)^{-1}
\, dt \, da \, du_L,
\end{align*}
where we write $l = u_L a t \in L (\A)$.

We want to estimate the expression in brackets and improve on the
obvious upper bound of  
$(F^{\psi,\epsilon;L} \otimes \tau^{(Q)}) (l^{-1} \gamma_L l)$.
For this, we first estimate summation over the commutator subgroup $[V,V] (\Q)$ of $V(\Q)$ trivially (cf. \cite{FiLa11}*{\S 4}), which gives 
\begin{multline*}
\delta_Q (a)^{-1} \int_{V (\Q) \backslash V (\A)}
\sum_{\nu \in V (\Q)} 
 f (l^{-1} \gamma_L l \cdot l^{-1} (\gamma_L^{-1} v^{-1} \gamma_L
\nu v) l) \, d v 
\le \delta_{V^{\operatorname{ab}}} (a)^{-1}
F^{\psi,\epsilon;L} (l^{-1} \gamma_L l) \\
\int_{\vf^{\operatorname{ab}} \backslash 
\vf^{\operatorname{ab}} \otimes \A}
\sum_{\nu \in \vf^{\operatorname{ab}}} 
( f_\epsilon \otimes \tilde{\tau}_{l^{-1} \gamma_L l} ) \left(\Ad (l)^{-1} 
((1 - \Ad (\gamma_L)^{-1}) v + \nu) \right)
\, d v
\end{multline*}
for
$\tilde{\tau}_{l} (v) = \int_{[V,V] (\A_{\fin})} \tau (l \exp (v) \tilde{v})
\, d \tilde{v}$.
An application of Lemma \ref{LemmaLatticePoints} to the sum-integral 
on the right-hand side gives an upper bound of
\[
\ll \epsilon^{k/2}
(F^{\psi,\epsilon;L} \otimes \tau^{(Q)}) (l^{-1} \gamma_L l).
\]
Since $w \neq 1$, the element $\gamma_L$ cannot lie in the center of $G$. However, the kernel of the representation $\Ad$ of the group $L$ on $\vf^{\operatorname{ab}}$ is precisely the center of $G$.
Therefore we have $k \ge 1$, which gives
\eqref{EquationDescentStep}.
\end{proof}

\subsection{Estimation for regular Bruhat cells}
We are reduced to the regular case for standard Levi subgroups $L$ of $G$.

We therefore turn to the estimation of $I_w^G (f)$ for regular Weyl group elements $w$.
We first replace \eqref{BruhatSiegel} by a variant with a smooth cutoff
of the fundamental domain. (This is a technical refinement that allows us to optimize the final estimate, cf. the remark after Proposition \ref{PropEstimateIw} below.) Let $\kappa(X)$ be a smooth function on 
$\af^G_0$ that majorizes $\tau_0 (X-T_1)$. Its Fourier transform
$\hat \kappa(\lambda)$ is holomorphic if $\Re \lambda$ lies in the open cone spanned by the positive roots, and it is rapidly decreasing in $\Im \lambda$ for fixed $\Re \lambda$ in this cone.
Clearly, $I_w^{G}(f)$ is bounded from above by
\begin{equation} \label{BruhatSiegelSmooth}
I_w^{G,\, {\rm sm}}(f)
= \int_{U_0(\A) / U_w (\A)} \int_{A^G_0} 
 \int_{U_0(\A)}
\int_{T_0(\A)^1} f (a^{-1} u w a u_1 t)  \kappa (H_0(a)) \, dt \, du_1 \, da \, du.
\end{equation}

For each $w \in W$ let $\Phi^+_w$ be the set of $\alpha\in\Phi^+$ with $w\alpha<0$, and let
\[
m (w, \lambda) =
\prod_{\alpha \in \Phi^+_{w}} m_0 (\sprod{\lambda}{\alpha^\vee})
\]
be the scalar-valued spherical intertwining operator (as in \cite{FiLa11}*{\S 3.3}). Denoting by
$\xi (s) = \pi^{-s/2} \Gamma (s/2) \zeta (s)$ the completed Riemann zeta function, we have $m_0 (s) = \xi (s)/\xi (s+1)$.

We write
\[
\Phi (f, \lambda) = \int_{A^G_0} \int_{P_0 (\A)^1} 
f (p a) a^{-\lambda - \rho} \, dp \, da
\]
for the adelic Harish-Chandra transform of $f$ (we view elements $\mu$ of
$(\af^G_0)^*_{\C}$ as characters of $A_0^G$, using the exponential notation
$a^\mu$).

\begin{proposition} \label{PropositionBruhatSmooth}
Let $\lambda_0 \in (\af^G_0)^*$ be such that
$\lambda_0 - \rho$ lies in the positive Weyl chamber.
For any right-$\mathbf{K}$-invariant function $f \in C^\infty_c (G(\A)^1)$
and any regular $w \in W$ we have
\begin{equation}\label{eq:smooth:bruhat} 
I_w^{G,\, {\rm sm}} (f) = 
\int_{\lambda_{0} + i (\af^G_0)^*} 
\hat\kappa ((1-w^{-1}) \lambda)
m (w^{-1}, \lambda) \Phi (f, \lambda) \, d \lambda.
\end{equation}
\end{proposition}

\begin{proof} In \cite{FiLa11}*{Proposition 5.1} (which we apply in the special case $Q(w)=G$ and $\xi=0$), it was shown using Mellin inversion that
\begin{equation} \label{EquationBruhatIntegral} 
I_w^G (f) 
= \vol (\af^G_0 / \Z \hat \Delta^\vee ) 
\int_{\lambda_{0} + i (\af^G_0)^*} \frac{e^{-\sprod{(1-w^{-1}) \lambda}{T_1}}}{\prod_{\varpi^\vee
\in \hat \Delta^\vee} \sprod{(1-w^{-1}) \lambda}{\varpi^\vee}}
m (w^{-1}, \lambda) \Phi (f, \lambda) \, d \lambda.
\end{equation}
If we replace $\tau_0 (X-T_1)$ by $\kappa (X)$,
a simple modification of the argument yields \eqref{eq:smooth:bruhat}.  
\end{proof}

Note that the integrals in \eqref{eq:smooth:bruhat} and \eqref{EquationBruhatIntegral} converge absolutely because $\Phi (f, \lambda)$ is rapidly decreasing, $(s-1) m_0 (s)$ is of moderate growth
in the right half-plane
$\Re s \ge 0$, and 
\begin{equation} \label{EquationOneWInverse} 
\sprod{(1-w^{-1}) \lambda_0}{\varpi^\vee} > 0
\quad \text{for all} \quad \varpi^\vee
\in \hat \Delta^\vee
\end{equation} 
for regular $w$.

We will show that at least for the test functions $f=F^{\phi,\epsilon} \otimes \tau$ we 
can estimate the integrals $I_w^{G,\, {\rm sm}} (f)$ essentially by the $L^1$-norm
of $f$ if we allow ourselves to lose a power of $\abs{\log \epsilon}$.

\begin{proposition} \label{PropEstimateIw} 
For regular $w \in W$ and all $0 < \epsilon < 1/2$ we have
\[
I_w^{G,\, {\rm sm}} (F^{\phi,\epsilon} \otimes \tau)  \ll_\phi  \epsilon^{\frac{r+\abs{\Phi^+}}{2}} \abs{\log \epsilon}^{\dim (\aaa^G_0)^w}
\norm{\tau}_1.
\]
\end{proposition}

We remark that by a direct use of the integral expression \eqref{EquationBruhatIntegral} 
for $I_w^G (F^{\phi,\epsilon} \otimes \tau)$ (instead of the expression \eqref{eq:smooth:bruhat} for $I_w^{G,\, {\rm sm}} (F^{\phi,\epsilon} \otimes \tau)$) one obtains an analogous estimate, but with a larger power of $ \abs{\log \epsilon}$.

\subsection{Auxiliary estimates}
As a preparation for the proof of Proposition \ref{PropEstimateIw}
we first record a simple estimate for $\Phi (F^{\phi,\epsilon} \otimes \tau, \lambda)$.
\begin{lemma} \label{LemmaEstimatePhi}
For all $\lambda \in (\af_0^G)^*_\C$ with 
 $\Re \lambda$ lying in the convex hull of the $W$-orbit of $\rho$ we can bound
\[
\Phi (F^{\phi,\epsilon} \otimes \tau, \lambda) \ll_\phi \epsilon^{\frac{r+\abs{\Phi^+}}{2}} 
\min (1, \norm{\Im \lambda}^{-1/2} \epsilon^{-1/4})
\norm{\tau}_1.
\]
\end{lemma}

\begin{proof}
We can factor $\Phi (F^{\phi,\epsilon} \otimes \tau, \lambda)$ as
\[
\Phi (F^{\phi,\epsilon} \otimes \tau, \lambda) 
= (\cH F^{\phi,\epsilon}) (-\lambda) 
\Phi_{\fin} (\tau, \lambda) 
\]
with
\[
\Phi_{\fin} (\tau, \lambda) 
= \int_{G(\A_{\fin})} \tau (g) \prod_p \phi_{p, \, -\lambda} (g_p)
\, dg
\]
satisfying the bound $\abs{\Phi_{\fin} (\tau, \lambda)}
\le \norm{\tau}_1$ by Lemma \ref{lem:spherical:fct:1}.
To estimate the factor 
\[
(\cH F^{\phi,\epsilon}) (-\lambda) =
\int_{G(\R) / A_G} F^{\phi,\epsilon} (g) \phi_{-\lambda} (g)
\, dg,
\]
we use the
simple estimate 
of Proposition \ref{PropositionSphericalSimple}
for the archimedean spherical functions.
By the integration formula for $G(\R) / A_G$ in the Cartan decomposition, one obtains
\[
(\cH F^{\phi,\epsilon}) (-\lambda)
\ll_\phi \int_0^{\epsilon^{1/2}} \frac{x^{r+\abs{\Phi^+}-1}}{
(1 + \norm{\Im \lambda} x)^{1/2}} \, dx 
\ll \epsilon^{\frac{r+\abs{\Phi^+}}{2}} 
\min (1, \norm{\Im \lambda}^{-1/2} \epsilon^{-1/4}),
\]
if $G^{\operatorname{der}}$ is simple. The case of general $G$ can easily be
reduced to the simple case.
\end{proof}

It is crucial for our argument that we can bound $m_0 (s)$, and that it decays in fact at a precise rate for $\abs{\Im s} \to \infty$.
\begin{lemma} \label{LemmaEstimateZeta} 
Let $0 < c < 1$ be fixed. 
The function $m_0 (s)$ is bounded in the domain
$\Re s \ge 0$, $\abs{s-1} \ge c$.
We have
\[
m_0 (s) \ll_c  \frac{\abs{\zeta (s)}}{(1+\abs{t})^{1/2}}
\]
for $s = \sigma + it \neq 1$, $\sigma \ge c$, 
and
\[
m_0 (s) \ll_c  (1 + \abs{t})^{-1/2}
\]
for $\sigma \ge 1+c$.
\end{lemma}

\begin{proof}
The first assertion is a standard consequence of the Maass--Selberg relations
(cf. \cite{Ti:Riemann}*{p. 42}),
together with the fact that the only pole of $\zeta (s)$ is at $s=1$. 
Write
\[
m_0 (s) = 
\frac{\pi^{-\frac{s}{2}} \Gamma (\frac{s}{2})}{\pi^{-\frac{s+1}{2}} \Gamma (\frac{s+1}{2})} 
\frac{\zeta (s)}{\zeta (s+1)}. 
\]
By Stirling's formula,
\[
\frac{\pi^{-\frac{s}{2}} \Gamma (\frac{s}{2})}{\pi^{-\frac{s+1}{2}} \Gamma (\frac{s+1}{2})} = 
\left( \frac{2 \pi}{s} \right)^{1/2} \left(1 + O (\abs{s}^{-1}) \right) 
\quad \text{for} \quad \sigma \ge c.
\]
The function $\zeta (s+1)^{-1}$ is bounded in every half-plane $\sigma \ge c$. This shows the second assertion.
For $\sigma \ge 1+c$, the zeta function is also bounded, which yields the third assertion.
\end{proof}

The following estimate allows us to control the influence of the
factor $\zeta (s)$ in Lemma~\ref{LemmaEstimateZeta}.

\begin{lemma} \label{LogarithmicIntegralZeta} 
For all $0< \delta < 1/2$, all real numbers $y_1$ and $y_2$ and all $X \ge 0$ we have
\begin{align*}
\int_{-X}^X \frac{dx}{(1+\abs{x+y_1})^{1/2} (1+\abs{x+y_2})^{1/2}} 
& \le  2 \log (1+X), \\
\int_{-X}^X \frac{\abs{\zeta (1 - \delta + i (x+y_1))}}{(1+\abs{x+y_1})^{1/2} (1+\abs{x+y_2})^{1/2}} \, dx
& \ll_{\delta} \log (1+X),
\end{align*}
and
\[
\int_{-X}^X \frac{\abs{\zeta (1 - \delta + i (x+y_1)) \zeta (1 - \delta + i (x+y_2))}}{(1+\abs{x+y_1})^{1/2} (1+\abs{x+y_2})^{1/2}} \, dx
\ll_{\delta} \log (1+X).
\]
\end{lemma}

\begin{proof}
We first note that an explicit calculation yields
\[
\int_{-X}^X \frac{dx}{1+\abs{x+y_1}}
\le 2 \log (1+X).
\]
Using this together with the Cauchy-Schwarz inequality,  we only need to show that
\[
\int_{-X}^X \frac{\abs{\zeta (1 - \delta + i (x+y_1))}^2}{1+\abs{x+y_1}} \, dx
\ll_{\delta} \log (1+X).
\]
By partial integration, this inequality is a direct consequence of the standard
second moment estimate 
\[
\int_{0}^X \abs{\zeta (1 - \delta + i x)}^2 \, dx
\ll_{\delta} X, 
\]
which is valid for $0< \delta < 1/2$ (cf.\ \cite{Ti:Riemann}*{Theorem 7.2 (A)}).
\end{proof}

\subsection{Bases for root spaces}
For any $L\in\cL$ let $(\af_L^G)^w$ and $(\af_L^G)^{*,w}$ denote the spaces of $w$-invariant vectors in $\af_L^G$ and $(\af_L^G)^*$, respectively. Let
$(\af_L^G)^{*,w\perp}$ denote the orthogonal complement of $(\af_L^G)^{*,w}$ in $(\af_L^G)^*$, so that
\begin{equation}\label{eq:decomp:winvariants}
 (\af_L^G)^*= (\af_L^G)^{*,w} \oplus (\af_L^G)^{*,w\perp}.
\end{equation}
\begin{proposition}\label{prop:generating:roots}
If $w$ is regular, then the roots in $\Phi_{w^{-1}}^+$
span the space
$(\af_0^G)^*$.
Moreover, for any $L\in\cL$ there exists a subset
$\Psi_L \subset \Phi_{w^{-1}}^+$ such that the 
restrictions of the co-roots
$\alpha^{\vee}$, $\alpha\in\Psi_L$, to the space $(\af_L^G)^w$ 
are of the form $b_1, -b_1, b_2,  \ldots, b_d, -b_d$ for a basis 
$b_1,  \ldots, b_d$ of $(\af_L^G)^{*,w}$.
\end{proposition}

We first record two simple properties of the roots in $\Phi_{w^{-1}}^+$. For a positive root $\alpha \in \Phi^+$ and a simple root $\beta \in \Delta$, we write 
$\beta \prec \alpha$ if the coefficient $n_\beta$ of $\beta$ in the expression $\alpha = \sum_{\delta \in \Delta} n_\delta \delta$ is positive.

\begin{lemma}\label{lemma:roots:ordering}
Let $w\in W$ be regular.
\begin{enumerate}[label=(\roman{*})]
\item For every $\beta\in\Delta$ there exists $\alpha\in\Phi_{w^{-1}}^+$ with $\beta\prec \alpha$.
\item Suppose $\alpha\in\Phi_{w^{-1}}^+$ can be written as $\alpha=\alpha_0+\alpha_1$ with $\alpha_0, \alpha_1\in\Phi^+$. Then $\alpha_0\in\Phi_{w^{-1}}^+$ or $\alpha_1\in\Phi_{w^{-1}}^+$.
\end{enumerate}
\end{lemma}

\begin{proof}
\begin{enumerate}[label=(\roman{*})]
\item  
For any $w \in W$ the set of simple roots $\Delta^{Q(w)}$ defining the standard parabolic subgroup $Q(w)=Q(w^{-1})$ is precisely the set of all simple roots $\beta\in\Delta$ for which there exists $\alpha\in\Phi_{w^{-1}}^+$ with $\beta\prec \alpha$ (cf. \cite{FiLa11}*{p. 787}). If $w$ is regular, we have $Q(w)=G$, which means that all simple roots 
$\beta$ have this property.
\item Suppose $\alpha_0\notin\Phi_{w^{-1}}^+$, so that $w^{-1}(\alpha_0)> 0$. Since $0> w^{-1}(\alpha)=w(\alpha_0)+w(\alpha_1)$, it follows that $w^{-1} (\alpha_1)<0$, that is, $\alpha_1\in\Phi_{w^{-1}}^+$.
\end{enumerate} 
\end{proof}

\begin{lemma}\label{LemmaRootDecompositions}
For a positive root $\alpha$ let $V_\alpha$ be the subspace of
$(\af_0^G)^*$ spanned by $\alpha$ and by the summands $\beta$
and $\gamma$ in all possible decompositions 
$\alpha = \beta+\gamma$ of $\alpha$ as a sum of two positive
roots. Then $V_\alpha$ has as a basis the simple roots  
$\delta\in\Delta$ with $\delta\prec \alpha$. 
\end{lemma}

\begin{proof} We use induction over the height of $\alpha$.
The case where $\alpha$ is a simple root is trivial.
Let $\alpha$ be a positive root that is not simple and assume the 
assertion established for all roots of height less than the height of $\alpha$.
Since $\alpha^\vee$ is a non-negative linear combination of the simple co-roots $\delta^\vee$, there exists a simple root $\delta\in\Delta$ with
$\sprod{\alpha}{\delta^\vee} > 0$. This implies that 
$\alpha - \delta$ is a positive root (since $\alpha$ is not simple). In particular, $\delta$ is contained in the space $V_\alpha$. Let $q$ be the largest positive integer such that $\alpha' = \alpha - q \delta$ is a root, necessarily positive. Clearly, $\alpha' \in V_\alpha$. 
Also, $\sprod{\alpha'}{\delta^\vee} = - p - q$ for a non-negative integer $p$
\cite{Bourbaki}*{Ch.\ VI, \S1, Corollary to Proposition 9}. Let $\alpha' = \beta' + \gamma'$
be a decomposition of $\alpha'$ as a sum of two positive roots.
Since $\sprod{\beta'}{\delta^\vee} + \sprod{\gamma'}{\delta^\vee}= \sprod{\alpha'}{\delta^\vee} = - p - q \le - q$, and
$\beta' - \sprod{\beta'}{\delta^\vee} \delta$ 
and $\gamma' - \sprod{\gamma'}{\delta^\vee} \delta$ are roots,
there exist non-negative integers $i$ and $j$ with $i+j=q$ such that 
$\beta=\beta'+i\delta$ and $\gamma=\gamma'+j\delta$ are positive roots.
Since $\alpha = \beta+\gamma$, we have $\beta$, $\gamma
\in V_\alpha$, and therefore $\beta'$, $\gamma' \in V_\alpha$.

We can conclude that $V_{\alpha'}$ is a subspace of 
$V_{\alpha}$. By the induction hypothesis, the simple roots 
$\delta'\in\Delta$ with $\delta'\prec \alpha'$ are therefore contained in
$V_{\alpha}$, as well as the simple root $\delta$. This means that
$V_{\alpha}$ contains all simple roots $\prec \alpha$, and therefore their span. Since the opposite inclusion is clear, we obtain the assertion
for $\alpha$.
\end{proof}

\begin{proof}[Proof of Proposition~\ref{prop:generating:roots}]
To show the first assertion, observe that by the second assertion
of Lemma \ref{lemma:roots:ordering} the span of 
$\Phi_{w^{-1}}^+$ is also the sum of the spaces
$V_{\alpha}$ of Lemma \ref{LemmaRootDecompositions} as 
$\alpha\in \Phi_{w^{-1}}^+$. By Lemma \ref{LemmaRootDecompositions} 
and the first assertion of Lemma \ref{lemma:roots:ordering}, this space
contains all simple roots, and is therefore the full space $(\af_0^G)^*$.

For the second assertion we only need to consider $L=T_0$. The first assertion implies that the restrictions of the co-roots $\alpha^{\vee}$, $\alpha\in\Phi_{w^{-1}}^+$, to $(\af_0^G)^w$ span $(\af_0^G)^{*,w}$. Given $\alpha\in \Phi_{w^{-1}}^+$, such that $\alpha^\vee$ does not vanish on $(\af_0^G)^w$, we now show that there exists at least one other root $\beta\in \Phi_{w^{-1}}^+$, $\beta \neq \alpha$, such that $-\alpha^\vee$ and $\beta^\vee$ have the same restriction to $(\af_0^G)^w$.

Let $m>0$ be the smallest positive integer with $w^{-m}(\alpha)>0$. Since $\alpha\in \Phi_{w^{-1}}^+$, we have $m\ge2$. Then $-w^{-(m-1)} (\alpha) \in\Phi_{w^{-1}}^+$.
If $-w^{-(m-1)} (\alpha) \neq \alpha$, we can take  $\beta =-w^{-(m-1)}(\alpha)$.
If $-w^{-(m-1)} (\alpha)=\alpha$, then for every $X\in(\af_0^G)^w$ we have
\[
  \sprod{\alpha}{X}=  \sprod{-w^{-(m-1)} (\alpha)}{X}=
   \sprod{-\alpha}{w^{m-1}(X)}=- \sprod{\alpha}{X},
\]
and hence the restriction of $\alpha^\vee$ to $(\af_0^G)^w$ vanishes.
\end{proof}

\subsection{End of the proof in the non-Borel case} \label{SubsectionNonBorel}
We can now finish the proofs of Proposition \ref{PropEstimateIw} and
Proposition \ref{PropositionWeylNeqOne}.
\begin{proof}[Proof of Proposition \ref{PropEstimateIw}]
The proof of Proposition \ref{PropEstimateIw} begins with an application of the 
residue theorem to \eqref{eq:smooth:bruhat}. 
We move $\Re \lambda$ one root coordinate at a time across the singular hyperplanes
$\sprod{\lambda}{\alpha^\vee} = 1$ defined by
the simple roots $\alpha \in \Delta \cap \Phi^+_{w^{-1}}$,
and
write \eqref{eq:smooth:bruhat} as a linear combination with fixed positive coefficients of 
\begin{equation}\label{EquationBruhatResidues}
\int_{\lambda_{1,L} + i (\af^G_L)^*} \hat\kappa((1-w^{-1}) \lambda)
\prod_{\alpha \in \Phi^+_{w^{-1}} \smallsetminus \Delta^L} m_0 (\sprod{\lambda}{\alpha^\vee})
 \Phi (f, \lambda) \, d \lambda,
\end{equation}
where $L$ ranges over all standard Levi subgroups of $G$ with
$\Delta^L \subset \Phi^+_{w^{-1}}$, and
$\lambda_{1,L} = \rho - \delta \sum_{\alpha \notin \Delta^L}
\varpi_\alpha$ with
$0 < \delta < 1/2$ arbitrary but fixed. Note that the condition 
$0 < \delta < 1/2$ ensures that we do not cross a singular hyperplane 
associated to a root that is not simple.

We prove this assertion by induction. Consider the more general integral
\[
 \frac{v_{T_0}}{v_L} c^{r_L} \int_{\lambda_{0,L,M} + i (\af^G_L)^*} \hat\kappa((1-w^{-1}) \lambda)
\prod_{\alpha \in \Phi^+_{w^{-1}} \smallsetminus \Delta^L} m_0 (\sprod{\lambda}{\alpha^\vee})
 \Phi (f, \lambda) \, d \lambda,
\]
where $M$ is another standard Levi subgroup of $G$ with $M \cap L = T_0$, 
$\lambda_{0,L,M} = \lambda_{1,L} + 2
\delta \sum_{\alpha \in \Delta^M}
\varpi_\alpha$, $c$ is the residue of $m_0 (s)$ at $s=1$, 
and $v_L = \vol ((\af^G_L)^* / \Z \hat \Delta_L)$, where $\hat\Delta_L$ is the set of all $\varpi\in\hat\Delta\smallsetminus\hat\Delta^L$ restricted to $\af_L^G$. 
 \eqref{eq:smooth:bruhat} is of this form for $L=T_0$ and $M=G$, which starts the induction. For the induction step, assume that $M \neq T_0$ and let $\alpha \in \Delta^M$. Moving the real part of the coordinate
$\sprod{\lambda}{\alpha^\vee}$ from $1+\delta$ to $1-\delta$ we get an integral of the same form but with $M$ replaced by the group $M'$ defined by $\Delta^{M'} = \Delta^M \smallsetminus \{ \alpha \}$. If $\alpha$ lies in $\Phi^+_{w^{-1}}$, then we get in addition a residual term, which is obtained by replacing $M$ by $M'$ and $L$ by the group $L'$ defined by $\Delta^{L'} = \Delta^L  \cup \{ \alpha \}$. 
Continuing this way, we can reduce the semisimple rank of $M$ successively, until we are left with terms with $M=T_0$, i.e., of the form 
\eqref{EquationBruhatResidues}, up to fixed positive constants.

We remark that if $w$ is the longest Weyl element, then the term with $L=G$ is a constant multiple of 
\[
\Phi (F^{\phi,\epsilon} \otimes \tau, \rho) 
= \int_{G(\A)^1} F^{\phi,\epsilon} (g_{\infty}) \tau (g_{\fin}) 
\, dg
\ll_\phi \epsilon^{\frac{r+\abs{\Phi^+}}{2}} \norm{\tau}_1.
\]
Our bound for $I_w^{G,\operatorname{sm}} (f)$ agrees with this up to a power of $\abs{\log \epsilon}$.

Consider now \eqref{EquationBruhatResidues} for a fixed $L$.
Note that $\sprod{\rho - \lambda_{1,L}}{\varpi^\vee} \ge 0$ for 
all $\varpi^\vee \in \hat{\Delta}^\vee$ and that $\lambda_{1,L}$ lies in the positive Weyl chamber, which together implies that
$\lambda_{1,L}$ lies in the convex hull of the $W$-orbit of $\rho$ (even for $0 \le \delta \le 1$).
Therefore, $\Phi (F^{\phi,\epsilon} \otimes \tau, \lambda)$ can 
for $\Re \lambda =
\lambda_{1,L}$ be
estimated by Lemma \ref{LemmaEstimatePhi}.
Let $\Psi_L \subset \Phi_{w^{-1}}^+$ be a subset as in 
Proposition \ref{prop:generating:roots}.  
We bound \eqref{EquationBruhatResidues} by 
\begin{multline} \label{EqnSmoothIntegralAV} 
\epsilon^{\frac{r+\abs{\Phi^+}}{2}} \norm{\tau}_1
\int_{\lambda_{1,L} + i (\af^G_L)^*} 
\abs{\hat\kappa ((1-w^{-1}) \lambda)}
\prod_{\alpha \in \Psi_L} (1 + \abs{\sprod{\lambda}{\alpha^\vee}})^{-1/2} \\
\prod_{\alpha \in \Psi_L \cap \Delta} 
\abs{\zeta (\sprod{\lambda}{\alpha^\vee})}
\min (1, \norm{\Im \lambda}^{-1/2} \epsilon^{-1/4})
\, d \lambda.
\end{multline}
Here, we apply
the estimates of Lemma \ref{LemmaEstimateZeta}
to the factors $m_0 (\sprod{\lambda}{\alpha^\vee})$,
$\alpha \in \Psi_L$,
and note that 
we can drop the factors $m_0 (\sprod{\lambda}{\alpha^\vee})$ for
$\alpha \notin \Psi_L$, since they are in any case bounded. 

Writing $\Im \lambda = \lambda_1 + \lambda_2$ according to the decomposition \eqref{eq:decomp:winvariants} of $i(\af^G_L)^*$, we can rewrite 
\eqref{EqnSmoothIntegralAV} as the product of $\epsilon^{\frac{r+\abs{\Phi^+}}{2}} \norm{\tau}_1$ with
\begin{multline*} 
\int_{(\af^G_L)^{*,w\perp}} 
\abs{\hat\kappa ((1-w^{-1}) \lambda_{1,L} + i (1-w^{-1}) \lambda_2)} 
\int_{(\af^G_L)^{*,w}} 
\prod_{\alpha \in \Psi_L} (1 + \abs{\sprod{\lambda_1}{\alpha^\vee}
+ \sprod{\lambda_2}{\alpha^\vee}})^{-1/2} \\
\cdot\prod_{\alpha \in \Psi_L \cap \Delta} 
\abs{\zeta (1- \delta + i \sprod{\lambda_1}{\alpha^\vee}
+ i \sprod{\lambda_2}{\alpha^\vee})}
\min (1, \norm{\lambda_1 + \lambda_2}^{-1/2} \epsilon^{-1/4})
\, d \lambda_1 d \lambda_2.
\end{multline*}
Let $d_1=\dim (\af_L^G)^{*,w}$ and $d_2=\dim (\af_L^G)^{*,w\perp}$. 
We split the integration domain into the parts
$\norm{\Im \lambda} \le\epsilon^{-1/2}$ and $\norm{\Im \lambda} > \epsilon^{-1/2}$.
In the integral over $\norm{\Im \lambda} \le\epsilon^{-1/2}$,
we use Lemma \ref{LogarithmicIntegralZeta} in each coordinate to bound 
the $\lambda_1$-integral by a constant multiple of $\abs{\log \epsilon}^{d_1}$ uniformly in $\lambda_2$. 
Since the integral of $\abs{\hat\kappa ((1-w^{-1}) \lambda_{1,L} + i (1-w^{-1}) \lambda_2)}$ over the entire space $(\af^G_L)^{*,w\perp}$ is finite, we obtain the estimate of a constant multiple of $\abs{\log \epsilon}^{d_1}$
for the first integral.

The integral over a shell
$2^k \epsilon^{-1/2} \le 
\norm{\Im \lambda} \le 2^{k+1} \epsilon^{-1/2}$, $k \ge 0$,
is by the same argument bounded by a constant multiple of
\[
2^{-k/2} (\abs{\log \epsilon} + k)^{d_1},
\]
and summing over $k$ gives the bound $\abs{\log \epsilon}^{d_1}$.
\end{proof}

\begin{proof}[Proof of Proposition \ref{PropositionWeylNeqOne}]
By the second estimate of Lemma \ref{LemmaDescent}, we 
bound 
\begin{align*}
\CmJ_w^{G,T} (F^{\phi,\epsilon;G} \otimes \tau) & \le  
v_L (T_L) \epsilon^{\frac{r-r_{L(w)}-\abs{\Phi^{L,+}}}{2}} 
I_w^{L(w)}
(F^{\tilde\psi,\epsilon;L(w)} \otimes \tau^{(Q(w))}) \\
& \le  
v_L (T_L) \epsilon^{\frac{r-r_{L(w)}-\abs{\Phi^{L,+}}}{2}} 
I_w^{L(w),\operatorname{sm}}
(F^{\tilde\psi,\epsilon;L(w)} \otimes \tau^{(Q(w))}).
\end{align*}
To estimate the integral $I_w^{L(w),\operatorname{sm}}
(F^{\tilde\psi,\epsilon;L(w)} \otimes \tau^{(Q(w))})$,
we can apply Proposition \ref{PropEstimateIw}, since $w$ is regular in $L(w)$. 
\end{proof}

\subsection{Estimation for the Borel subgroup} \label{SubsectionBorel}
We now consider the estimation of $\CmJ_{1, \operatorname{nc}}^{G,T}(f)$, i.e., we prove Proposition \ref{PropositionUnipotent}.
We proceed by induction over standard Levi subgroups, as in the proof of Lemma \ref{LemmaDescent} above.
Let $\chi$ be a compactly supported bounded measurable function on $A^G_0$.
We need the variant
\begin{multline*}
\CmJ_{1, \operatorname{nc}}^{G}(f; \chi)   = 
\int_{\cpt} \int_{U_0(\Q)\backslash U_0(\A)} \int_{A^G_0} \int_{T_0(\Q)\backslash T_0(\A)^1}  \sum_{\gamma\in P_0(\Q) \smallsetminus Z(\Q)} f ((uatk)^{-1}\gamma uatk) \\
 \chi (a) \delta_0 (a)^{-1} \, dt \, da \, du \, dk
\end{multline*}
of $\CmJ_{1, \operatorname{nc}}^{G,T}(f)$. For any standard parabolic
subgroup $P=MU$ of $G$ we define 
$\CmJ_{P_0 \smallsetminus Z U}^{G}(f; \chi)$ and $\CmJ_{Z U \smallsetminus Z}^{G}(f; \chi)$ by restricting the sum over $\gamma$ to the sets
$P_0(\Q) \smallsetminus Z(\Q) U(\Q)$ and $Z(\Q) U (\Q) \smallsetminus Z(\Q)$, respectively. Obviously,
\begin{equation} \label{decompositionj1nc}
\CmJ_{1, \operatorname{nc}}^{G}(f; \chi)  = 
\CmJ_{P_0 \smallsetminus Z U}^{G}(f; \chi) + \CmJ_{Z U \smallsetminus Z}^{G}(f; \chi).
\end{equation}
The induction step in the proof of Proposition 
\ref{PropositionUnipotent} is given by the following lemma:

\begin{lemma} \label{InductionStepUnipotent}
Let $P =MU$ be an arbitrary standard maximal parabolic subgroup of $G$.
For all $0 < \epsilon < 1/2$ we have
\begin{multline} \label{EquationInductionStepUnipotent}
\CmJ_{P_0 \smallsetminus Z U}^{G}(F^{\phi,\epsilon} \otimes \tau; \chi) 
\le \epsilon^{1/2}
\CmJ_{1, \operatorname{nc}}^{M}(F^{\psi,\epsilon;M} \otimes \tau^{(P)}; \int_{A_M^G} \chi (\cdot b) \, db) \\ 
 + \epsilon^{r/2} \norm{\tau^{(P)}}_1
\int_{A_0^G} \chi (a) \delta_0^M (a)^{-1} \, da
\end{multline}
for a suitable function $\psi \in C^\infty_c (\R^{\ge 0})$ depending on $\phi$.
\end{lemma}

\begin{proof} The proof is very similar to the induction step in the proof
of Lemma \ref{LemmaDescent}. In fact, writing $\gamma = \gamma_M \nu$
with $\gamma_M \in P_0^M (\Q) \smallsetminus Z_G (\Q)$ and $\nu \in U (\Q)$
in the definition of $\CmJ_{P_0 \smallsetminus Z U}^{G}(f; \chi)$, 
we can split into the cases $\gamma_M \in P_0^M (\Q) \smallsetminus Z_M (\Q)$
and $\gamma_M \in Z_M (\Q) \smallsetminus Z_G (\Q)$.
The first case can be treated in exactly the same way as before, and
the corresponding contribution is estimated by the first term in \eqref{EquationInductionStepUnipotent}.
It remains to consider the second case $\gamma_M \in Z_M (\Q) \smallsetminus Z_G (\Q)$. In this case, the endomorphism $1 - \Ad (\gamma_M)^{-1}$ of the vector space $\mathfrak{u}^{\operatorname{ab}}$ has full rank, and the dimension of this space is 
at least $r$. Applying Lemma \ref{LemmaLatticePoints}, and estimating the
result trivially, we obtain that the contribution of 
the sum over $Z_M (\Q) \smallsetminus Z_G (\Q)$ is bounded by the second term in
\eqref{EquationInductionStepUnipotent}.
\end{proof}

The base of the induction is an estimate for $\CmJ_{Z U \smallsetminus Z}^{G}(F^{\phi,\epsilon} \otimes \tau; \chi)$ for at least one standard maximal parabolic subgroup $P=MU$ of $G$ in the case where 
$G^{\operatorname{der}}$ is simple. We consider the standard parabolic subgroups $P$ such that $\dim U$ is minimal. Depending on $G$, there are two cases. The first case is that there exists such a $P$ for which $U$ is abelian. This will be the case for the root systems of type $A_n$, $B_n$, $D_n$, $E_6$ and $E_7$ (cf. \cite{MR1189494}).

The second case is that the unipotent radicals of minimal dimension are all non-abelian. In this case we take $P$ such that $U$ is a Heisenberg group, i.e., such that its commutator subgroup $[U,U]$ is one-dimensional. This will be the case for the root systems of type $C_n$ ($n \ge 3$), $G_2$, $F_4$ and $E_8$.

We remark that if $G^{\operatorname{der}}$ is simple with root system not of type $A_n$, then a unique standard maximal parabolic subgroup $P$ has a unipotent radical $U$ which is a Heisenberg group. $P$ is the standard parabolic subgroup with weight 
$\tilde{\alpha}^\vee$, where $\tilde{\alpha}$ is the highest root. The commutator subgroup $[U,U]$ of the unipotent radical is the root subgroup of the highest root (cf. \cite{MR1247502}, where this standard parabolic subgroup is called extraspecial).

We summarize the choice of $P$ for each irreducible root system in
Table \ref{tab:parabolics} (also included are the values of $d_{\min}$, which is half the minimal dimension of a non-trivial geometric unipotent orbit, cf. Lemma \ref{rootlemma} and the discussion before).

\begin{table}
  \centering
  \begin{tabular}{ | l | l | l | l | l |}
    \hline
    Root system & $d_{\operatorname{min}}$ & $R = \min \dim U$ & Simple root defining $P$ & $U_P$ abelian \\ \hline
    $A_n$ & $n$ & $n$ & $\alpha_1$ or $\alpha_n$ & yes \\ \hline
    $B_n$ & $2n-2$ & $2n-1$ & $\alpha_1$ & yes \\ \hline
    $C_n$, $n \ge 3$ & $n$ & $2n-1$ & $\alpha_1$ & no \\ \hline
    $D_n$, $n\ge4$ & $2n-3$ & $2n-2$ & $\alpha_1$ (or $\alpha_3$ or $\alpha_4$
    for $n=4$) & yes \\ \hline
    $E_6$ & $11$ & $16$ & $\alpha_1$ or $\alpha_6$ & yes \\ \hline
    $E_7$ & $17$ & $27$ & $\alpha_7$ & yes \\ \hline
    $E_8$ & $29$ & $57$ & $\alpha_8$ & no \\ \hline
    $F_4$ & $8$ & $15$ & $\alpha_1$ & no \\ \hline
    $G_2$ & $3$ & $5$ & $\alpha_2$ & no \\ 
    \hline
  \end{tabular}
  
  \caption{Standard parabolic subgroups used in 
  Lemma \ref{LemmaMaximalParabolicAbelian} and Lemma \ref{LemmaMaximalParabolicHeisenberg} and dimensions of minimal unipotent orbits (cf. Lemma \ref{rootlemma} and the discussion before). The numbering of the simple roots follows Bourbaki.}
  \label{tab:parabolics}
\end{table}

\begin{lemma} \label{LemmaMaximalParabolicAbelian}
Let $G^{\operatorname{der}}$ be simple
and $P=MU$ be a standard maximal parabolic subgroup with 
$U$ of minimal dimension. Assume that $U$ is abelian.
Then for all $0 < \epsilon < 1/2$ 
we have
\[
\CmJ_{Z U \smallsetminus Z}^{G}(F^{\phi,\epsilon} \otimes \tau; \chi) 
\ll_h
 \epsilon^{r/2} \norm{\tau^{(P_0)}}_{1,\operatorname{div}}
 \int_{A^G_0} \chi (a) \, da. 
\]
\end{lemma}

\begin{lemma} \label{LemmaMaximalParabolicHeisenberg}
Let $G^{\operatorname{der}}$ be simple
and $P=MU$ be a standard maximal parabolic subgroup with 
$U$ of minimal dimension. Assume that $U$ is a Heisenberg group.
Then for all $0 < \epsilon < 1/2$ 
we have
\[
\CmJ_{Z U \smallsetminus Z}^{G}(F^{\phi,\epsilon} \otimes \tau; \chi) 
\ll_h
 \epsilon^{r/2} \norm{\tau^{(P_0)}}_{1,\operatorname{div}}
 \int_{A^G_0} \chi (a) \, da. 
\]
\end{lemma}

Before we prove Lemma \ref{LemmaMaximalParabolicAbelian} and
Lemma \ref{LemmaMaximalParabolicHeisenberg}, we finish the proof of
Proposition \ref{PropositionUnipotent}.

\begin{proof}[Proof of Proposition \ref{PropositionUnipotent}]
We prove the statement
\begin{equation} \label{equationj1nc}
\CmJ_{1,\operatorname{nc}}^{M}(F^{\phi,\epsilon} \otimes \tau; \chi) 
\ll_\phi
 \epsilon^{r_M/2} \norm{\tau^{(P^M_0)}}_{1,\operatorname{div}}
 \int_{A^M_0} \chi (a) \, da 
\end{equation}
for all Levi subgroups $M \in \mathcal{L}$ by induction over the semisimple rank of $M$. The proposition follows then by taking $\chi = \chi_T$.

The case of the torus (rank zero) is trivial. Let now $G$ denote a fixed group in $\mathcal{L}$, and assume that \eqref{equationj1nc} has been already established for all of its Levi subgroups $M$.
If $G^{\operatorname{der}}$ is not simple, then the intersection of the unipotent radicals of the standard maximal parabolic subgroups of $G$ is trivial. In this case we can by Lemma \ref{InductionStepUnipotent} 
estimate $\CmJ_{1, \operatorname{nc}}^{G}(F^{\phi,\epsilon} \otimes \tau; \chi)$ by the sum of
\[
\epsilon^{1/2}
\CmJ_{1, \operatorname{nc}}^{M}(F^{\psi,\epsilon;M} \otimes \tau^{(P)}; \int_{A_M^G} \chi (\cdot b) \, db) 
+ \epsilon^{r/2} \norm{\tau^{(P)}}_1
\int_{A_0^G} \chi (a) \delta_0^M (a)^{-1} \, da 
\]
over the standard maximal parabolic subgroups $P=MU$ of $G$, which gives the assertion by applying the induction hypothesis.

If $G^{\operatorname{der}}$ is simple, then we combine
\eqref{decompositionj1nc}, Lemma \ref{InductionStepUnipotent}, 
Lemma \ref{LemmaMaximalParabolicAbelian} and Lemma
\ref{LemmaMaximalParabolicHeisenberg} to show \eqref{equationj1nc} for $G$.
\end{proof}

\subsection{Sums over special unipotent radicals} \label{SubsectionUnipotent}
We use a combination of two strategies to estimate $\CmJ_{Z U \smallsetminus Z}^{G}(F^{\phi,\epsilon} \otimes \tau; \chi)$ in the cases
of Lemma \ref{LemmaMaximalParabolicAbelian} and Lemma
\ref{LemmaMaximalParabolicHeisenberg}. 
First, we deal with sums over unipotent elements that are supported on sets of root spaces spanning $(\af_0^G)^*$ (or at least a subspace of sufficiently large dimension). Second, we consider subsums supported on small collections of root spaces using the strategy of \cite{FLM14}*{\S 3.4}.
The following lemma implements the first part of this strategy.

\begin{lemma} \label{LemmaLatticePointsLargeSupport}
Let $\Psi$ be a subset of the set $\Phi^+$ of all positive roots and 
$\mathfrak{v} = \bigoplus_{\alpha \in \Psi} \mathfrak{u}_{\alpha}$ be 
the associated subspace of the Lie algebra of $U_0$. 
For any $v \in \mathfrak{v}$ let $\operatorname{supp} (v)$ be the set of all
$\alpha \in \Psi$ with $v_{\alpha} \neq 0$.

Let $f_0$ be a fixed bounded compactly supported function on $\R^{\ge 0}$ and
$\tau_0$ be a compactly supported function on $\mathfrak{v} \otimes \A_{\fin}$ that is invariant under translation by 
$\mathfrak{v}_\Z \otimes \hat{\Z}$ and under conjugation by elements of
$T_0 (\hat{\Z})$. Let $N$ be a positive integer such that the support of 
$\tau_0$ is contained in $N^{-1} \mathfrak{v}_\Z \otimes \hat{\Z}$. 
Then we have
\begin{equation} \label{EqnLemmaLatticePointsLargeSupport}
\sum_{v \in \mathfrak{v}, \, \dim \langle \operatorname{supp} (v) \rangle\ge k} \left( f_0 (\norm{\cdot}^2 \epsilon^{-1}) \otimes \tau_0 \right) (\Ad (a)^{-1} v)
\ll_{f_0} \epsilon^{k/2} \left[\prod_{\alpha \in \Psi} \alpha (a)\right] C^{\omega (N)}
\int_{\mathfrak{v} \otimes \A_{\fin}} \abs{\tau_0 (v)} \, dv
\end{equation}
for all $0 < \epsilon < 1$ and
$a \in A_0$ with $\alpha (a) \gg 1$ for $\alpha \in \Psi$.
Here $\omega (N)$ is the 
number of distinct prime factors of $N$
and $C$ is a positive constant depending only on $G$.
\end{lemma} 

\begin{proof}
It is evidently enough to fix a linearly independent subset $S \subset \Psi$ 
with $k$ elements and to estimate 
\begin{equation} \label{EqnSumSupportContainsS}
\sum_{v \in \mathfrak{v}, \, \operatorname{supp} (v) \supset S} \left( f_0 (\norm{\cdot}^2 \epsilon^{-1}) \otimes \tau_0 \right) (\Ad (a)^{-1} v)
\end{equation}
by the right-hand side of \eqref{EqnLemmaLatticePointsLargeSupport}.
Write $\mathfrak{v} = \mathfrak{v}_S \oplus \mathfrak{v}^S$ with
$\mathfrak{v}_S = \bigoplus_{\alpha \in S} \mathfrak{u}_{\alpha}$ 
and $\mathfrak{v}^S = \bigoplus_{\alpha \in \Psi \smallsetminus S} \mathfrak{u}_{\alpha}$, and split the summation in \eqref{EqnSumSupportContainsS} accordingly. Since $\tau_0$ 
is invariant under 
translation by 
$\mathfrak{v}_\Z \otimes \hat{\Z}$,
it follows by a simple lattice point counting argument (cf. \cite{FiLa11}*{\S 4}) that
for any $x_S \in \mathfrak{v}_S \otimes \A$ we have
\begin{equation} \label{EqnSumCoset}
\sum_{v^S \in \mathfrak{v}^S} \left( f_0 (\norm{\cdot}^2 \epsilon^{-1}) \otimes \tau_0 \right) (\Ad (a)^{-1} (x_S + v^S))
\ll_{f_0} \left[ \prod_{\alpha \in \Psi \smallsetminus S} \alpha (a) \right] 
\tau_S (x_{S,\fin}),
\end{equation}
where the function $\tau_S$ on $\mathfrak{v}_S \otimes \A_{\fin}$ is
defined by
\[
\tau_S (\xi_S) = 
\int_{\mathfrak{v}^S \otimes \A_{\fin}} \abs{\tau_0 (\xi_S + y^S)} \, dy^S.
\]
Furthermore, the left-hand side of \eqref{EqnSumCoset} can only be non-zero if
$\norm{\Ad (a)^{-1} (x_{S, \infty})} \le c \epsilon^{1/2}$,
where $c > 0$ is a constant depending on the support of $f_0$.
We can therefore estimate \eqref{EqnSumSupportContainsS} by
\begin{equation}
\left[ \prod_{\alpha \in \Psi \smallsetminus S} \alpha (a) \right]
\sum_{\substack{v_S \in \mathfrak{v}_S, \ 
v_\alpha \neq 0 \ \forall \alpha \in S, \\ 
\norm{\Ad (a)^{-1} (v_{S})} \le c \epsilon^{1/2}}} \tau_S (v_S).
\end{equation}

Let us for the remainder of the proof call a function on $\mathfrak{v}_S \otimes \A_{\fin}$ simple,
if it is a linear combination of the characteristic functions of the sets
\[
L_t = \bigoplus_{\alpha \in S} t_\alpha^{-1} \mathfrak{u}_{\alpha,\Z} \otimes \hat{\Z}, \quad t = (t_\alpha)_{\alpha \in S} \in \N^S,
\] 
with non-negative coefficients.
We claim that there exists a simple function $\tilde{\tau}_S$ with
$\tilde{\tau}_S \ge \tau_S$ and 
\[
\int \tilde{\tau}_S (x_S) \, dx_S \le 
C^{\omega (N)}
\int \tau_S (x_S) \, dx_S = C^{\omega (N)} \int_{\mathfrak{v} \otimes \A_{\fin}} \abs{\tau_0 (v)} \, dv,
\]
where $C>0$ depends only on $G$.
Granted this claim, we only need to show that
\begin{equation}  \label{EqnSimpleFunction}
\sum_{\substack{v_S \in \mathfrak{v}_S, \, 
v_\alpha \neq 0 \ \forall \alpha \in S, \\ 
\norm{\Ad (a)^{-1} (v_{S})} \le c \epsilon^{1/2}}} 
\sigma (v_S) \ll 
\epsilon^{k/2} \left[ \prod_{\alpha \in S} \alpha (a) \right] \int \sigma (x_S) \, dx_S
\end{equation}
for simple functions $\sigma$.
We may reduce to the case where 
$\sigma$ is the
characteristic function of a set
$L_t$.
In this case, we can bound the left-hand side of \eqref{EqnSimpleFunction} by
 \[
 \prod_{\alpha \in S}
 \sum_{\substack{v_\alpha \in t_\alpha^{-1} \Z \smallsetminus \{ 0 \}, \\ 
\abs{v_{\alpha}} \ll \alpha (a) \epsilon^{1/2}}} 1 
\ll
\epsilon^{k/2} \prod_{\alpha \in S} \alpha (a) 
\prod_{\alpha \in S} t_\alpha
= \epsilon^{k/2} \vol (L_t) \prod_{\alpha \in S} \alpha (a), 
\]
which finishes the proof.

To establish the claim, 
write $\mathfrak{v}_S \otimes \A_{\fin}$
as the disjoint union of the sets
\[
\Delta_t = L_t \smallsetminus \bigcup_{t' \, | t, \, t' \neq t} L_{t'}
\]
for
$t \in \N^S$. Note that 
\[
\vol (\Delta_t) = \vol (L_t) \prod_{\alpha \in S}
\prod_{p \, | \, t_\alpha} (1-p^{-1}).
\]
Set
$c_t = \max_{x \in 
\Delta_t} \tau_S (x)$
for any $t = (t_\alpha)_{\alpha \in S} \in \N^S$. Note that 
$c_t > 0$ implies $t_\alpha \, | \, N$ for all $\alpha \in S$, by our assumption
on the support of $\tau_0$.
Define the simple function $\tilde{\tau}_S$ by 
\[
\tilde{\tau}_S = \sum_{t} 
c_t \mathbf{1}_{L_t}.
\]
The estimate $\tilde{\tau}_S \ge \tau_S$ is then clear, since 
$\Delta_t \subset L_t$. 
Let $E = (\hat{\Z}^\times)^S$, acting on 
$\mathfrak{v}_S \otimes \A_{\fin}$ coordinatewise. The orbits of
$E$ on the quotient $\mathfrak{v}_S \otimes \A_{\fin} / \mathfrak{v}_{S,\Z} \otimes \hat{\Z}$ are then precisely the projections of the sets $\Delta_t$.
Since $\tau_0$
is invariant under conjugation
by $T_0 (\hat{\Z})$, and the roots in the set $S$ are linearly independent, the integral  
$\tau_S$ is invariant under the subgroup $E^\nu$ of $\nu$-th powers in $E$ for a positive integer $\nu$ depending only on $G$.
Each set $\Delta_t$, where all $t_\alpha$ divide $N$, splits into at most $\nu^{k(1+\omega(N))}$ many
orbits under multiplication by $E^\nu$. Therefore,
\begin{align*}
\int \tau_S (x) \, dx & \ge  \nu^{-k(1+\omega(N))} \sum_{t} c_t \vol (\Delta_t) \\
& \ge  \nu^{-k(1+\omega(N))} \prod_{p \, | \, N} (1-p^{-1})^{k} \sum_{t} c_t \vol (L_t) \\
& \ge  C^{-\omega(N)} \int \tilde{\tau}_S (x) \, dx,
\end{align*}
for a suitable $C$, as claimed.
\end{proof}

Let $d_{\operatorname{min}}$ be half the minimal dimension of a non-trivial 
geometric unipotent orbit of $G$. We have 
$d_{\operatorname{min}} = 
\sprod{\rho}{\tilde{\alpha}^\vee}$, where $\tilde{\alpha}$ is the highest root of $G$. The values of $d_{\operatorname{min}}$ for the irreducible root systems are included in Table \ref{tab:parabolics}.
Note that $d_{\operatorname{min}} \ge r$.
We need the following property of $d_{\operatorname{min}}$
\cite{FLM14}*{Lemma 3.11}. 

\begin{lemma} \label{rootlemma}
For all $\beta \in \Phi^+$ and all subsets $S \subset \Phi^+$ we have
\[
\sum_{\alpha \in \Phi^+ \smallsetminus S} \alpha \in (
d_{\operatorname{min}} - \abs{S} ) \beta + \sum_{\alpha \in \Phi^+} \R^{\ge 0} \alpha.
\]
\end{lemma}

\begin{lemma} \label{LemmaCardinalitySupport}
\begin{enumerate} 
\item
Let $P=MU$ be such that $U$ is abelian of minimal dimension. Then any subset $S$ of $\Phi_U^+$ 
that does not span $(\af_0^G)^*$ has less than $d_{\operatorname{min}}$ many elements.
\item
Let $G$ be of type $G_2$, $F_4$ or $E_8$ and $P=MU$ be the parabolic subgroup of Heisenberg type. Then any subset $S$ of $\Phi^+_U$ spanning a subspace of $(\af_0^G)^*$ of codimension at least two has less than $d_{\operatorname{min}}$ many elements.
\end{enumerate} 
\end{lemma}

\begin{proof} 
One checks by a case-by-case analysis that the maximum cardinality of the set $\Phi^+_U \cap \Phi^{L,+}
= \Phi^{L,+} \smallsetminus \Phi^{L \cap M, +}$, where $L$ ranges over all semistandard maximal Levi subgroups of $G$, is always $d_{\operatorname{min}} - 1$ in the abelian cases. 

The roots in $\Phi^+_U$ span the space $(\mathfrak{a}^G_0)^*$, and a non-spanning subset 
is therefore necessarily proper. This observation already suffices for the $A_n$ case, where
$d_{\operatorname{min}} = \dim U = n$, and the bound $n-1$ for the cardinality of a non-spanning subset is easily seen to be best possible. In the orthogonal cases ($B_n$ and $D_n$), it is easy to see that 
$\Phi^+_U \smallsetminus \{ \alpha \}$ still spans the entire space for any 
$\alpha \in \Phi^+_U$, and that the maximum cardinality of a non-spanning subset is therefore at most $\dim U - 2 = 
d_{\operatorname{min}} - 1$ (which is again best possible).
In the $E_6$ and $E_7$ cases one needs to examine the possibilities for $L$. In the $E_6$ case, the extremal configuration occurs when $L$ is of type $D_5$ and intersects $M$, which is also of type $D_5$, in a group of type $A_4$, yielding a non-spanning subset with $10$ elements. In the $E_7$ case, two groups of type $E_6$ intersect in a group of type $D_5$, yielding a non-spanning subset with $16$ elements.

Consider the remaining three exceptional cases, where the parabolic subgroups in question are of Heisenberg type. The $G_2$ case is trivial. In the $F_4$ case, a subspace of dimension two contains at most $4$ positive roots, while $d_{\operatorname{min}}=8$.
In the $E_8$ case, we have $d_{\operatorname{min}}=29$, while the cardinality of $\Phi^+_U \cap \Phi^{L,+}
= \Phi^{L,+} \smallsetminus \Phi^{L \cap M, +}$, where $L$ is a semistandard Levi subgroup of semisimple rank $6$, is at most $21$. The maximum is achieved for a group $L$ of type $E_6$ that intersects $M$, which is of type $E_7$, in a group of type $A_5$.
\end{proof} 

\begin{lemma} \label{LemmaCartanAbelian}
\begin{enumerate} 
\item
Let $P=MU$ be a standard parabolic subgroup of $G$ such that $U$ is abelian. 
Let $\xi \in p^{-n} \mathfrak{u}_\Z \otimes \Z_p$ for an integer $n \ge 0$, and let $\exp (\xi) = k_1 t k_2$ with $k_1$, $k_2 \in \K_p$, $t \in T_0 (\Q_p)$, be the Cartan decomposition of $\exp (\xi)$, where $\lambda = (\log p)^{-1} H_{0,p} (t) \in \af_0^G$ is dominant. Then
$2 n  \le \sprod{\tilde{\alpha}}{\lambda}$.
\item
Let $P=MU$ be a standard parabolic subgroup such that $U$ is a Heisenberg
group. 
Let $\xi_1 \in p^{-n_1} \mathfrak{u}^{\operatorname{ab}}_\Z \otimes \Z_p$, $n_1 \ge 0$,
$\xi_2 \in [\mathfrak{u}, \mathfrak{u}] \otimes \Q_p$,
 and let $\exp (\xi_1 + \xi_2) = k_1 t k_2$ with $k_1$, $k_2 \in \K_p$, $t \in T_0 (\Q_p)$, be the Cartan decomposition of $\exp (\xi_1 + \xi_2)$, where $\lambda = (\log p)^{-1} H_{0,p} (t) \in \af_0^G$ is dominant. Then
$2 n_1 \le \sprod{\tilde{\alpha}}{\lambda}$.
\item
In the $F_4$ case, let
$n_2 = \sprod{\varpi_4}{\lambda}$, where $\varpi_4 = \tilde{\beta}$ is the highest short root (equivalently, the fourth fundamental weight in the Bourbaki numbering, associated to the standard Levi subgroup of type $B_3$).
Then for a fixed $\xi_1 \in \mathfrak{u}^{\operatorname{ab}} \otimes \Q_p$, the set of all possible $\xi_2
\in [\mathfrak{u}, \mathfrak{u}] \otimes \Q_p$ with
$\exp (\xi_1 + \xi_2)  \in \K_p t  \K_p$ is contained in a coset
$\eta + p^{-n_2} [\mathfrak{u}_\Z, \mathfrak{u}_\Z] \otimes \Z_p$. 
\item
In the $E_8$ case,
let $n_2 = \lfloor \frac12 \sprod{\varpi_1}{\lambda} \rfloor$ and
$n'_2 = \lfloor \frac12 (\sprod{\varpi_1}{\lambda} - 1) \rfloor$, where $\varpi_1$ is the first fundamental weight in the Bourbaki numbering (associated to the standard Levi subgroup of type $D_7$).
Then
for a fixed $\xi_1 \in \mathfrak{u}^{\operatorname{ab}} \otimes \Q_p$, the set of all possible 
$\xi_2
\in [\mathfrak{u}, \mathfrak{u}] \otimes \Q_p$ with
$\exp (\xi_1 + \xi_2)  \in \K_p t  \K_p$ is either contained in a coset
$\eta + p^{-n_2} [\mathfrak{u}_\Z, \mathfrak{u}_\Z] \otimes \Z_p$, or in 
a union of two cosets
$\eta_i + p^{-n'_2} [\mathfrak{u}_\Z, \mathfrak{u}_\Z] \otimes \Z_p$, $i = 1$, $2$ (of course, the second case is only possible for $\lambda \neq 0$).
\item
In the $C_n$ case, we have the inequality
$2 n_1 \le \sprod{\varpi_2}{\lambda}$, where $\varpi_2 = \tilde{\beta}$ is the highest short root, and moreover
$\xi_2 \in p^{-n_2} [\mathfrak{u}_\Z, \mathfrak{u}_\Z] \otimes \Z_p$ 
for $n_2 = \frac12 \sprod{\tilde{\alpha}}{\lambda} = \sprod{\varpi_1}{\lambda}$.
\end{enumerate} 
\end{lemma}

We note that the first part of the lemma, which concerns the abelian case, was essentially already obtained in \cite{MR3001800}*{Corollary 29} (using a variant of 
the normal form of
Richardson--R\"{o}hrle--Steinberg \cite{MR1189494} for the $M$-orbits on $U$). 

\begin{proof} We will treat the first two assertions together. Note that 
the lattice $\gf_\Z$ in $\gf$ considered in \S \ref{sec:notation} induces a norm 
on $\gf \otimes \Q_p$ that is invariant under $\Ad (\K_p)$. We denote the corresponding operator norm on $\operatorname{End} (\gf \otimes \Q_p)$ by $\norm{\cdot}$.
Let $P$ be a standard parabolic subgroup with unipotent radical $U$ either abelian or of Heisenberg type. Let $\xi \in \mathfrak{u} \otimes \Q_p$ and
write $\xi = \xi_1 + \xi_2$ with $\xi_1\in \mathfrak{u}^{\operatorname{ab}} \otimes \Q_p$ and
$\xi_2 \in [\mathfrak{u}, \mathfrak{u}] \otimes \Q_p$. Introduce the root coordinates
$\xi_{\alpha}$, $\alpha \in \Phi^+_{U^{\operatorname{ab}}}$, of
$\xi_1$ with respect to the Chevalley system $(X_{\alpha})$ of \S \ref{sec:notation}, and consider the coefficient
$c_{-\alpha,\alpha}$ of $X_{\alpha}$ in 
$\Ad (\exp \xi ) X_{-\alpha}$ for a root 
$\alpha \in \Phi^+_{U^{\operatorname{ab}}}$. 
The coefficient $c_{-\alpha,\alpha}$ is the sum of
$\xi_{\alpha}^2$ and of integer multiples of the products $\xi_{\beta} \xi_{\gamma}$, where
$\beta$, $\gamma \in \Phi^+_{U^{\operatorname{ab}}}$ with
$\beta + \gamma = 2\alpha$. 
In the simply laced case, this equation has no solutions $\beta$ and $\gamma$, and we obtain
\begin{equation} \label{EstimateAdjointRepresentation}
p^{\sprod{\tilde{\alpha}}{\lambda}} = \norm{\Ad (\exp \xi)}
\ge \abs{\xi_\alpha}^2
\end{equation}
for all $\alpha \in \Phi^+_{U^{\operatorname{ab}}}$,
as claimed. In the non-simply laced case, the equation $\beta + \gamma = 2\alpha$
implies that $\alpha$ is short and that at least one of $\beta$ and $\gamma$
is long (in fact, both have to be long if we are not in the $G_2$ case).
Therefore, we obtain the inequality \eqref{EstimateAdjointRepresentation} for all long roots $\alpha \in \Phi^+_{U^{\operatorname{ab}}}$. If the maximum of $\abs{\xi_\alpha}$ is assumed for a long root, we are done. Assume otherwise, and
let $\alpha_0$ be the short root for which $\abs{\xi_{\alpha}}$ is maximal.
By assumption,
$\abs{\xi_{\alpha_0}} > \abs{\xi_\beta}$ for any long root $\beta$. This implies that $\abs{\xi_\beta \xi_\gamma} < \abs{\xi_{\alpha_0}}^2$, if
$\beta + \gamma = 2\alpha_0$, since at least one of 
$\beta$ and $\gamma$ is long. We conclude that also in this case 
$\abs{\xi_{\alpha}}^2 \le 
\abs{\xi_{\alpha_0}}^2 \le p^{\sprod{\tilde{\alpha}}{\lambda}}$ 
for all $\alpha \in \Phi^+_{U^{\operatorname{ab}}}$, as asserted.

In the $F_4$ case, consider the fourth fundamental representation $\rho_4$
of $G$, together with a lattice defining a $\K_p$-invariant norm for every $p$. The weights of $\rho_4$ are the short roots (with multiplicity one) and zero (with multiplicity two). 
It is easy to see that the operator $\rho_4 ( \exp(\xi_1+\xi_2) )$
is an affine linear function of the root coordinate $\xi_{\tilde{\alpha}}$ of $\xi_2$ with leading coefficient independent of $\xi_1$. 
Moreover, for $\xi_1=0$
the norm of $\rho_4 ( \exp \xi_2 )$ is equal to $p^{n_2 \sprod{\varpi_4}{\tilde{\alpha}^\vee}} = p^{n_2}$. This easily yields the assertion.

In the $E_8$ case, we consider similarly the first fundamental
representation $\rho_1$ of $G$.
The operator $\rho_1 ( \exp(\xi_1+\xi_2) )$ is a
quadratic polynomial in 
$\xi_{\tilde{\alpha}}$ with leading coefficient independent of
$\xi_1$. 
For $\xi_1=0$
the norm of $\rho_1 ( \exp \xi_2 )$ is equal to $p^{n_2 \sprod{\varpi_1}{\tilde{\alpha}^\vee}} = p^{2n_2}$.
For fixed $\xi_1$, we obtain the 
condition $\abs{\xi_{\tilde{\alpha}}^2 + x \xi_{\tilde{\alpha}} + y} \le p^{\sprod{\varpi_1}{\lambda}}$ for some $x$, $y \in \Q_p$, which is easily seen to imply the assertion.

The assertions in the $C_n$ case can easily be verified by explicit calculation. Note that without loss of generality we may assume $G^{\dgp}$ to be
simply connected, i.e., isomorphic to the symplectic group $\operatorname{Sp} (2n)$.
\end{proof}

Before we turn to the proof of Lemma \ref{LemmaMaximalParabolicAbelian} and Lemma \ref{LemmaMaximalParabolicHeisenberg}, we record an easy 
fact on the $L^1$-norms $\norm{\tau^{(P_0)}}_1$ that will be used repeatedly: let $\tau_m$ be the characteristic function
of the double coset $\K_{\fin} m \K_{\fin}$, where
$m \in T_0 (\A_{\fin})$, and let 
$\lambda_p = (\log p)^{-1} H_{0,p} (m_p) \in \af_0$ for every prime $p$. Then
\begin{equation} \label{InequalityTauP0}
e^{\sprod{\rho}{H_0 (m)}} = \prod_p p^{\sprod{\rho}{\lambda_p}}
\le \norm{\tau_m^{(P_0)}}_1.
\end{equation}

\begin{proof}[Proof of Lemma \ref{LemmaMaximalParabolicAbelian}]
Let $P=MU$ with $U$ of minimal dimension and abelian. We have
\begin{multline*}
\CmJ_{ZU \smallsetminus Z}^{G}(F^{\phi,\epsilon} \otimes \tau; \chi)   = 
\int_{U_0(\Q)\backslash U_0(\A)} \int_{A^G_0} \int_{T_0(\Q)\backslash T_0(\A)^1}  \sum_{\zeta\in Z_G (\Q)} \\
  \sum_{\nu \in U(\Q), \, \nu \neq 1} 
(F^{\phi,\epsilon} \otimes \tau) (\zeta (uat)^{-1}\nu uat) 
\chi (a) \delta_0 (a)^{-1} \, dt \, da \, du.
\end{multline*}
Without loss of generality let
$\tau$ be the characteristic function of a double coset 
$\K_{\fin} m \K_{\fin}$, where $m \in T_0 (\A_{\fin})$.
Fix $\zeta \in Z_G (\Q)$ and assume that
$\K_{\fin} \zeta^{-1} m \K_{\fin}$ intersects $U (\A_{\fin})$. The number of $\zeta$ with this property is bounded in terms of $G$ only.
By passing to the Lie algebra, we can write the sum over $\nu$ as
\begin{multline} \label{splitsumabelian}
\sum_{v \in \mathfrak{u}, \, v \ne 0} (F^{\phi,\epsilon} \otimes \tau) (\zeta \exp (\Ad (a \kappa)^{-1} v))
 = 
\sum_{v \in \mathfrak{u}, \, \dim \langle \operatorname{supp} (v) \rangle\ge r} 
(F^{\phi,\epsilon} \otimes \tau) (\zeta 
\exp (\Ad (a \kappa)^{-1} v)) \\
\mbox{} +
\sum_{v \in \mathfrak{u}, \, 1 \le \dim \langle \operatorname{supp} (v) \rangle\le r-1} 
(F^{\phi,\epsilon} \otimes \tau) (\zeta 
\exp (\Ad (a \kappa)^{-1} v)),
\end{multline}
where $\kappa = a^{-1} u a t$ lies in a compact subset of $P_0 (\A)$.

The first subsum here is treated by applying Lemma \ref{LemmaLatticePointsLargeSupport}, which yields the bound
\[
\ll_\phi \delta_P (a) \epsilon^{r/2}  C^{\omega (N)} \norm{\tau^{(P_0)}}_1,
\]
where $N = \prod_p p^{n_p}$ is the smallest integer such that the intersection of
$\K_{\fin} \zeta^{-1} m \K_{\fin}$ with $U (\A_{\fin})$ is contained in $\exp (N^{-1} \mathfrak{u}_\Z \otimes \hat\Z)$.
For the second subsum, one splits it further according to 
the possible subsets $S = \operatorname{supp} (v) \subset \Phi^+_U$, which by Lemma 
\ref{LemmaCardinalitySupport} have at most 
$d_{\operatorname{min}}-1$ many elements. Since each root coordinate of $v \in \mathfrak{u}$ with 
$\tau (\zeta 
\exp (\Ad (\kappa)^{-1} v)) \neq 0$ lies in $N^{-1} \Z$, we get 
\[
\sum_{v \in \mathfrak{u}, \, \operatorname{supp} (v) = S} 
(F^{\phi,\epsilon} \otimes \tau) (\zeta 
\exp (\Ad (a \kappa)^{-1} v))
\ll_\phi \delta_0 (a) \prod_{\alpha \notin S} \alpha (a)^{-1} 
\epsilon^{\abs{S}/2}  
N^{\abs{S}}.
\]
Furthermore, the sum can be only non-zero if $\beta(a)^{-1} \ll \epsilon^{1/2} N$ for all $\beta \in S$.
Using Lemma~\ref{rootlemma} with arbitrary $\beta \in S$, we can 
therefore estimate 
\[
\prod_{\alpha \notin S} \alpha (a)^{-1} \ll \beta(a)^{\abs{S}-d_{\operatorname{min}}} 
\ll (\epsilon^{1/2} N)^{d_{\operatorname{min}}-\abs{S}}
\]
and obtain for the second subsum a bound of 
\[
\ll_\phi \delta_0 (a) \epsilon^{d_{\operatorname{min}}/2}  
N^{d_{\operatorname{min}}}.
\]

For a prime $p$ write  
$\lambda_p = (\log p)^{-1} H_{0,p} (\zeta^{-1} m_p) \in \af_0^G$. By
Weyl group conjugation, we can assume that each $\lambda_p$ is dominant.
By Lemma \ref{LemmaCartanAbelian}, the exponent $n_p$ of $p$ in $N$ is bounded from above by $\frac12 \sprod{\tilde{\alpha}}{\lambda_p}$.
Since by Lemma \ref{rootlemma} the difference
$2 \rho - d_{\operatorname{min}} \tilde{\alpha}$ is a non-negative linear combination of positive roots, using 
\eqref{InequalityTauP0} we see that 
\[
N^{d_{\operatorname{min}}}
\ll \prod_p p^{\frac{d_{\operatorname{min}}}{2} \sprod{\tilde{\alpha}}{\lambda_p}}
\le \prod_p p^{\sprod{\rho}{\lambda_p}}
\le \norm{\tau^{(P_0)}}_1.
\]
The second subsum in \eqref{splitsumabelian} is therefore bounded by
\[
\ll_\phi \delta_0 (a) \epsilon^{d_{\operatorname{min}}/2}  
\norm{\tau^{(P_0)}}_1.
\]
Putting the two results together, we obtain
\[
\delta_0 (a)^{-1}
\sum_{\zeta \in Z_G(\Q)} 
\sum_{\nu \in U(\Q), \, \nu \neq 1} 
(F^{\phi,\epsilon} \otimes \tau) (\zeta (uat)^{-1}\nu uat) 
\ll_\phi 
\left( \delta^M (a)^{-1} \epsilon^{r/2}  C^{\omega (N)} 
+ \epsilon^{d_{\operatorname{min}}/2} \right)
\norm{\tau^{(P_0)}}_1,
\]
which finishes the estimation of $\CmJ_{Z U \smallsetminus Z}^{G}(F^{\phi,\epsilon} \otimes \tau; \chi)$ in the abelian case.
\end{proof}

\begin{proof}[Proof of Lemma \ref{LemmaMaximalParabolicHeisenberg}]
Let $P=MU$ with $U$ of minimal dimension. Assume that the root system $\Phi$ is of type $G_2$, $F_4$, $E_8$ or $C_n$ ($n \ge 3$) and that $U$ is a Heisenberg group. We have
\begin{align*}
\CmJ_{ZU \smallsetminus Z}^{G}(F^{\phi,\epsilon} \otimes \tau; \chi)   = &
\int_{U^M_0(\Q)\backslash U^M_0(\A)} \int_{A^G_0} \int_{T_0(\Q)\backslash T_0(\A)^1}  \sum_{\zeta\in Z_G (\Q)} \\
&  
\left[ \int_{U (\Q)\backslash U (\A)}
\sum_{\nu \in U(\Q), \, \nu \neq 1} 
(F^{\phi,\epsilon} \otimes \tau) (\zeta (u^M at)^{-1} u^{-1} \nu u (u^M at)) 
\, du \right] \\
& \quad\quad\quad\quad\quad\quad\quad\quad\quad\quad\quad\quad\quad\quad\quad\quad\quad \chi (a) \delta_0 (a)^{-1} \, dt \, da \, du^M.
\end{align*}
We consider
\begin{equation} \label{InnerSumIntegralHeisenberg}
\int_{U (\Q)\backslash U (\A)}
\sum_{\nu \in U(\Q), \, \nu \neq 1} 
(F^{\phi,\epsilon} \otimes \tau) (\zeta (a \kappa)^{-1} u^{-1} \nu u (a \kappa)) 
\, du,
\end{equation}
where $\kappa = a^{-1} u^M a t$ lies in a compact subset of $P_0^M (\A)$.
We can split \eqref{InnerSumIntegralHeisenberg}
into the two parts
\begin{equation} \label{SumCenterHeisenberg}
\sum_{\nu \in U_{\tilde{\alpha}} (\Q), \, \nu \neq 1} 
(F^{\phi,\epsilon} \otimes \tau) (\zeta (a \kappa)^{-1} \nu (a \kappa)) 
\end{equation}
and
\begin{equation} \label{SumNonCenterHeisenberg}
\sum_{\nu \in U (\Q) / U_{\tilde{\alpha}} (\Q), \, \nu \neq 1} 
\int_{U_{\tilde{\alpha}} (\A)}
(F^{\phi,\epsilon} \otimes \tau) (\zeta (a \kappa)^{-1} \nu \tilde{u} (a \kappa)) 
\, d\tilde{u}.
\end{equation}
Assume again that
$\tau$ is the characteristic function of a double coset 
$\K_{\fin} m \K_{\fin}$ and that
$\K_{\fin} \zeta^{-1} m \K_{\fin}$ intersects $U (\A_{\fin})$
for some $\zeta \in Z_G (\Q)$. 
For each prime $p$ assume that
$\lambda_p = (\log p)^{-1} H_{0,p} (\zeta^{-1} m_p) \in \af_0^G$ 
is dominant.

Clearly, \eqref{SumCenterHeisenberg} can only be non-zero if 
$\K_{\fin} \zeta^{-1} m \K_{\fin}$ intersects $U_{\tilde{\alpha}} (\A_{\fin})$, and in this case we have $\lambda_p = \tilde{n}_p \tilde{\alpha}^\vee$ for each prime $p$, where
$\tilde{N} = \prod_p p^{\tilde{n}_p}$ the smallest integer with
\[
\K_{\fin} \zeta^{-1} m \K_{\fin} \cap U_{\tilde{\alpha}} (\A_{\fin})
\subset
\exp (\tilde{N}^{-1} \mathfrak{u}_{\tilde{\alpha},\Z} \otimes \hat\Z).
\]
Therefore, \eqref{SumCenterHeisenberg}
is bounded by
\[
\ll_\phi \delta_0 (a) \left[ \prod_{\alpha \neq \tilde{\alpha}} \alpha (a)
\right]^{-1} 
\epsilon^{1/2}  
\tilde{N}
\ll \delta_0 (a) \epsilon^{d_{\operatorname{min}}/2}  
\tilde{N}^{d_{\operatorname{min}}}.
\]
Since by \eqref{InequalityTauP0} we have
\[
\tilde{N}^{d_{\operatorname{min}}} = 
\prod_p p^{\tilde{n}_p d_{\operatorname{min}}} 
= \prod_p p^{\sprod{\rho}{\lambda_p}}
\le \norm{\tau^{(P_0)}}_1,
\]
we obtain a bound of 
\[
\ll_\phi \delta_0 (a) \epsilon^{d_{\operatorname{min}}/2}  
\norm{\tau^{(P_0)}}_1
\]
for \eqref{SumCenterHeisenberg}.

We now turn to \eqref{SumNonCenterHeisenberg}. Of course, we can no longer assume that 
$\lambda_p$ is a multiple of $\tilde{\alpha}^\vee$.
Write $\mathfrak{v} = \mathfrak{u}^{\operatorname{ab}}$ and
let $N_1 = \prod_p p^{n_p}$ be the smallest integer such that
\[
\K_{\fin} \zeta^{-1} m \K_{\fin} \cap U (\A_{\fin})
\subset
\exp (N_1^{-1} \mathfrak{v}_\Z \otimes \hat\Z + \mathfrak{u}_{\tilde{\alpha}} \otimes \A_{\fin} ).
\]
Let
\begin{align*}
\tilde{f}_\epsilon (v) &  =  \int_{U_{\tilde{\alpha}} (\R)}
F^{\phi,\epsilon} (\zeta (\exp v)
\tilde{u}) \, d \tilde{u}, \quad v \in \mathfrak{v} \otimes \R, \\
\tilde{\tau} (v) & =  \int_{U_{\tilde{\alpha}} (\A_{\fin})}
\tau (\zeta (\exp v) \tilde{u}) \, d \tilde{u},
\quad v \in \mathfrak{v} \otimes \A_{\fin}.
\end{align*}
Note that $\tilde{f}_\epsilon (v) \le \epsilon^{1/2}
f_0 (\norm{v}^2 \epsilon^{-1})$ for a fixed bounded
compactly supported function $f_0$ on $\mathfrak{v} \otimes \R$.
We can write \eqref{SumNonCenterHeisenberg} as
\begin{equation} \label{splitsumHeisenberg}
\tilde{\alpha} (a)
\sum_{v \in \mathfrak{v}, \, v \ne 0}
(\tilde{f}_\epsilon \otimes \tilde{\tau})
(\Ad (a \kappa)^{-1} v),
\end{equation}
which is the sum of 
\[
\tilde{\alpha} (a) \sum_{v \in \mathfrak{v}, \, \dim \langle \operatorname{supp} (v) \rangle\ge r-1}
(\tilde{f}_\epsilon \otimes \tilde{\tau})
(\Ad (a \kappa)^{-1} v)  
\]
and
\[
\tilde{\alpha} (a) \sum_{v \in \mathfrak{v}, \, 
1 \le \dim \langle \operatorname{supp} (v) \rangle\le r-2}
(\tilde{f}_\epsilon \otimes \tilde{\tau})
(\Ad (a \kappa)^{-1} v).
\]
The first subsum here can again by treated by applying Lemma \ref{LemmaLatticePointsLargeSupport}, which yields the bound
\[
\ll_\phi \delta_P (a) \epsilon^{r/2}  C^{\omega (N_1)} \norm{\tau^{(P_0)}}_1.
\]
The second subsum is clearly only relevant for $r \ge 3$, which excludes the $G_2$ case.
One splits it further according to 
the possible subsets $S = \operatorname{supp} (v) \subset \Phi^+_U$. 
Note that by the third to fifth parts of Lemma \ref{LemmaCartanAbelian},
we can estimate
\[
\tilde{\tau} (v) \ll N_2 \mathbf{1}_{N_1^{-1} \mathfrak{v}_\Z \otimes \hat\Z} (v),
\]
where $N_2 = \prod_p p^{\nu_p}$ with $\nu_p =  \sprod{\varpi_4}{\lambda_p}$ in the $F_4$ case, $\nu_p = \frac12 \sprod{\varpi_1}{\lambda_p}$ in the $E_8$ case and $\nu_p = \sprod{\varpi_1}{\lambda_p}$ in the $C_n$ case.
We get 
\[
\tilde{\alpha} (a)
\sum_{v \in \mathfrak{v}, \, \operatorname{supp} (v) = S} 
(\tilde{f}_\epsilon \otimes \tilde{\tau}) (\Ad (a \kappa)^{-1} v)
\ll_\phi \delta_0 (a) \left[ \prod_{\alpha \notin S \cup \{ \tilde{\alpha} \} } \alpha (a) \right]^{-1} 
\epsilon^{\frac{\abs{S}+1}2}  
N_1^{\abs{S}} N_2.
\]
By Lemma 
\ref{LemmaCardinalitySupport}, in the $F_4$ and $E_8$ cases the set $S$ has at most
$d_{\operatorname{min}}-1$ many elements. 
In the $C_n$ case, we at first only consider such sets $S$ and
postpone the full treatment of this case to the end.
Using Lemma \ref{rootlemma} 
with $\beta \in S$, 
we obtain for the second subsum a bound of 
\[
\ll \delta_0 (a) \epsilon^{d_{\operatorname{min}}/2}  
N_1^{d_{\operatorname{min}}-1} N_2
\]
in the $F_4$ and $E_8$ cases. In the $C_n$ case, the same bound
applies to the subsum over all $v$ with 
$1 \le \abs{\operatorname{supp} (v)} \le d_{\operatorname{min}}-1 = n-1$.

We use the second part of Lemma \ref{LemmaCartanAbelian}
to bound the exponents $n_p$. In the $E_8$ case, one can easily verify that $\rho - 14 \tilde{\alpha} - \frac12 \varpi_1$ is a
non-negative linear combination of positive roots. In the $F_4$ case, the same is true for 
$\rho - \frac72 \tilde{\alpha} - \varpi_4$, and in the $C_n$ case
for $\rho - \frac{n-1}2 \varpi_2 - \varpi_1$. Therefore in all cases
\[
\frac{d_{\operatorname{min}}-1}2 n_p + \nu_p
\le
\sprod{\rho}{\lambda_p},
\]
which implies by \eqref{InequalityTauP0} that 
\[
N_1^{d_{\operatorname{min}}-1} N_2 
\le \prod_p p^{\sprod{\rho}{\lambda_p}}
\le \norm{\tau^{(P_0)}}_1.
\]
The second subsum in \eqref{splitsumHeisenberg} is therefore bounded by
\[
\ll_\phi \delta_0 (a) \epsilon^{d_{\operatorname{min}}/2}  
\norm{\tau^{(P_0)}}_1
\]
in the $F_4$ and $E_8$ cases. In the $C_n$ case, the 
subsum over $v$ with 
$1\le \abs{\operatorname{supp} (v)} \le  n-1$ satisfies this bound.
This finishes the estimation of $\CmJ_{Z U \smallsetminus Z}^{G}(F^{\phi,\epsilon} \otimes \tau; \chi)$ in the Heisenberg case for $G_2$, $F_4$ and $E_8$.

We still need to finish the proof in the $C_n$ case. Here, we simply group the entire sum over $v$ in \eqref{splitsumHeisenberg} according to $S = \operatorname{supp} (v)$. The case where $\abs{S} \le d_{\operatorname{min}} - 1 = n-1$ has already been treated. In the case where 
$n-1 \le \abs{S}$ we have $\abs{S} \le \dim \mathfrak{v} = 2n-2$, and we obtain the estimate of
\[
\ll_\phi \delta_0 (a) \epsilon^{n/2}  
N_1^{2n-2} N_2.
\]
Since $\rho - (n-1) \varpi_2 - \varpi_1$ is still a 
non-negative linear combination of positive roots, we have the bound
$N_1^{2n-2} N_2 \le \norm{\tau^{(P_0)}}_1$, which finishes the proof also in the $C_n$ case.
\end{proof}

\section{Spectral estimates} \label{SectionSpectral} 
In this section we bound the contribution of the
discrete spectrum of a proper Levi subgroup $M$ to the trace formula for $G$, assuming a good bound for the logarithmic derivatives of the normalizing factors (property (TWN+) for $G$, cf. \eqref{eq:twn} below). The main result is Proposition \ref{prop:spectral:smooth}. The proof of this result necessitates an inductive proof of an upper bound for the part of the spectrum of $M$ in a complex ball (Proposition \ref{prop:inductive:step:smooth}), which also implies our bound for the non-tempered spectrum (Theorem \ref{thm:complementary}). While this bound is of the right order of magnitude, it is superseded by our main result
Theorem \ref{thm:main}, which we prove at the end of this section by combining the previous geometric and spectral estimates.
The arguments of this section follow \cite{LaMu09}. The spectral bounds are independent of our results on the geometric side.

\subsection{Review of the spectral side of the trace formula}
We recall the basic structure of the spectral expansion of the trace formula following \cites{FLM11,FLM14}.
For $f\in C_c^\infty(G(\A)^1)$ we have 
\begin{equation}\label{eq:specside}
 J (f)
 = \sum_{[M]} J_{\spec,M} (f) =
 \sum_{[M]} \sum_{s\in W(M)} J_{\spec,M,s}(f),
\end{equation}
where $[M]$ runs over conjugacy classes of Levi subgroups with representatives $M\in\cL$, and $W(M) = N (M) / M$, represented by elements of $W$. 
The distribution $J_{\spec,M} = \sum_{s\in W(M)} J_{\spec,M,s}$ is the contribution of the discrete spectrum of $M$ to the trace formula.
In particular, the distribution $J_{\disc}=J_{\spec,G}$ equals the trace of $f$ on the discrete spectrum of $G$.

To give a description of $J_{\spec,M,s}$ in general, we need to introduce some notation.
For any parabolic subgroup $P=MU \in \cP (M)$ let $\cA^2(P)$ be the space of automorphic forms $\phi$ on $U (\A) M (\Q) \backslash G (\A)$ such that
$\delta_P^{-\frac12} \phi (\cdot k)$ is a square-integrable automorphic form on 
$A_M M (\Q) \backslash M (\A)$ for all $k \in \K$, and $\bar\cA^2(P)$ its completion with respect to the natural inner product. Let
$\rho(P,\lambda, g)$, $\lambda \in (\af_M^G)^*_\C$, be the natural representation of $G (\A)$ on $\bar \cA^2(P)$. It is isomorphic to the representation of $G (\A)$ induced from the representation $L^2_{\disc} (A_M M (\Q) \backslash M (\A)) \otimes e^{\sprod{\lambda}{H_M(\cdot)}}$ of $P (\A)$.

For two parabolic subgroups $P, Q \in \cP (M)$ let
$M_{P|Q}(\lambda):\cA^2(Q)\longrightarrow\cA^2(P)$ be the global intertwining operator introduced in \cite{Ar82a}. If $P$ and $Q$ are
adjacent along $\alpha \in \Phi_M$, 
then $M_{P|Q}(\lambda)$ depends only on the scalar $\langle \lambda, \alpha^\vee\rangle$.
In this situation we define the operator 
\[
\delta_{P|Q}(\lambda) = M_{Q|P}(\lambda) D_{\varpi}M_{P|Q}(\lambda): \cA^2(Q) \longrightarrow\cA^2(Q),
\]
where $\varpi\in (\af_M^G)^*$ is such that $\langle\varpi,\alpha\rangle=1$, and $D_\varpi M_{P|Q}(\lambda)$ denotes the derivative in the direction of $\varpi$ (we use the convention of \cite{FLM14} for $\delta_{P|Q}(\lambda)$). 

For $P \in \cP (M)$ and $L \in \cL$ with $L \supset M$ let 
$\cB_{P,L}$ be the set of all $m$-tuples $\underline{\beta}=(\alpha_1^\vee, \ldots, \alpha_m^\vee)$ of coroots in $\Phi_P^\vee$, $m=\dim\af_L^G$, whose projections to $\af_L^G$ form a basis of $\af^G_L$.
For $\underline{\beta} \in \cB_{P,L}$, let $\Xi_L (\underline{\beta})$ be the set of all $m$-tuples $(Q_1, \ldots, Q_m)$ of parabolic subgroups of $G$ containing $M$ such that $\af^{Q_i}_M$ is the line spanned by $\alpha_i^\vee$. This means that $Q_i$ is generated by a unique pair of parabolic subgroups
$P_i, P_i' \in \cP (M)$ with $P_i |^{\alpha_i} P'_i$.
To any element $\cX$ of $\Xi_L (\underline{\beta})$ we associate 
an operator $\Delta_{\cX}(P,\lambda): \cA^2(P) \longrightarrow\cA^2(P)$, which is a
product of rank one intertwining operators and their derivatives:
\[
 \Delta_{\cX}(P,\lambda)
 = \frac{\vol(\underline{\beta})}{m!} M_{P_1'|P}(\lambda)^{-1}\delta_{P_1|P_1'}(\lambda) M_{P_1'|P_2'}(\lambda)\cdots M_{P_{m-1}'|P_m'}(\lambda)\delta_{P_m|P_m'}(\lambda) M_{P_m'|P}(\lambda).
\]
In \cite{FLM11}*{pp. 179-180}, a map $\cX_L: \cB_{P,L} \longrightarrow \bigcup_{\underline{\beta}} \Xi_L (\underline{\beta})$ with
the property that $\cX_L (\underline{\beta}) \in  \Xi_L (\underline{\beta})$
for any $\underline{\beta} \in \cB_{P,L}$ was constructed in a purely combinatorial way. In fact, $\cX_L$ depends on certain choices, which we may however regard as being fixed once and for all.

Given $M \in \cL$ and $s \in W (M)$, let $L_s\in \cL$ be the smallest Levi subgroup containing $M$ and $s$
and 
\[
\iota_s=|\det(1-s) |_{\af_{M}^{L_s}}|^{-1}.
\]
For $P \in \cP (M)$ let $s: \cA^2(P) \to \cA^2(sP)$ be left translation by $s^{-1}$ and $M (P,s) = M_{P|sP} (0) \circ s: 
\cA^2(P) \to \cA^2(P)$ as in \cite{Ar82a}*{p. 1309}.
The main result of \cite{FLM11} is that the distribution $J_{\spec,M,s}$ is given by
\begin{equation}\label{eq:spectral:int}
J_{\spec,M,s}(f)=  \frac{\iota_s}{|W(M)|} \sum_{\beta\in \cB_{P,L_s}}\int_{i(\af_{L_s}^G)^*} \tr\left(\Delta_{\cX_{L_s}(\underline{\beta})}(P,\lambda) M(P,s) \rho(P,\lambda, f)\right)\, d\lambda,
\end{equation}
where 
$P \in \cP (M)$ is arbitrary. The operators are of trace class and the integrals are absolutely convergent (even when the trace is replaced by the trace norm).

For our purposes the behavior of the intertwining operators $M_{P|Q}(\lambda)$ is controlled by the global normalizing factors $n_\alpha (\pi, s)$ introduced by Langlands and Arthur. They are meromorphic functions of a complex variable $s$ associated to discrete automorphic representations
$\pi\in\Pi_{\disc}(M (\A))$ and roots $\alpha\in \Phi_M$.
Let $\pi \in \Pi_{\disc} (M (\A))$ and let
$\AF^2_\pi(P)$ be the space of all $\phi \in \AF^2 (P)$ for which the functions $\modulus_P^{-\frac12}\phi(\cdot g)$, $g \in G (\A)$,
belong to the $\pi$-isotypic subspace of $L^2 (A_M M (\Q) \bs M (\A))$.
For any $P\in\cP(M)$ we have a canonical isomorphism of $G(\A_{\fin})\times(\gf_\C,\K_\infty)$-modules
\[
j_P:\Hom(\pi,L^2(A_MM(\Q)\bs M(\A)))\otimes\Ind_{P(\A)}^{G(\A)}(\pi)\rightarrow\AF^2_\pi(P).
\]
Suppose that $P|^\alpha Q$.
The restriction of the operator $M_{Q|P}(\lambda)$
to the space $\AF^2_\pi(P)$ satisfies
\begin{equation} \label{eq: normalization}
M_{Q|P}(\lambda)\circ j_P=n_\alpha(\pi,\sprod{\lambda}{\alpha^\vee})\cdot j_Q\circ \left(\Id\otimes R_{Q|P}(\pi,\sprod{\lambda}{\alpha^\vee})\right),
\end{equation}
where $R_{Q|P}(\pi,s)=\otimes_v R_{Q|P}(\pi_v,s)$ is the product
of the locally defined normalized intertwining operators and $\pi=\otimes_v\pi_v$
\cite{Ar82a}*{\S 6}.
Note that $|n_\alpha(\pi,it)| = 1$ and that the operators $R_{Q|P}(\pi_v,it)$ are unitary for all $t \in \R$. 

In \cite{FiLaintertwining}*{Definition 3.3}, a conditional bound on the derivatives of the
normalizing factors $n_\alpha(\pi,it)$ was formulated and verified in several cases. The group $G$ satisfies property (TWN+), if 
for any proper Levi subgroup $M$ of $G$ defined over $\Q$ and every $\alpha\in \Phi_M$ we have 
 \[
  \int_T^{T+1} |n_\alpha'(\pi, it)| \, dt
  \ll \log\left(|T|+ \Lambda (\pi_\infty; p^{\operatorname{sc}}) + 
\operatorname{level} (\pi; p^{\operatorname{sc}}) \right)
 \]
 for all $\pi\in\Pi_{\disc}(M)$ and all $T\in\R$. The definitions of $\Lambda (\pi_\infty; p^{\operatorname{sc}})$ and $\operatorname{level} (\pi; p^{\operatorname{sc}})$ are explained in [ibid., \S 2.5]. Here, we only note that
$\Lambda (\pi_\infty; p^{\operatorname{sc}}) \le 1 +  \|\lambda_{\pi_\infty}\|^2$ and that 
the nonnegative integer
$\operatorname{level} (\pi; p^{\operatorname{sc}})$ is bounded for 
representations $\pi$ of bounded level, i.e., for all
$\pi$ containing a vector fixed by a given open compact subgroup $K_M$ of $M(\A_\fin)$.  Therefore, if $G$ satisfies property (TWN+), then for any open normal subgroup $K$ of $\K_\fin$, any proper Levi subgroup $M$ of  $G$ defined over $\Q$ and every $\alpha\in \Phi_M$ the estimate
 \begin{equation}\label{eq:twn}
  \int_T^{T+1} |n_\alpha'(\pi, it)| \, dt
  \ll_K \log\left(|T|+ \|\lambda_{\pi_\infty}\| + 2 \right)
 \end{equation}
holds for all $\pi\in\Pi_{\disc}(M)$ which contain a vector fixed by 
$K_M = M (\A_\fin) \cap K$, 
and all $T\in\R$. 
In this paper, we only use 
\eqref{eq:twn} instead of the full property (TWN+).

In any case, in this section we will assume that $G$ satisfies property (TWN+), which includes a large number of groups.
\begin{theorem}[\cite{FiLaintertwining}]
Any split classical group, as well as the split group of type $G_2$, satisfies property (TWN+).
\end{theorem}

\subsection{Bounds on the spherical unitary dual}\label{SectionUnitary}
As in \cite{LaMu09}*{\S 3.3} we define $(\af_{un}^G)^*\subset (\af_0^G)_\C^*$ as the set of all $\lambda\in (\af_0^G)_\C^*$ such that the irreducible spherical subquotient of the induced representation ${\rm Ind}_{P_0}^G(e^{\langle\lambda, H_G(\cdot)\rangle})$ of $G(\R)$ is unitarizable. 
Then for any $M\in\cL$ we have $(\af_{un}^M)^*+ i(\af_M^G)^*\subset (\af_{un}^G)^*$. 
For $w\in W$ let $\af_w^*=\{\lambda\in(\af_0^G)_\C^*\mid w\lambda=-\overline{\lambda}\}$. Recall that $L_w\in \cL$ denotes the smallest $L\in\cL$ such that $w\in W_L$. Then $\af_w^*= \af_{w, -1}^*+ i(\af_{L_w}^G)^*$ where $\af_{w, -1}^*= \{\lambda\in(\af_0^G)^*\mid w\lambda= -\lambda\}$. 

\begin{lemma}[\cite{DKV79}*{Lemma 8.1}, \cite{LaMu09}*{\S 3.3}]\label{lem:properties:non-tempered}
 Let $\lambda\in (\af_{un}^G)^*$. Then $\|\Re\lambda\|\le \|\rho\|$ and there exists $w\in W$ with $\lambda\in\af_w^*$, that is,
 \[
  (\af_{un}^G)^*\subset \bigcup_{w\in W} \af_w^*.
 \]
\end{lemma}

\subsection{Upper bounds for the spectrum}
Without loss of generality we can always replace $K$ by an open subgroup.
Fix therefore in the following a factorizable open normal subgroup $K = \prod_p K_p$ of $\K_\fin$,
and set $K_M = M (\A_\fin) \cap K$ for all $M \in \cL$. 
We set
\[
m^{M,K_M} (\lambda) = m^M (\lambda, \vol (K_M)^{-1} \mathbf{1}_{K_M}), \quad \lambda \in (\af_0^M)_\C^*,
\]
which is the multiplicity of the the spectral parameter $\lambda$ in the automorphic spectrum of $M$, weighted with the factor $\abs{W_M \lambda}^{-1}$, and $m^{M, K_M} (B) = \sum_{\lambda \in B} m^{M,K_M} (\lambda)$.

For the proof of our spectral bounds, we need to consider the following two properties of 
a Levi subgroup $M$, which we first formulate and then prove by induction. Recall the definition of $\tilde\beta^M$ from \S\ref{sec:plancherel}.
The first property is an upper bound of the correct order of magnitude on the trace of the discrete spectrum for our family of test functions:
\begin{equation} \label{propertyTB}
 J_{\disc}^M (\fn{t,\mu}{h}{\mathbf{1}_{K_M}})
 \ll_h t^{r_M} \tilde\beta^M (t,\mu),
\quad \mu\in i (\af_0^M)^*, \quad t\ge 1.
\end{equation}
The second property is a simple upper bound, again of the correct order of magnitude, for the part of the spectrum in a complex ball. For $t\ge 1$ and $\mu\in(\af_0^G)_\C^*$ let $B_t(\mu)\subset (\af_0^G)_\C^*$ denote the ball of radius $t$ around $\mu$.
\begin{equation}\label{propertySB}
 m^{M,K_M} (B_t(\mu)) \ll_K t^{r_M} \tilde\beta^M (t,\mu),
 \quad \mu\in i (\af_0^M)^*, \quad t\ge 1.
\end{equation}

\begin{remark}
Since the real part of $\lambda$ is bounded,
property \eqref{propertySB} for $M$ immediately implies the following strengthening:
\begin{equation}\label{eq:spectral:nontemp:ball}
 m^{M,K_M} \left(\{\lambda\in(\af_0^M)_\C^*\mid \Im\lambda\in B_t(\Im\mu)\}\right)
 \ll_K  t^{r_M} \tilde\beta^M (t,\mu),
 \quad \mu\in (\af_0^M)_{\C}^*, \quad t\ge 1.
\end{equation}
\end{remark}

The inductive procedure is based on two auxiliary results. The first of these 
(which is due to \cite{DKV79} and \cite{LaMu09}) is that the trace bound \eqref{propertyTB} for a group $M$ implies the
multiplicity bound \eqref{propertySB} for $M$. 

\begin{lemma}\label{cor:estimate:balls}
Assume \eqref{propertyTB} holds for a Levi subgroup $M$. Then $M$ also satisfies \eqref{propertySB}.
\end{lemma}

The second result is a bound for the spectral contribution 
$J_{\spec,M,s} (f)$ to the trace formula for $G$ assuming that \eqref{propertySB} holds for $M$. 
For $M, L\in\cL$, $M\subset L$, let $(\af_M^L)^\perp$ be the orthogonal complement of $\af_M^L$ in $(\af_0^G)^*$. We have 
\[
i(\af_M^L)^\perp= i (\af_0^M)^*\oplus i(\af_L^G)^*. 
\]
If $\nu\in i(\af_M^L)^\perp$, we write $\nu^M$ for the projection of $\nu$ onto $i (\af_0^M)^*$ along $i(\af_L^G)^*$.
For $h\in C_c^\infty(\af_0^G)^W$ and $\nu\in i (\af_0^G)^*$ let
\[
\cM_M (\hat{h}) (\nu) = 
\max_{\lambda \in B_{2+\|\rho^M\|}(\nu)} |\hat{h} (\lambda)|
\] 
denote the maximal value of $|\hat{h}|$ on the ball of radius 
$2+\|\rho^M\|$ around $\nu$ in $(\af_0^G)^*_\C$.

\begin{proposition}\label{prop:spectral:smooth}
Suppose that $G$ satisfies property (TWN+) and that property \eqref{propertySB} holds for a Levi subgroup $M \in \cL$.
Then for $h\in C_c^\infty(\af_0^G)^W$, every compactly supported  $\tau:G(\A_{\fin})/K\longrightarrow\C$, and $s\in W(M)$ we have
 \begin{equation}\label{eq:spectral}
J_{\spec,M,s}(\fn{}{h}{\tau})
\ll  \|\tau\|_{1}^{(P)} \int_{i(\af_M^L)^\perp} 
\cM_M (\hat{h}) (\nu) \; \tilde{\beta}^M(\nu^M) \left(\log\left(2+\|\nu\|\right)\right)^{r_G - r_{L_s}} \, d \nu,
\end{equation}
where $P$ is a parabolic subgroup of $G$ with Levi subgroup $M$.
\end{proposition}

For the special test functions $\fn{t,\mu}{h}{\tau}$ we obtain the following estimate:

\begin{corollary}\label{prop:spectral}
Suppose that $G$ satisfies property (TWN+) and 
that property 
\eqref{propertySB} holds for a Levi subgroup $M \in \cL$. 
Then for every compactly supported  $\tau:G(\A_{\fin})/K\longrightarrow\C$ and $s\in W(M)$ we have for any $N \ge 1$:
 \begin{equation}\label{SpectralUpperBound}
J_{\spec,M,s}(\fn{t,\mu}{h}{\tau})
\ll_h  \|\tau\|_{1}^{(P)} t^{r_G + r_M - r_{L_s}}
\sum_{w \in W}
\frac{ (\log (1 + t + \norm{\mu}))^{r_G-r_{L_s}}}{(1 + t^{-1} \norm{\mu^{wL_s}_{wM}})^{N}}
\tilde{\beta}^{wM} (t, \mu^{wM}).
\end{equation}
\end{corollary}

We defer the proof of Proposition \ref{prop:spectral:smooth} and of its corollary to the next subsection. 

\begin{proof}[Proof of Lemma \ref{cor:estimate:balls}]
To simplify the notation, let
during the proof $G$ denote a Levi subgroup in $\cL$ with $r_G > 0$ satisfying \eqref{propertyTB}.
Let $\mu\in i(\af_0^G)^*$ and let $M\in \cL$ be maximal with the property that $\mu\in i(\af_M^G)^*$.  
 We use induction on $k=\dim (\af_M^G)^*$.    
If $k=0$, that is $M=G$, then $\mu=0$. 
By \cite{DKV79}*{Lemma 6.3} we can choose a function $h\in C_c^{\infty}(\af_0^G)^W$ such that $\hat{h}(\lambda)\ge0$ for any $\lambda\in (\af_{un}^G)^*$, and $|\hat{h}(\lambda)|\ge 1$ for every $\lambda\in B_1(0)$. 
 Then
 \[
 \vol (K)^{-1} J_{\disc}(\fn{t,0}{h}{\mathbf{1}_K}) 
  = \sum_{\lambda} \hat{h} (t^{-1} \lambda) m (\lambda)
  \ge m(B_t(0)).
 \]
Since $J_{\disc}(\fn{t,0}{h}{\mathbf{1}_K}) \ll_h t^d$ by the assumption \eqref{propertyTB}, the assertion follows for $\mu=0$.

Now suppose $0\le k<r$ and that the assertion is true for every $\mu$ with $\mu\in i(\af_M^G)^*$ and $\dim (\af_M^G)\le k$. Let $h$ be as before. 
Recall the upper bound for the spherical unitary dual 
$(\af_{un}^G)^*$ described in \S\ref{SectionUnitary} above.
Writing 
\[\af_{\not\subseteq M}^* = (\af_{un}^G)^* - \bigcup_{w\in W_M} \af_w^*,
\]
we get as in the proof of~\cite{LaMu09}*{Proposition 4.5} that
\[
m (B_t(\mu))\le \vol (K)^{-1} |J_{\disc}(\fn{t,\mu}{h}{\mathbf{1}_K}) | + \sum_{\lambda\in \af_{\not\subseteq M}^* } m(\lambda) \hat{h}_t (\lambda-\mu) 
+ m(B_t(\mu)\cap \af_{\not\subseteq M}^*) .
\]
To bound the first term on the right hand side, we use again the assumption $J_{\disc}(\fn{t,\mu}{h}{\mathbf{1}_K}) \ll_{h,K} t^r \tilde\beta(t,\mu)$. The remaining two terms can be bounded exactly as in the proof of~\cite{LaMu09}*{Proposition 4.5}, using only the induction hypothesis.
\end{proof}

We can now carry out the induction, assuming the validity of Proposition \ref{prop:spectral:smooth}.

\begin{proposition}\label{prop:inductive:step:smooth}
Assume property (TWN+) for $G$. Then \eqref{propertyTB} and
\eqref{propertySB} hold for all Levi subgroups $M \in \cL$ of $G$.
\end{proposition} 

\begin{proof}
Use induction on $\dim\af_0^M$. The case $M=T_0$ is trivial. 
For the induction step, let for the remaining part of the proof $G$ denote a Levi subgroup in $\cL$ with $r_G > 0$, and assume \eqref{propertySB} for all proper Levi subgroups $M$ of $G$, as well as property (TWN+) for $G$. Note that property (TWN+) is hereditary for Levi subgroups.

According to \eqref{eq:specside} we can then write
 \[
 J_{\disc}(\fn{t,\mu}{h}{\mathbf{1}_K})
 = J (\fn{t,\mu}{h}{\mathbf{1}_K}) - \sum_{{M}\neq G} \sum_{s\in W(M)} J_{\spec,M,s}(\fn{t,\mu}{h}{\mathbf{1}_K})
 \]
for $\mu\in i (\af_0^G)^*$, $t\ge 1$,
We apply Lemma \ref{LemmaSimpleGeometricBound} 
to bound 
\[
J (\fn{t,\mu}{h}{\mathbf{1}_K})
\ll_h t^{r_G} \tilde\beta (t,\mu),
\]
and Corollary \ref{prop:spectral}, \eqref{SpectralUpperBound} for the contributions of 
$M \neq G$. Note here that $\tilde{\beta}^M (t, \mu) \ll \tilde{\beta} (t, \mu) (t + \norm{\mu})^{-1}$ for $M \neq G$, which implies that the right hand side of \eqref{SpectralUpperBound} is indeed $\ll t^{r_G} \tilde\beta (t,\mu)$.

Therefore, \eqref{propertyTB} holds for the group $G$, and because of Lemma \ref{cor:estimate:balls} we obtain
\eqref{propertySB} for $G$ as well, which finishes the proof.
\end{proof}

\subsection{Bounding the contribution of the continuous spectrum}
It remains to prove Proposition \ref{prop:spectral:smooth} and its corollary.

\begin{proof}[Proof of Proposition \ref{prop:spectral:smooth}]
The proof follows along the lines of \cite{LaMu09}*{\S 6}. 
We write $L=L_s$. 
Decompose $\bar{\cA}^2 (P)$ as the completed direct sum of the 
Hilbert spaces
$\bar{\cA}^2_\pi (P)$, the completions of the spaces
$\cA^2_\pi (P)$, for $\pi\in\Pi_{\disc}(M (\A))$.
Denote by $J_{M,s,\underline{\beta}}(f)$ the integral on the right hand side of \eqref{eq:spectral:int} corresponding to $\underline{\beta}\in \cB_{P,L}$. 
Then
\begin{equation}\label{eq:spectral1}
\left| J_{M,s,\underline{\beta}}(\fn{}{h}{\tau})\right|
\le \sum_{\pi\in\Pi_{\disc}(M (\A))} \int_{i(\af_L^G)^*} \left|\tr\left( \left. \Delta_{\cX_{L}(\underline{\beta})}(P,\lambda) M(P,s) \rho(P,\lambda, \fn{}{h}{\tau})\right|_{\bar{\cA}^2_\pi(P)}\right)\right|\, d\lambda.
\end{equation}
We have 
\begin{align*}
\left. \rho(P,\lambda, \fn{}{h}{\tau}) \right|_{\bar{\cA}^2_\pi(P)} 
& =   \left. \rho(P,\lambda, \fna{}{h}) \right|_{\bar{\cA}^2_\pi(P)}
\left. \rho_{\fin} (P, \lambda, \tau) \right|_{\bar{\cA}^2_\pi(P)} \\
& =  \hat{h}(\lambda_{\pi_{\infty}}+\lambda)\Pi_{\cpt_\infty}
\left. \rho_{\fin} (P, \lambda, \tau) \right|_{\bar{\cA}^2_\pi(P)},
\end{align*}
where $\Pi_{\cpt_\infty}$ denotes the projection of $\bar{\cA}^2_\pi(P)$ onto the $\cpt_\infty$-invariants, and $\lambda_{\pi_{\infty}}$ is the infinitesimal character of $\pi_{\infty}$.
The range of the operator $\rho(P,\lambda, \fn{}{h}{\tau})$ is contained in the finite-dimensional space of $\cpt_\infty K$-invariants.
The operator  $M(P,s)$  is unitary and commutes with $\rho(P,\lambda)$ 
for $\lambda\in i(\af_L^G)^*$. Moreover, the operators $M_{P|Q}(\lambda)$ are also unitary for $\lambda\in i(\af_L^G)^*$. We can therefore estimate
the integrand
\[
\left|\tr\left( \left. \Delta_{\cX_{L}(\underline{\beta})}(P,\lambda) M(P,s) \rho(P,\lambda, \fn{}{h}{\tau})\right|_{\bar{\cA}^2_\pi(P)}\right)\right|
\]
by
\[
\dim \cA_\pi^2(P)^{K} \abs{\hat{h}(\lambda_{\pi_{\infty}}+\lambda)}
\;
\| \rho_{\fin} (P, \lambda, \tau)|_{\bar{\cA}^2_\pi(P)} \| \;
\prod_{j=1}^m \| \delta_{P_j|P_j'}(\lambda)|_{\cA_\pi^2(P_j')^{\cpt_\infty K}}\|,
\]
where $\| \cdot \|$ denotes the operator norm.
Using the unitarity of $\pi$, it is easy to see that
\[
\| \rho_{\fin} (P, \lambda, \tau)|_{\bar{\cA}^2_\pi(P)} \| 
\le \vol (K)^{-1/2} \|\tau\|_{1}^{(P)}.
\]
Recall that $P_j$ and $P_j'$ are adjacent along the root $\alpha_j$.  
From the factorization \eqref{eq: normalization} we get
\[
 \|\delta_{P_j|P_j'}(\lambda)|_{\cA_\pi^2(P_j')^{\cpt_\infty K}}\|
 \le |n_{\alpha_j}'(\pi,\sprod{\lambda}{\alpha_j^\vee})| +  \sum_p \|R_{P_j|P_j'}'(\pi_p,\sprod{\lambda}{\alpha_j^\vee})^{K_p} \|.
\]
Since we are considering only representations with a $\K_\infty$-fixed vector, the infinite place does not contribute to the sum. The finite set of primes that possibly contributes is determined by $K$. By \cite{FLM11}*{Lemma 2}, the matrix coefficients of the operators 
$R_{P_j|P_j'} (\pi_p, s)^{K_p}$ are rational functions of $p^{-s}$ whose numerators and denominators have their degrees bounded in terms of $K_p$. The integral of 
$\sum_p \|R_{P_j|P_j'}'(\pi_p,i t)^{K_p}\|$ over any interval of length one in $t$ is by [ibid., Lemma 1] therefore bounded in terms of $K$ only.

Fix a lattice $\Xi_M^L\subset i(\af_M^L)^\perp$ such that 
\[
i (\af_M^L)^\perp
=i(\af_M^L)^\perp\cap \bigcup_{\nu\in\Xi_M^L} B_1(\nu)
\]
and cover the sum-integral \eqref{eq:spectral} by the corresponding sum-integrals restricted to $\pi$ and $\lambda$ with $\Im \lambda_{\pi_\infty} + \lambda \in B_1 (\nu)$.
Using \eqref{eq:twn}, we can bound the left hand side of \eqref{eq:spectral} by a constant multiple of 
\[
\|\tau\|_{1}^{(P)} \sum_{\nu\in\Xi_M^{L}} \max_{\mu \in B_{1+\|\rho^M\|}(\nu)} |\hat{h} (\mu)| 
\sum_{\substack{\pi\in\Pi_{\disc}(M (\A)): \\ \lambda_{\pi_\infty}\in B_1^M(\nu^M)}} \dim \cA_\pi^2(P)^{\cpt_\infty K} \left(\log (2+\|\nu\|)\right)^{r_G - r_L}.
\]
For the sum over $\pi$ we use our assumption that
\eqref{propertySB} is valid for $M$, which yields
\[
 \sum_{\substack{\pi\in\Pi_{\disc}(M (\A)): \\ \lambda_{\pi_\infty}\in B_1^M(\nu^M)}} \dim \cA_\pi^2(P)^{\cpt_\infty K}
 \ll 
 \sum_{\substack{\pi\in\Pi_{\disc}(M (\A)): \\ \lambda_{\pi_\infty}\in B_1^M(\nu^M)}} m_{\disc} (\pi)  \dim \pi^{\cpt_\infty K_M}
 \ll 
 \tilde{\beta}^M (\nu^M).
\]
We obtain an upper bound of (a constant multiple of)
\begin{equation} \label{EstimateLatticeSum} 
\|\tau\|_{1}^{(P)} \sum_{\nu\in\Xi_M^{L}} \max_{\mu \in B_{1+\|\rho^M\|}(\nu)} |\hat{h} (\mu)| \; \;
\tilde{\beta}^M (\nu^M)
\left(\log (2+\|\nu\|)\right)^{r_G - r_L}.
\end{equation}
There exists a number $0 < c \le 1$ such that the balls of radius $c$ around the points of $\Xi_M^{L}$ are disjoint. We can easily estimate the sum over $\nu \in \Xi_M^{L}$ in \eqref{EstimateLatticeSum} by 
a constant multiple of the integral over the set $\bigcup_{\nu\in\Xi_M^L} B_c (\nu) \cap i(\af_M^L)^\perp$ of
$\cM_M (\hat{h}) (\nu) \;
\tilde{\beta}^M (\nu^M)
\left(\log (2+\|\nu\|)\right)^{r_G - r_L}$.
This establishes the upper bound of \eqref{eq:spectral}, concluding the proof.
\end{proof}

\begin{proof}[Proof of Corollary \ref{prop:spectral}]
Apply Proposition \ref{prop:spectral:smooth} to the 
function $h_{t,\mu}^W = \abs{W}^{-1} \sum_{w \in W} (h \circ w^{-1})_{t,w\mu}$. Estimating the average over $W$ by the triangle inequality, and
replacing $h \circ w^{-1}$ by $h$, we are reduced to treating $h_{t,\mu}$.
Note that 
\[
\cM_{M} (\hat{h}_{t,\mu}) (\nu)
\le \cM_{M} (\hat{h}) (t^{-1} (\nu-\mu))
\ll_{N, h} (1 + t^{-1} \norm{\mu^L_M})^{-N} (1 + t^{-1} \norm{\nu-\mu^M - \mu_L})^{-N}
\]
for any nonnegative integer $N$.
Substituting $\nu = \mu^M + \mu_L + t \xi$ in the integral, and using that
\[
\tilde{\beta}^M (\mu^M + t \xi^M) \ll (1 + \norm{\xi})^{d_M-r_M}
\tilde{\beta}^M (t, \mu^M), 
\]
we obtain an upper bound of
\[
\|\tau\|_{1}^{(P)} \frac{\left(\log (1+ t + \|\mu\|)\right)^{r_G - r_L}}{(1 + t^{-1} \norm{\mu^L_M})^{N}} \tilde{\beta}^M (t, \mu^M)
\int_{i(\af_M^L)^\perp} (1 + \norm{\xi})^{d_M-r_M-N} 
\left(\log (2 + \|\xi\|)\right)^{r_G - r_L} \, d \xi,
\]
which yields the assertion after averaging over $W$, if $N$ is taken to be sufficiently large.
\end{proof}

\subsection{The non--tempered spectrum} \label{SectionNontempered}
We can now prove that the non-tempered spectrum is of lower order than the Plancherel measure.
\begin{proof}[Proof of Theorem \ref{thm:complementary}]
Noting that $\abs{m(\lambda,\tau)} \le \|\tau\|_{1} m(\lambda)$,
we are reduced to showing that 
\[
m(B_t(0) \smallsetminus i(\af_0^G)^*) \ll t^{d-2}
\]
for all $t\ge1$.
This is a consequence of Proposition~\ref{prop:inductive:step:smooth}, and we can argue as in the proof of~\cite{LaMu09}*{Corollary 4.6}.
If $\lambda\in B_t(0)$ with $\lambda\not\in i(\af_0^G)^*$, then there exists $w\in W$, $w\neq\Mid$, with $\lambda\in\af_w^*\cap\af_{un}^*$. Note that $w\neq \Mid$ implies $\dim(\af_{M_w}^G)^*\le r-1$. 
By Lemma~\ref{lem:properties:non-tempered} we have $\|\Re\lambda\|\le \|\rho\|$, and $\Im\lambda\in i(\af_{M_w}^G)^*$ because of $\af_w^*= \af_{w, -1} + i (\af_{M_w}^G)^*$. Hence $\lambda\in (C+ i(\af_{M_w}^G)^*)\cap B_t(0)$ with $C= B_{\|\rho\|}(0)\cap (\af_0^G)^*$. Hence we can find a constant $c>0$ which is independent of $t$, and  $\lfloor t^{r-1} \rfloor$ many points $\mu_1,\ldots, \mu_{\lfloor t^{r-1} \rfloor}\in i(\af_{M_w}^G)^* \cap B_{t}(0)$ such that
 \[
  \left(C+ i(\af_{M_w}^G)^*\right)\cap B_t(0)
  \subset \bigcup_{j=1}^{\lfloor t^{r-1} \rfloor} B_c(\mu_j).
 \]
Since $\mu_j\in i(\af_{M_w}^G)^*$, we have $\langle\mu_j,\alpha^\vee\rangle=0$ for all $\alpha\in \Phi^{M_w,+}$. Hence 
 \[
 \tilde{\beta}(\mu_j)
 \ll (1+\|\mu_j\|)^{\abs{\Phi^+}-1} 
 \ll t^{\abs{\Phi^+}-1}.
 \]
The asserted bound follows now from Proposition~\ref{prop:inductive:step:smooth} by taking the sum over the points $\mu_1,\ldots, \mu_{\lfloor t^{r-1} \rfloor}$. 
\end{proof}

\subsection{Asymptotics for the spectrum in a bounded domain} \label{SectionAsymptotics} 
\begin{proof}[Proof of Theorem \ref{thm:main}]
Let at first $\Omega$ be an arbitrary bounded measurable set.
By integrating Corollary \ref{prop:spectral} over $t \Omega$,
we have for $t \ge 1$ the estimate
\[
  \int_{t\Omega} \left|J_{\spec,M,s}(\fn{1,\mu}{h}{\tau})\right| \,d\mu
  \ll_{h,\Omega} \|\tau\|_{1}^{(P)} t^{d_M + r_G - r_{L_s}} (\log ( 1 + t))^{r_G - r_{L_s}}.
\]
Summing these estimates over $M \neq G$ and $s \in W(M)$ we obtain:
\begin{equation} \label{IntegratedSpectral}
 \int_{t\Omega}  \left| J (\fn{1,\mu}{h}{\tau}) - \sum_{\lambda\in(\af_0^G)_\C^*} \hat{h}_{1,\mu}(\lambda) m(\lambda,\tau) \right|  \,d\mu  
\ll_{h,\Omega}  \|\tau \|_{1} t^{d_G - R} \log (1+t),
\end{equation}
where $R = \min_P \dim U_P \ge r$.
On the other hand, integrating the main geometric estimate 
of Corollary \ref{TheoremGeometricFtmu} yields
\begin{equation} \label{IntegratedGeometric}
 \int_{t\Omega}  \left| J_{\nc} (\fn{1,\mu}{h}{\tau}) \right|  \,d\mu  
\ll_{h,\Omega}  \|\tau \|_{1} t^{d_G - \min (R/2,r)} \log (1+t)^{r+\nu},
\end{equation}
where $\nu=0$ for classical and $\nu=1$ for exceptional groups.
Assume now that $\Omega$ is a 
bounded domain with rectifiable boundary (it is enough to assume that
the $(r-1)$-dimensional upper Minkowski content of the boundary is finite, cf.  \cite{DKV}*{Lemma 8.7}).
If $h(0)=1$, then by \cite{LaMu09}*{(4.6)} we have  
\[
\int_{t\Omega}  \int_{i (\af_0^G)^*} \hat{h}_{1,\mu}(\lambda) 
\beta (\lambda)
\, d\lambda d \mu - \int_{t\Omega} \beta (\lambda)
\, d\lambda \ll_{h,\Omega} t^{d-1},
\]
and therefore 
\begin{equation} \label{IntegratedPlancherel}
\int_{t\Omega}  
\left[ J (\fn{1,\mu}{h}{\tau}) - J_{\nc} (\fn{1,\mu}{h}{\tau}) \right] 
\, d \mu - \Lambda_{\Omega} (t) 
\sum_{\gamma \in Z (\Q)} \tau (\gamma)  \ll_{h,\Omega} \|\tau\|_{1} t^{d-1}.
\end{equation}
Furthermore, since $\abs{m(\lambda,\tau)} \le \|\tau\|_{1} m(\lambda)$, Proposition \ref{prop:inductive:step:smooth} implies the upper bound
\[
\sum_{\lambda \in B_t(\mu)} \abs{m(\lambda,\tau)} \ll \|\tau\|_{1} t^{r} \tilde\beta (t,\mu).
\]
Using this bound and the analogous bound 
\[
\sum_{\lambda \in B_t(0) \smallsetminus i (\af^G_0)^*} \abs{m(\lambda,\tau)} \ll \|\tau\|_{1} t^{d-2} 
\]
of Theorem \ref{thm:complementary}
for the non-tempered spectrum, the argument of 
\cite{LaMu09}*{p. 138} gives formally
\begin{equation} \label{CutoffEstimate}
\int_{t\Omega}  \sum_{\lambda} m(\lambda,\tau) \hat{h}_{1,\mu}(\lambda) \, d\mu - m(t\Omega, \tau)  
 \ll_{h,\Omega} \|\tau\|_{1} t^{d-1},
\end{equation}
the analogue of [ibid., (4.7)].
The theorem follows by combining
\eqref{IntegratedSpectral}, \eqref{IntegratedGeometric}, \eqref{IntegratedPlancherel} and
\eqref{CutoffEstimate}. 
\end{proof}

\begin{remark} \label{RemarkErrorTermA12} 
For $r = 1$ we obtained an error term of 
$O (\|\tau\|_{1} t^{3/2} \log (1+t))$. A more refined treatment of the geometric side yields an error term of $O (\|\tau\|_{1} t \log (1+t))$ in this case, which is in fact optimal regarding its dependence on $t$. We omit the details.
One can also hope to improve the error term in the $A_2$ case 
from $O (\|\tau\|_{1} t^{d-1} (\log (1+t))^2)$ to 
$O (\|\tau\|_{1} t^{d-1})$.
\end{remark}

\begin{remark} \label{RemarkSmoothedVariant} 
The combination of \eqref{IntegratedSpectral} and \eqref{IntegratedGeometric} yields a ''smoothed'' variant of Theorem 
\ref{thm:main} with a better remainder term. Let
$\Omega$ be an arbitrary bounded measurable set and $\chi_{t\Omega} (\lambda) = \int_{t\Omega} \hat{h} (\lambda - \mu) \, d \mu$ be the convolution of $\hat{h}$ with the characteristic function of $t\Omega$,
$t \ge 1$. Then
\begin{multline} 
 \sum_{\lambda} \chi_{t\Omega}(\lambda) m(\lambda,\tau) - \frac{\vol(G(\Q)\backslash G(\A)^1)}{\abs{W}} 
 \sum_{\gamma \in Z (\Q)} \tau (\gamma)
 \int \chi_{t  \Omega} (\lambda)  \beta ( \lambda) \, d \lambda
 \\
\ll_{h,\Omega}  \|\tau \|_{1} t^{d_G - \min (R/2,r)} \log (1+t)^{r+\nu}.
\end{multline}
Note that, apart from using a ''smooth cutoff,'' we do not separate the tempered and non-tempered parts of the spectrum here. 
\end{remark}

\begin{appendix}
\section{Estimates for spherical functions} \label{AppendixSpherical} 
Let for this appendix $G$ be a connected semisimple real Lie group with Iwasawa decomposition $G=ANK$. Let $\gf$ be the Lie algebra of $G$, and  $\gf=\kf\oplus \sg$ the Cartan decomposition of $\gf$ with $\kf$ the Lie algebra of $K$. Let $\af$ be the Lie algebra of $A$. Then $\af$ is a maximal abelian subalgebra of $\sg$, and $\sg=\Ad(K)\af$. Let $\Phi$ denote the set of roots of $A$ on $\gf$, and let $\Phi^+\subset \Phi$ be the subset of positive roots determined by $N$. If $M$ is a Levi subgroup of $G$ with $A\subset M$ and Lie algebra $\mf$, we let $\Phi^{M, +}$ be the set of roots $\alpha\in \Phi^+$ on $\mf$. For $\alpha\in \Phi$, we let $m_\alpha$ denote the dimension of the $\alpha$-eigenspace in $\gf$.

For $\lambda\in \af_\C^*$ define
\[
 \tilde D(\lambda)
 =\min_{M\subsetneq G} \prod_{\alpha\in\Phi^+\smallsetminus \Phi^{M, +}} (1+|\langle \lambda, \alpha^\vee\rangle|)^{m_\alpha/2},
\]
where the minimum is taken over all Levi subgroups $M\neq G$ with $A\subset M$. 

Let $H_0:G\longrightarrow \af$ be the Iwasawa projection. For $\lambda\in \af^*_\C$ and $g\in G$ we define the elementary spherical function as usual by
\[
 \phi_\lambda(g)=\int_K e^{\langle \lambda +\rho, H_0(kg)\rangle} \, dk,
\]
where we normalize the Haar measure on $K$ to have volume $1$.

For any compact set $\omega\subset \af$ and any $w\in W$ we define
\[
\Delta_w(\omega) =\{\alpha\in\Phi^+\mid \forall X\in\omega: ~ \langle w\alpha, X\rangle\neq0\}
\]
Our refined estimate for the spherical functions is the following:
\begin{proposition}\label{prop:sph:decay}
Let $\omega\subset \af$ be a compact set. Then
\begin{equation}\label{eq:sph:fct:decay}
\left| \phi_\lambda(e^{tX})\right|
\ll_\omega  \sum_{w\in W}\prod_{\alpha\in \Delta_w(\omega)} (1+ t |\langle\lambda, \alpha^\vee\rangle|)^{-m_\alpha/2}
\end{equation}
uniformly in $0\le t\le1$, $X\in \omega$, and $\lambda\in i\af^*$.
\end{proposition}

Note that the case where $t \ge t_0 > 0$ is already contained in \cite{DKV}*{Theorem 11.1}. The new feature of this estimate, which is crucial for the
application to the trace formula, is that we may let $t$ approach $0$.

\begin{corollary}\label{prop:est:sph:fct:levis}
Let $\CmC\subset \af$ be a compact set. 
Then
\begin{equation*}
 \phi_\lambda(e^X)
 \ll_{\CmC} \tilde D(\|X\|\lambda)^{-1}
\end{equation*}
for all $X\in \CmC$ and all $\lambda\in i\af^*$.
\end{corollary}

We first recall the main result from \cite{Du84}. 
Let $\pi:\sg\longrightarrow \af$ be the Cartan projection.

  \begin{proposition}[\cite{Du84}*{(1.11)}]\label{prop_duistermaat}
   There exists an analytic function $b:\sg\longrightarrow\R$ such that 
   \begin{equation}\label{eq:spherical:duistermaat}
    \phi_\lambda(e^X)
    = \int_{K} e^{\langle \lambda,\pi(\Ad(k) X)\rangle} b(\Ad(k) X)\, dk
   \end{equation}
   for all $X\in\af$ and all $\lambda\in \af^*_{\C}$.
  \end{proposition}

\begin{lemma}\label{lem:oscillation} 
Let $m=\dim N/2$ and let $\omega\subset \af$ be a compact set.
There exists a seminorm $\nu$ on $C^m(K)$ such that for every $\psi\in C^m(K)$ we have 
\begin{equation}\label{eq:DKV:fine:est}
\left| \int_{K} e^{\langle \lambda, \pi(\Ad(k) X)\rangle} \psi(k)\, dk \right|
 \le \nu(\psi) \sum_{w\in W}\prod_{\alpha\in \Delta_w(\omega)} (1+ |\langle\lambda, \alpha^\vee\rangle|)^{-m_\alpha/2}
\end{equation}
for all $X\in \omega$ and $\lambda\in i\af^*$. 
\end{lemma}
Similar estimates for  uniformly regular parameters can be derived from the asymptotic expansions of the spherical functions of Cartan motion groups in \cite{BaCl} and \cite{Ra88}*{Lemma 30}. 
\begin{proof}
The lemma is the analogue of \cite{DKV}*{Theorem 11.1}, but with the phase function $f_{X,\lambda}(k)= \langle \lambda, \pi(\Ad(k)X)\rangle$ instead of $F_{e^X,\lambda}(k)= \langle\lambda,  H_0(e^X k)\rangle$. To prove the lemma we can follow the proof of that theorem very closely.
This is possible since the function $f_{X,\lambda}$ shares all the relevant properties with $F_{e^X,\lambda}$. In fact, it is often easier to establish those properties for $f_{X,\lambda}$.
 
 More precisely, by \cite{DKV}*{Proposition 5.4} the sets of critical points of $F_{e^X, \lambda}$ and $f_{X,\lambda}$ coincide, and the Hessians of both phase functions are non--singular at the critical points by \cite{DKV}*{\S 1, \S5--6}. Furthermore, $f_{X,\lambda}$ is right invariant under the centralizer of $X$ in $K$ and left invariant under the centralizer of $\lambda$ in $K$ as is $F_{e^X,\lambda}$ \cite{DKV}*{\S 5}. 
 
It remains to check that the analogue of \cite{DKV}*{Lemma 11.8}, in particular part (b), holds for $f_{X,\lambda}$, that is, we need to show the following: Suppose $X_0$ is given, $k_0\in K$, and $E$ is the (unique) subset of $\Delta$ such that $k_0$ is a critical point for $f_{X_0, \alpha}$ for every $\alpha\in E$, but it is not a critical point for $f_{X_0, \alpha}$ for any $\alpha\in \Delta\smallsetminus E$.  Then there exists $Z\in\kf$ such that (a) $Z\in\kf_E:=\bigcap_{\alpha\in E} \kf_{\alpha^\vee}$ and (b) $f_{X_0, \alpha}(k_0; Z)>0$ for all $\alpha\in \Delta\smallsetminus E$. 
Here $\kf_{\alpha^\vee}$ denotes the centralizer of $\alpha^\vee$ in $\kf$.

We choose the element $Z\in\kf$ exactly as in the proof of \cite{DKV}*{Lemma 11.8}. Assertion (a) then follows exactly as in \cite{DKV}, as it depends only on the fact that $f_{X_0, \alpha}$ is left invariant under the centralizer of $\alpha$ in $K$, and right invariant under the centralizer of $X_0$ in $K$. For (b), we use \cite{DKV}*{(1.2)},  which gives a formula for the derivatives of $f_{X_0, \alpha}$. Together with an explicit computation this yields (b).

These properties of the phase function $f_{X,\lambda}$ suffice to prove the lemma along the lines of \cite{DKV}*{\S 11}.
\end{proof}

\begin{proof}[Proof of Proposition \ref{prop:sph:decay}]
We apply the estimate of Lemma \ref{lem:oscillation} to \eqref{eq:spherical:duistermaat}.  
For $X\in \omega$, $t\in[0,1]$ put $a_{t,X}(k)=b(t\Ad(k) X)$, $k\in K$. Then $a_{t,X}\in C^\infty(K)$ and if $\nu$ is a seminorm on $C^m(K)$, there exists $c_\nu>0$ such that  $\nu(a_{t,X})\le c_\nu$ for all $X\in \omega$ and $t\in[0,1]$ because of compactness. 
Hence by \eqref{eq:DKV:fine:est} we can find a constant $c=c_\omega$ such that for every $X\in \omega$,  $t\in[0,1]$, and $\lambda\in i \af^*$ we have
\[
\left| \phi_\lambda(e^{tX})\right|
\le  c_\omega \sum_{w\in W}\prod_{\alpha\in \Delta_w(\omega)} (1+ |\langle t\lambda, \alpha^\vee\rangle|)^{-m_\alpha/2},
\]
which proves the proposition.
\end{proof}

\begin{proof}[Proof of Corollary \ref{prop:est:sph:fct:levis}]
Let $R>0$ be such that $\CmC$ is contained in the ball of radius $R$ around $0$.
For a parabolic subgroup $P=MU\subset G$ with $A\subset M$ and a small constant $c>0$ define
\[
 \Omega_M=\{X\in\af\mid\|X\|=R,\ \forall \alpha\in\Phi^+\smallsetminus \Phi^{M,+}:\ |\langle\alpha,X\rangle|\ge c\}. 
\]
If $c$ is sufficiently small, the compact sets $\Omega_M$ cover the sphere $\{X\mid \|X\|=R\}$ as $M$ varies over all Levi subgroups $M \neq G$ with $A\subset M$.
Moreover, by definition, for each $w\in W$ the set $\Delta_w(\Omega_M)$ is equal to $\Phi^+\smallsetminus \Phi^{w^{-1}(M),+}$.
Applying Proposition \ref{prop:sph:decay} to $\omega=\Omega_M$ for each $M$, we can replace the product over $\alpha\in\Delta_w(\omega)$ in \eqref{eq:sph:fct:decay} by a product over $\alpha\in \Phi^+\smallsetminus \Phi^{M,+}$, and the sum over $w\in W$ by a maximum over all semi-standard Levi subgroups $M$.  
\end{proof}

\section{An elementary inequality for classical root systems} \label{AppendixInequalityClassical} 
Let for this appendix $\Phi$ be a (reduced) irreducible root system in a real vector space $V$,
$r = \dim V$ its rank, and $\Phi^+$ the set of positive roots with respect to a fixed basis.
Let $\cL$ be the intersection lattice of the associated hyperplane arrangement in the dual space $V^*$. The elements of $M$ correspond to semi-standard Levi subgroups of a reductive group $G$ with root system $\Phi$. Let $\cL'$ be the truncation of $\cL$ by its unique maximal element
(corresponding to $G$) and $\cL_{\max}$ be the set of elements of $\cL$ of rank $r-1$ (corresponding to maximal semi-standard Levi subgroups). For $M \in \cL$, let $\Phi^{M}$ be the set of roots and $\Phi^{M,+} = \Phi^{M} \cap \Phi^{+}$ be the set of positive roots vanishing on $M$.
For convenience, we set $\Phi^+_{\neg M} = \Phi^+ \smallsetminus \Phi^{M, +}$. The cardinality of this set is the dimension of the unipotent radical of a parabolic subgroup with Levi subgroup corresponding to $M$. (However, $\Phi^+_{\neg M}$ is 
the set of roots on the unipotent radical only if $M$ is standard.)
For  $M \in \cL$, we consider
\begin{equation}
\Delta_M ( \lambda) = 
\prod_{\alpha \in \Phi\smallsetminus \Phi^{M}} (1 + \abs{\sprod{\alpha}{\lambda}})^{\frac12} =
\prod_{\alpha \in \Phi^+_{\neg M}} (1 + \abs{\sprod{\alpha}{\lambda}})
\end{equation}
for arbitrary $\lambda \in V^*$.
We give a partial answer to the question: given a vector
$\lambda$, for which $M \in \cL'$ is the minimum
value of $\Delta_M ( \lambda)$ achieved? It is clear that we can restrict to 
$M \in \cL_{\max}$. 
For the classical root systems $A_n$, $B_n$, $C_n$ and $D_n$ the answer is the following. 

\begin{proposition} \label{PropositionMinimum}
Let $\Phi$ be a classical root system. Then for every 
$\lambda \in V^*$ and $M \in \cL'$, there exists 
$M' \in \cL_{\max}$ with $\abs{\Phi^+_{\neg M'}}$ minimal (or equivalently, $\abs{\Phi^{M',+}}$ maximal) among the elements of $\cL_{\max}$ 
and 
$\Delta_M ( \lambda) \ge \Delta_{M'} ( \lambda)$.
\end{proposition} 
 
We give a slightly more precise statement in Proposition 
\ref{PropositionMinimumPrecise} below and describe the $M'$ for which  
$\abs{\Phi^+_{\neg M'}}$ is minimal explicitly for each classical root system.
Unfortunately, we can only offer a proof that works through a number of cases individually. 
As a preliminary reduction step, we may obviously assume that $\lambda$ lies in the closure of the positive Weyl chamber, which implies that
\[
\Delta_M ( \lambda) = 
\prod_{\alpha \in \Phi^+_{\neg M}} (1 + \sprod{\alpha}{\lambda}).
\]
For subsets of $\Phi^+$ we define $S_1 \le S_2$ if there exists an injection $\iota: S_1 \to S_2$ with $\alpha \le \iota(\alpha)$ for all $\alpha \in S_1$, where as usual $\alpha \le \beta$ means that $\beta - \alpha$ is a linear combination of positive roots with non-negative coefficients.
The following fact is clear:
\[
\Phi^+_{\neg M'} \le \Phi^+_{\neg M} \quad
\text{implies that} \quad
\Delta_{M'} ( \lambda) \le \Delta_{M} ( \lambda).
\]
We now summarize the description of the sets $\Phi^+_{\neg M}$ for the classical root systems.
\begin{enumerate}
\item In the $A_{n-1}$ case, the elements $M$ of the set $\cL_{\max}$ are
parametrized by non-empty subsets $I \subset \{ 1, \ldots, n \}$ with non-empty complement $\bar{I} = \{ 1, \ldots, n \} \smallsetminus I$, keeping in mind that
$I$ and $\bar{I}$ define the same $M$.
The set of roots $\Phi^+_{\neg M}$ is given by
\[
\Phi^+_{\neg M} = \{ e_i - e_j \, | \, 
1 \le i < j \le n, \quad i \in I, j \in \bar{I} 
\text{ or } i \in  \bar{I}, j \in I \}.
\]
The cardinality of $\Phi^+_{\neg M}$ is
minimal precisely if $I$ or $\bar{I}$ is a singleton $\{ i \}$. We denote the corresponding element of $\cL_{\max}$ by $M_i$ and the set $\Phi^+_{\neg M_i}$ by $\Phi^+_i$. We have
\[
\Phi^+_i = \{ e_k - e_i \, | \, 
1 \le k < i \} \cup 
\{ e_i - e_j \, | \, 
i < j \le n \},
\]
a set with $n-1$ elements.
\item For types $B_n$, $C_n$ and $D_n$ we parametrize 
$\cL_{\max}$ by non-empty subsets $I \subset \{ 1, \ldots, n \}$ and
equivalence classes of maps $\epsilon: I \to \{ \pm 1 \}$, where $\epsilon$ and $-\epsilon$ are regarded as equivalent, and moreover $\abs{I} \neq n-1$ in the $D_n$ case. The associated set of roots $\Phi^+_{\neg M}$ is 
\begin{eqnarray*}
\Phi^+_{\neg M} & = & \{ e_i  \pm e_j \, | \, 
1 \le i < j \le n, \quad i \in I, j \in \bar{I} 
\text{ or } i \in  \bar{I}, j \in I \} \\
& & \cup 
\{ e_i + \epsilon_i  \epsilon_j e_j \, | \, 
1 \le i < j \le n, \quad i \in I, j \in I  \} 
\end{eqnarray*}
in the $D_n$ case. In the $B_n$ case, in addition the roots 
$e_i$, $i \in I$, and in the $C_n$ case the roots $2e_i$, $i \in I$, are also included in $\Phi^+_{\neg M}$.
We say that $M$ is of Siegel type if $I = \{ 1, \ldots, n \}$.
The minimum cardinality of $\Phi^+_{\neg M}$ is again achieved in the case where $I$ is a singleton
$\{ i \}$. For $\abs{I} > 1$ the cardinality of $\Phi^+_{\neg M}$ is strictly greater than the minimum value except in the following cases: for $B_2 \simeq C_2$ the cardinality is the same for all $M$, and for $D_4$ the minimum is also achieved in the Siegel case $I = \{ 1, \ldots, 4 \}$.
We denote the element of $\cL_{\max}$ corresponding to 
$I = \{ i \}$
by $M_i$ and the set $\Phi^+_{\neg M_i}$ by $\Phi^+_i$. We have
\[
\Phi^+_i = \{ e_k  \pm e_i \, | \, 
1 \le k < i \} \cup 
\{ e_i \pm e_j \, | \, 
i < j \le n \}
\]
in the $D_n$ case. The cardinality of this set is $2n-2$. In the $B_n$ case, we include in addition the root $e_i$, in the $C_n$ case the root $2e_i$. The minimum cardinality of $\Phi^+_{\neg M}$ is $2n-1$ in these cases.
\end{enumerate}

\begin{proposition} \label{PropositionMinimumPrecise}
Let $M \in \cL_{\max}$ for a classical root system $\Phi$.
\begin{enumerate}
\item Assume that $M$ is not of Siegel type in the $B_n$ and $D_n$ cases. Then there exists an index $i$ such that $\Phi^+_i \le \Phi^+_{\neg M}$. 
\item For all $M \in \cL_{\max}$, except when $\Phi$ is of type $B_2$ or when $\Phi$ is of type $D_4$ and $\abs{\Phi^+_{M}}$ is maximal, there exists an index $i$ with 
$\Delta_M ( \lambda) \ge \Delta_{M_i} ( \lambda)$.
\end{enumerate}
\end{proposition}

Proposition \ref{PropositionMinimumPrecise} obviously implies
Proposition \ref{PropositionMinimum}.
Proposition \ref{PropositionMinimumPrecise} will follow from a series of lemmas treating special cases.

\begin{lemma} \label{LemmaIndexTriple} 
Assume that 
there exist integers with $j_1 < i < j_2$ with $j_1$, $j_2 \in \bar{I}$ and $i \in I$, or $j_1$, $j_2 \in I$ and $i \in \bar{I}$.
Then $\Phi^+_i \le \Phi^+_{\neg M}$.
\end{lemma}

\begin{proof}
Consider the $A_{n-1}$ case. We have
$e_k - e_i \in \Phi^+_{\neg M}$ for $k<i$ and $k \in \bar{I}$ and $e_i - e_j
\in \Phi^+_{\neg M}$ for $i < j$, $j \in \bar{I}$. 
On the other hand, for $k<i$ and $k \in I$ we have 
$e_k - e_i \le e_k - e_{j_2} \in \Phi^+_{\neg M}$ and for
$j > i$ with $j \in I$ we have 
$e_i - e_j \le e_{j_1} - e_{j} \in \Phi^+_{\neg M}$. This shows the assertion
for $A_{n-1}$.

In the $B_n$, $C_n$ and $D_n$ cases note first that
$e_k \pm e_i \in \Phi^+_{\neg M}$ for $k<i$ and $k \in \bar{I}$ and $e_i \pm e_j
\in \Phi^+_{\neg M}$ for $i < j$, $j \in \bar{I}$. 
On the other hand, for $k<i$ and $k \in I$ we have 
$e_k - e_i \le e_k - e_{j_2} \in \Phi^+_{\neg M}$ and for
$j > i$ with $j \in I$ we have 
$e_i - e_j \le e_{j_1} - e_{j} \in \Phi^+_{\neg M}$. For all $k \neq i$,  $k \in I$, we have
$e_k + e_i \le e_k + e_{j_1} \in \Phi^+_{\neg M}$. In the $B_n$ or $C_n$ case, we also have $e_i \in \Phi^+_{\neg M}$ or $2e_i \in \Phi^+_{\neg M}$, respectively.

In the $A_{n-1}$ and $D_n$ cases, it is possible to reverse the roles of $I$ and $\bar{I}$ in this argument. In the $B_n$ and $C_n$ cases, we have to observe that $e_{j_1} = e_{i} + (e_{j_1} - e_i) \ge e_{i}$
and $2e_{j_1} = 2e_{i} + 2 (e_{j_1} - e_i) \ge 2 e_{i}$, respectively.
\end{proof}

\begin{lemma}
For $I = \{1, \ldots, i-1 \}$ 
or $I = \{i, \ldots, n \}$,
$2 \le i \le n$, 
we have
$\Phi^+_i \le \Phi^+_{\neg M}$.
\end{lemma}

\begin{proof}
In the $A_{n-1}$ case, the roots $e_k - e_i$ for $k<i$ are elements of $\Phi^+_{\neg M}$. For $j > i$ we have 
$e_i - e_j \le e_{i-1} - e_j \in \Phi^+_{\neg M}$.
 
In the other cases, the roots $e_k \pm e_i$ for $k<i$ are elements of $\Phi^+_{\neg M}$. For $j > i$ we have 
$e_i \pm e_j \le e_{i-1} \pm e_j \in \Phi^+_{\neg M}$. 
In the $B_n$ case, for $i \in I$, we have $e_i \in \Phi^+_{\neg M}$, and
otherwise one may observe that
$\Phi^+_{\neg M} \ni e_{i-1} = e_i + (e_{i-1} - e_i) \ge e_i$. The $C_n$ case is analogous.
\end{proof}

\begin{lemma}
For $I = \{1, \ldots, n \}$ in the $C_n$ case we have
$\Phi^+_n \le \Phi^+_{\neg M}$.
\end{lemma}

\begin{proof}
We have $2e_n \in \Phi^+_{\neg M}$.
Let $k < n$. Then $e_k + \epsilon_k  \epsilon_n e_n \in \Phi^+_{\neg M}$ and $\Phi^+_{\neg M} \ni 2e_k = (e_k + e_n) + (e_k - e_n) \ge e_k - \epsilon_k \epsilon_n e_n$.
\end{proof}

\begin{lemma}
For $I = \{1, \ldots, n \}$ in the $B_n$ case, $n \ge 3$, we have
$\Delta_M ( \lambda) \ge \Delta_{M_n} ( \lambda)$
for $\lambda$ in the closure of the positive Weyl chamber.
\end{lemma}

\begin{proof}
The roots in the set $\Phi^+_{\neg M}$ are $e_i$, $1 \le i  \le n$, and
$e_i + \epsilon_i \epsilon_j e_j$
for
$1 \le i < j \le n$. We can assume that $\epsilon_n = +1$.
Set $\mu_i = \lambda_i - \lambda_n \ge 0$ for $1 \le i  \le n-1$.
We have to show that
\[
\prod_{1 \le i  \le n-1} (1 + \mu_i + \lambda_n)
\prod_{1 \le i < j \le n} (1 + \lambda_i + \epsilon_i \epsilon_j \lambda_j)
\ge 
\prod_{1 \le i  \le n-1, \, \epsilon_i = +1} (1 + \mu_i)
\prod_{1 \le i  \le n-1, \, \epsilon_i =  -1} (1 + \mu_i + 2 \lambda_n).
\]
For this to hold it is enough that
\[
(1 + 2 \lambda_n)^{N_2}
\prod_{1 \le i  \le n-1} (1 + \mu_i + \lambda_n)
\ge 
\prod_{1 \le i  \le n-1, \, \epsilon_i = +1} (1 + \mu_i)
\prod_{1 \le i  \le n-1, \, \epsilon_i =  -1} (1 + \mu_i + 2 \lambda_n),
\]
where $N_2$ is the number of pairs $(i,j)$ with $1\le i <j  \le n-1$ and $\epsilon_i = \epsilon_j$. For this inequality to hold it is in turn sufficient that
\[
(1 + \lambda_n)^{N_1} (1 + 2 \lambda_n)^{N_2-N_1} \ge 1,
\]
where $N_1$ is the number of $1 \le i  \le n-1$ with $\epsilon_i =  -1$.
Since $N_2 \ge N_1 (N_1-1) / 2$, it is easy to see that this inequality holds for $N_1 \neq 1$ or $N_2 > 0$.
The only remaining case is when $N_1=1$ and $N_2=0$ which implies that $n = 3$ and $\epsilon=(-1,1,1)$ or $\epsilon=(1,-1,1)$.
Denote the simple roots by $\alpha_1 = e_1-e_2$, $\alpha_2 = e_2-e_3$, $\alpha_3 = e_3$.
In the first case
\[
\frac{\Delta_M ( \lambda)}{\Delta_{M_3} (\lambda)}
= \frac{(1+ \alpha_1) (1 + \alpha_1+\alpha_2+\alpha_3)(1+\alpha_2+\alpha_3)}{(1+\alpha_1+\alpha_2+2\alpha_3) (1+\alpha_2)} \ge 1,
\]
and in the second case
\[
\frac{\Delta_M ( \lambda)}{\Delta_{M_3} (\lambda)}
= \frac{(1+ \alpha_1) (1 + \alpha_1+\alpha_2+\alpha_3)(1+\alpha_2+\alpha_3)}{(1+\alpha_1+\alpha_2) (1+\alpha_2+2\alpha_3)} \ge 1,
\]
which finishes the proof.
\end{proof}

\begin{lemma} \label{LemmaDn} 
For $I = \{1, \ldots, n \}$ in the $D_n$ case, $n \ge 5$, we have
$\Delta_M ( \lambda) \ge \Delta_{M_{n-1}} ( \lambda)$
for $\lambda$ in the closure of the positive Weyl chamber.
\end{lemma}

\begin{proof}
The roots in $\Phi^+_{\neg M}$ are 
$e_i + \epsilon_i \epsilon_j e_j$
for
$1 \le i < j \le n$. We can assume that $\epsilon_n = +1$.
Set $x_{\eta} = \lambda_{n-1} + \eta \lambda_{n}$ for 
$\eta \in \{ \pm 1 \}$ and
$z_i = 1 + \lambda_i - \lambda_{n-1}$ for 
$1 \le i  \le n-2$.
Then
\[
\Delta_{M_{n-1}} ( \lambda)
= (1+x_1) (1+ x_{-1})
\prod_{1 \le i  \le n-2} z_i (z_i + x_1 + x_{-1}) .
\]
On the other hand,
\begin{eqnarray*}
\Delta_{M} ( \lambda) & \ge & (1+x_{\epsilon_{n-1}})
\prod_{1 \le i  \le n-2} (z_i + x_{\epsilon_i})
\prod_{1 \le i  \le n-2, \, \epsilon_i = \epsilon_{n-1}} (z_i + x_1 + x_{-1}) \\
& & \prod_{1 \le i  \le n-2, \, \epsilon_i \neq \epsilon_{n-1}} z_i 
\prod_{1 \le i < j \le n-2, 
\epsilon_i = \epsilon_{j} } (z_i + x_1 + x_{-1}).
\end{eqnarray*}
A crude lower bound for the quotient is 
\[
\Delta_{M} ( \lambda) \Delta_{M_{n-1}} ( \lambda)^{-1} \ge (1+x_{-\epsilon_{n-1}})^{N_1-1}
(1 + x_1 + x_{-1})^{N_2-N_1},
\]
where 
$N_1$ is the number of $1 \le i  \le n-2$ with
$\epsilon_i \neq \epsilon_{n-1}$ and
$N_2$ is the number of pairs $(i,j)$ with $1 \le i <j  \le n-2$ and $\epsilon_i = \epsilon_j$.

A sufficient condition for the assertion
$\Delta_{M} ( \lambda) \ge \Delta_{M_{n-1}} ( \lambda)$ is therefore that
$N_2 \ge \max (N_1, 1)$. Since $N_2 \ge N_1 (N_1-1)/2$, 
it is easy to see that this settles all cases except for
$N_1=2$ and $N_2=1$, which implies that $n=5$.

This case can be dealt with by explicit calculation. 
Denote the simple roots by $\alpha_i$, $1 \le i \le 5$.
There are six sign vectors $\epsilon$ to be considered, but we can
reduce to three cases by exchanging $\alpha_4$
and $\alpha_5$, if necessary.
For $\epsilon = (-1,-1,1,1,1)$ we have
\begin{eqnarray*}
\Delta_M ( \lambda) \Delta_{M_4} (\lambda)^{-1}
& = & (1 + \alpha_2) (1 + \alpha_1 + \alpha_2)
(1 + \alpha_1+2\alpha_2+2\alpha_3+\alpha_4+\alpha_5) \\
& & (1+ \alpha_1 + \alpha_2 + \alpha_3 + \alpha_4) 
(1 + \alpha_2 + \alpha_3 + \alpha_4) (1 + \alpha_3 + \alpha_5) \\
 & & (1 + \alpha_1+\alpha_2+\alpha_3+\alpha_4+\alpha_5)^{-1}  \\
& & (1+ \alpha_2 + \alpha_3 + \alpha_4 + \alpha_5)^{-1} 
(1 + \alpha_3)^{-1} (1 + \alpha_4)^{-1} \\
& \ge & 1,
\end{eqnarray*}
for $\epsilon = (-1,1,-1,1,1)$ we have
\begin{eqnarray*}
\Delta_M ( \lambda) \Delta_{M_4} (\lambda)^{-1}
& = & (1 + \alpha_1) (1 + \alpha_2)
(1 + \alpha_1+\alpha_2+2\alpha_3+\alpha_4+\alpha_5) \\
& & (1+ \alpha_1 + \alpha_2 + \alpha_3 + \alpha_4) 
(1 + \alpha_2 + \alpha_3 + \alpha_5) (1 + \alpha_3 + \alpha_4) \\
 & & (1 + \alpha_1+\alpha_2+\alpha_3+\alpha_4+\alpha_5)^{-1}  \\
& & (1+ \alpha_3 + \alpha_4 + \alpha_5)^{-1} 
(1 + \alpha_2 + \alpha_3)^{-1} (1 + \alpha_4)^{-1} \\
& \ge & 1,
\end{eqnarray*}
and for $\epsilon = (1,-1,-1,1,1)$ we have
\begin{eqnarray*}
\Delta_M ( \lambda) \Delta_{M_4} (\lambda)^{-1}
& = & (1 + \alpha_1) (1 +  \alpha_1 + \alpha_2)
(1 + \alpha_2+2\alpha_3+\alpha_4+\alpha_5) \\
& & (1+ \alpha_1 + \alpha_2 + \alpha_3 + \alpha_5) 
(1 + \alpha_2 + \alpha_3 + \alpha_4) (1 + \alpha_3 + \alpha_4) \\
 & & (1 + \alpha_2+\alpha_3+\alpha_4+\alpha_5)^{-1}  \\
& & (1+ \alpha_3 + \alpha_4 + \alpha_5)^{-1} 
(1 + \alpha_1 + \alpha_2 + \alpha_3)^{-1} (1 + \alpha_4)^{-1} \\
& \ge & 1.
\end{eqnarray*}
This finishes the proof of the lemma in all cases.
\end{proof}

It is now clear that Proposition \ref{PropositionMinimumPrecise} and Proposition \ref{PropositionMinimum} follow from Lemma \ref{LemmaIndexTriple} to \ref{LemmaDn}.
\end{appendix}

\bibliographystyle{amsalpha}
\bibliography{bib}

\end{document}